%% file: LargeScaleRegularity_arxiv_v1.tex
\documentclass[12pt,a4paper]{article}
\usepackage{a4,a4wide}
\usepackage{amsfonts,amsmath,amssymb,amsthm,enumitem}

\usepackage{mathtools}

\usepackage{csquotes}

\usepackage{color}
\usepackage{ifpdf}
\ifpdf
    \usepackage[pdftex]{graphicx}
    \DeclareGraphicsExtensions{.png,.pdf,.jpg}
\else
   \usepackage[dvips]{graphicx}
   \DeclareGraphicsExtensions{.eps}
\fi
\graphicspath{{.}{figures/}}
\usepackage[latin1]{inputenc}
\numberwithin{equation}{section}

\usepackage{doi}

\newtheorem{theorem}{Theorem}[section]
\newtheorem{lemma}[theorem]{Lemma}
\newtheorem{proposition}[theorem]{Proposition}

\newtheorem{remark}[theorem]{Remark}
\newtheorem*{theorem*}{Theorem}
\newtheorem*{lemma*}{Lemma}
\newtheorem*{proposition*}{Proposition}
\newtheorem*{corollary*}{Corollary}

\renewcommand\tilde{\widetilde}

\def\R{\mathbb{R}}

\def\T{\mathbb{T}}

\def\Z{\mathbb{Z}}
\def\N{\mathbb{N}}
\def\EE{\mathbb{E}}
\def\PP{\mathbb{P}}

\def\H{\mathcal{H}}

\def\F{\mathcal{F}}

\def\LM#1{\hbox{\vrule width.2pt \vbox to#1pt{\vfill \hrule width#1pt
height.2pt}}}
\def\LL{{\mathchoice {\>\LM7\>}{\>\LM7\>}{\,\LM5\,}{\,\LM{3.35}\,}}}
\def\restr{{\LL}}
\renewcommand{\phi}{\varphi}

\def\Div{\nabla \cdot}

\def\dist{\textup{dist}}

\def\1{\mathbf{1}}

\def\XXint#1#2#3{{\setbox0=\hbox{$#1{#2#3}{\int}$ }
\vcenter{\hbox{$#2#3$ }}\kern-.57\wd0}}

\def\eps{\varepsilon}
\def\lt{\left}
\def\rt{\right}
\def\les{\lesssim}
\def\ges{\gtrsim}

\DeclareMathOperator*{\argmin}{\arg\!\min}

\def\Bf{\overline{f}}
\def\Brho{\overline{\rho}}
\def\Brhop{\Brho_+}
\def\Brhom{\Brho_-}
\def\lay{\textrm{lay}}
\def\bdr{\textrm{bdr}}
\def\bulk{\textrm{bulk}}

\def\spt{\textup{Spt}\,}
\def\nabbdr{\nabla^{\bdr}}
\def\per{\textrm{per}}
\def\Wper{W_{2,\textrm{per}}^2}
\def\xL{x_L}
\def\yL{y_L}

\newcommand{\bra}[1]{\left( #1 \right)}
\newcommand{\sqa}[1]{\left[ #1 \right]}
\newcommand{\cur}[1]{\left\{ #1 \right\}}

\def\K{\mathcal{K}}

\begin{document}
\title{A large-scale regularity theory for the Monge-Amp\`ere equation with rough data and application to the optimal matching problem}

\author{Michael Goldman\thanks{ Universit\'e Paris-Diderot, Sorbonne Paris-Cit\'e, Sorbonne Universit\'e,  CNRS,  Laboratoire Jacques-Louis Lions, LJLL, F-75013 Paris, France, \texttt{goldman@math.univ-paris-diderot.fr}} \and Martin Huesmann\thanks{Martin Huesmann, Institut f\"ur Mathematik, Rheinische Friedrich-Wilhelms-Universit\"at Bonn, 53115 Bonn, Germany, \texttt{huesmann@iam.uni-bonn.de}} \and Felix Otto\thanks{Max Planck Institute for Mathematics in the Sciences, 04103 Leipzig, Germany, \texttt{Felix.Otto@mis.mpg.de}}}

\date{\today}

\maketitle

\abstract{The aim of this paper is to obtain  quantitative bounds for solutions to the optimal matching problem in dimension two. 
These bounds show that up to a logarithmically divergent shift,  the optimal transport maps are close to be the identity at every scale. These bounds allow us to pass to the limit as the system size goes to infinity and construct a locally optimal
coupling between the Lebesgue measure and the Poisson point process which retains the stationarity properties of the Poisson
point process only at the level of second-order differences. Our quantitative bounds are obtained through a Campanato iteration scheme 
based on a deterministic and a stochastic ingredient. The deterministic part, which can be seen as our main contribution, is a regularity result for Monge-Amp\`ere equations with rough right-hand side. 
Since we believe that it could be useful in other contexts, we prove it for general space dimensions.  The stochastic part is a concentration result for the optimal matching problem which builds on previous work by Ambrosio, Stra and Trevisan.  
}

\tableofcontents

\section{Introduction}
\setlength\parindent{0pt}

We are interested in  the
 optimal matching problem between the Lebesgue measure and the Poisson point process $\mu$ on the torus $Q_L:=\lt[-\frac{L}{2},\frac{L}{2}\rt)^d$ (i.e.\ $Q_L$ with the periodized induced metric $|\cdot|_{\mathrm{per}}$ from $\R^d$):
\begin{align}\label{eq:Poisson matching}
\Wper\bra{\frac{\mu(Q_L)}{|Q_L|}\chi_{Q_L}, \mu},
\end{align}
where $\Wper$ denotes the squared $L^2-$Wasserstein distance on $Q_L$ with respect to $|\cdot|_{\mathrm{per}}$.

This problem and some of its variants such as generalizations to $L^p-$costs or more general reference measures, have been the subject of intensive work in the past thirty years (see for instance \cite{Ta14,BaBo13,AmStTr16,Le17,Le18}). 
As far as we know, essentially all the previous papers investigating \eqref{eq:Poisson matching} were focusing on estimating the mean of \eqref{eq:Poisson matching}, e.g.\ \cite{AKT84, DoYu95, BaBo13, Ta94,AmStTr16}, 
or deviation  from the mean, e.g.\ \cite{DeScSc13, FoGu15}, by constructing on the one hand sophisticated couplings whose costs are asymptotically optimal  and proving on the other hand ansatz free lower bounds. The only exception is \cite{HuSt13} where for $d\ge 3$, stationary couplings between the Lebesgue 
measure and a  Poison point process on $\R^d$ minimizing the cost per unit volume are constructed.
From these works, in fact since \cite{AKT84}, it is understood that $d=2$ is the critical dimension for \eqref{eq:Poisson matching}. Indeed, while for $d\ge 3$, $\EE_L\sqa{\frac{1}{L^d}\Wper\bra{\frac{\mu(Q_L)}{|Q_L|}\chi_{Q_L}, \mu}}$
is of order  $1$, it is logarithmically diverging for $d=2$ (see Section \ref{sec:intro stoch}).

We focus here on the critical dimension $d=2$  and aim at a better description of the optimal transport maps, that is, of minimizers of  \eqref{eq:Poisson matching}. Building on a large-scale regularity theory for convex maps solving the Monge-Amp\`ere equation
\begin{equation}\label{eq:MAintro}
 \nabla \psi\# dx=\mu,
\end{equation}
which we develop along the way, we prove that  from the macroscopic scale $L$ down to the microscopic scale the solution of \eqref{eq:Poisson matching} is close to the identity plus a shift. Here
 closeness is measured with respect to  a scale-invariant $L^2$ norm.
Our main result is the following:
\begin{theorem}\label{theo:mainresult intro}
Assume that $d=2$. There exists $c>0$ such that for each dyadic $L\ge 1$ there exists a  random variable $r_{*,L}=r_{*,L}(\mu)\ge 1$ satisfying the exponential bound $\sup_L\EE_L\sqa{\exp\bra{\frac{c r_{*,L}^2}{\log 2 r_{*,L}}}  } <\infty$ and such that if\footnote{We use the short-hand notation $A\ll 1$ to indicate that there exists $\eps>0$ depending only on the dimension such that $A\le \eps$. Similarly, $A\les B$ means that there exists a dimensional
constant $C>0$ such that $A\le C B$.} $r_{*,L}\ll L$ there exists $\xL=\xL(\mu)\in Q_L$ with\footnote{Here and in the rest of the paper $\log$ denotes the natural logarithm}
\begin{equation}\label{eq:shift intro} |\xL|^2 \les r_{*,L}^2 \log^3\bra{\frac{L}{r_{*,L}}} \end{equation}
such that  if $T=T_{\mu,L}$ is the minimizer of \eqref{eq:Poisson matching}, then for every $2r_{*,L}\leq \ell\leq L$, there holds
\begin{equation}\label{eq:main estimate intro}
  \frac{1}{\ell^4} \int_{B_\ell(\xL)} |T-(x-\xL)|^2\les \frac{\log^3\lt(\frac{\ell}{r_{*,L}}\rt)}{\lt(\frac{\ell}{r_{*,L}}\rt)^2}.
\end{equation}
\end{theorem}
 Note that the bounds in \eqref{eq:shift intro} and \eqref{eq:main estimate intro} are probably not optimal. Indeed, in both estimates one would rather expect a linear dependence on  the logarithms. 
 Similarly, using a similar proof in dimension $d\geq 3$ we would get bounds of the order of $\log^2(L)$ for the shift $x_L$ even though it is expected to be of order one.
 In order to improve our bounds, one would need to better capture some cancellation effects. Notice however that the proof of \eqref{eq:main estimate intro} lead to the optimal estimate (cf.\ Remark \ref{rem:optimal rate})
  \[
  \inf_{\xi\in \R^2} \ \frac{1}{\ell^4} \int_{B_\ell(x_L)} |T-(x-\xi)|^2\les \frac{\log \lt(\frac{\ell}{r_{*,L}}\rt)}{\lt(\frac{\ell}{r_{*,L}}\rt)^2} \qquad \qquad \forall\  2 r_{*,L}\le \ell\le L.
 \]

Let us point out that Theorem \ref{theo:mainresult intro} would also hold for the Euclidean  transport problem on $Q_L$. The motivation for considering instead the transport problem on the torus comes from the good stationarity properties of the optimal transport maps in this setting.
 Indeed, since the bound \eqref{eq:main estimate intro} is uniform in $L$ one can try   to construct  a covariant and locally optimal coupling between
 the Lebesgue measure and the Poisson point process on $\R^2$  by taking the limit $L\to \infty$  in $T_{\mu,L}$. A similar strategy has been implemented in \cite{HuSt13} for $d\ge 3$ using ergodic-type arguments\footnote{Notice that actually in \cite{HuSt13} a slightly relaxed version of \eqref{eq:Poisson matching} was considered and it is not known (although conjectured) that solutions of \eqref{eq:Poisson matching} converge to the unique stationary coupling with minimal cost per unit volume.}.
  One of the key ingredients used in that paper, namely the fact that the minimal cost per unit volume is finite, 
is missing for $d=2$. 
 The presence of the logarithmically divergent shift $\xL$  in \eqref{eq:main estimate intro} can be seen as a manifestation of the logarithmic divergence of the minimal cost per unit volume. In order to pass to the limit we thus need to renormalize the transport map by subtracting this shift.
 Because of this  renormalization the limiting map will loose its stationarity properties which will roughly speaking only survive at the level of the gradient. Also since we do not have any uniqueness statement for the limit objects, we will have to pass to the limit in the sense of Young measures.\\
 
 In order to state our second main result, we need  more notation. It is easier to pass to the limit at the level of the Kantorovich potentials rather than for the corresponding transport maps. Also, since the Lebesgue measure on $\R^2$ is invariant under arbitrary shifts while the Poisson point process on $Q_L$ (extended by periodicity to $\R^2$) is not, 
 it is more natural to make the shift in the domain and keep the image unchanged. This could also serve as motivation for centering the estimate \eqref{eq:main estimate intro} around $x_L$ (which is approximately  equal to $T^{-1}(0)$) rather than around $0$.\\ 
To be more precise, denote by $\widehat \psi=\widehat \psi_{\mu,L}$ the Kantorovich potential defined on $\R^2$ and associated to $T_{\mu,L}$, i.e. $T_{\mu,L}=\nabla \widehat{\psi}_{\mu,L}$, satisfying $\widehat \psi_{\mu,L}(0)=0$ (see Section \ref{sec:optimaltransp}). Define,
$$\psi_{\mu,L}(x):= \widehat\psi_{\mu,L}(x+\xL)$$
with corresponding Legendre dual $\psi^*_{\mu,L}(y)=\widehat \psi^*_{\mu,L}(y) - \xL\cdot y$. With this renormalization we still  have 
\[
\nabla \psi_{\mu,L} \# \frac{\mu(Q_L)}{|Q_L|} =\mu.
\]
Denote the space of all real-valued convex functions $\psi:\R^2\to\R$ by $\mathcal K$ and the space of all locally finite point configurations by $\Gamma$. We equip $\mathcal{K}$ with the topology of locally uniform convergence and $\Gamma$ with the 
topology obtained by testing against continuous and compactly supported functions. There is a natural action on $\Gamma$ by $\R^2$ denoted by $\theta_z$ and given by $\theta_z\mu:=\mu(\cdot +z)$ for $z\in\R^2.$
 Define the map $\Psi_L:\Gamma\to\mathcal K$ by $\mu\mapsto \psi_{\mu,L}$ and denote by $\PP_L$ the Poisson point process on $Q_L$. For each dyadic $L\ge 1$ we define the Young measure associated to $\Psi_L$ by
$$q_L:=(id,\Psi_L)\#\PP_L=\PP_L\otimes \delta_{\Psi_L}.$$
Then, we have the following result:
\begin{theorem}\label{theo:limit intro}
 The sequence $(q_L)_{L}$ of probability measures is tight. Moreover, any accumulation point $q$ satisfies the following properties:
\begin{enumerate}
 \item[(i)] The first marginal of $q$ is the Poisson point process $\PP_L$;
\item[(ii)] $q$ almost surely $\nabla \psi\#dx=\mu$;
\item[(iii)] for any $h,z\in\R^2$ and $f\in C_b(\Gamma\times C^0(\R^2))$ there  holds
$$ \int_{\Gamma\times \K} f(\mu,D^2_h\psi^*) dq = \int_{\Gamma\times \K} f(\theta_{-z}\mu,D^2_h \psi^*(\cdot -z)) dq,  $$
where  $D^2_h \psi^*(y):= \psi^*(y+h)+\psi^*(y-h)-2\psi^*(y)$.
\end{enumerate}
\end{theorem}
 Part (iii) of Theorem \ref{theo:limit intro} says that for any $h$, under the measure $q$ the random variable $(\mu,\psi)\mapsto (\mu,D^2_h\psi^*)$ is stationary with respect to the action induced by the natural shifts on $\Gamma$ and $\mathcal K$.
Observe that as already pointed out, while the second-order increments of the potentials $\psi$ are stationary,  the induced couplings $(id,\nabla\psi)\#dx$ are not. This is due to the necessary renormalization by $\xL$.
It is an interesting open problem to understand whether one can prove that the sequence $(\Psi_L)_L$ actually converges and get rid of the Young measures. A slightly weaker open problem is to show non/uniqueness of the accumulation points of $(q_L)_L.$

\subsection{Main ideas for the proof of Theorem \ref{theo:mainresult intro}}
The proof of Theorem \ref{theo:mainresult intro} is  inspired by
(quantitative) stochastic homogenization, in the sense that it is based on a Campanato iteration scheme which allows to transfer the information that  \eqref{eq:main estimate intro} holds at the ``thermodynamic'' scale  (here, the scale $L\uparrow\infty$ of the torus) by \cite{AKT84,AmStTr16} 
to scales of order one (here, the scale $r_{*,L}$).  This is reminiscent of the approach of Avellaneda and Lin \cite{MR0910954} to a regularity theory for (linear) elliptic equations
with periodic coefficients: The good regularity theory of the homogenized operator, i.e.~the regularity theory
on the thermodynamic scale, is passed down to the scale of the periodicity. This approach has been adapted
by Armstrong and Smart \cite{MR3481355} to the case of random coefficients; the approach has been further refined by Gloria, Neukamm and
the last author \cite{arXiv:1409.2678} (see also \cite{AKMbook})
where the random analogue of the scale of periodicity, and an analogue to $r_{*,L}$ in this paper,
has been introduced and optimally estimated (incidentally also by concentration-of-measure arguments
as in this paper).\\

The  Campanato  scheme is obtained by a combination of a deterministic and a stochastic argument.
The deterministic one is similar in spirit to \cite{GO}. It asserts that if at some scale $R>0$ the excess energy is small and if the Wasserstein distance  of $\mu\LL B_R$ to $\frac{\mu(B_R)}{|B_R|}\chi_{B_R}$ is also small then up to an affine change of variables  the excess energy is well controlled by these two quantities at scale $\theta R$ for some $\theta\ll1$. 
 The aim of the stochastic part is to prove that  with overwhelming probability, the Euclidean $L^2-$Wasserstein distance $\frac{1}{R^{4}}W_2^2\lt(\mu\LL B_R,\frac{\mu(B_R)}{|B_R|}\rt)$ is small for every (dyadic) scale $R$ between $L$ and $1$ so that the Campanato scheme can indeed be iterated down to the microscopic scale. 
Since our proof of the stochastic estimate is based on the results of \cite{AmStTr16} which are stated for cubes, we will actually prove the stochastic estimate on cubes instead of balls.   
We now describe these two parts separately in some more detail. We start with the stochastic aspect since it is simpler.
 
\subsubsection{The stochastic part}\label{sec:intro stoch}

 For a measure $\mu$ on $Q_L$ and $\ell \le L$ we denote its restriction to $Q_\ell\subset Q_L$ by $\mu_\ell:=\mu\LL Q_{\ell}$. Then, the main stochastic ingredient for the proof of Theorem \ref{theo:mainresult intro} is the following result:
\begin{theorem}\label{theo:stoch intro}
For dyadic $L\ge 1$ and $\mu$ a Poisson point process on $Q_L$ there exists a universal constant $c>0$ and a family of  random variables $r_{*,L}\ge 1$
 satisfying $\sup_L \EE_L\sqa{\exp\lt(\frac{c r_{*,L}^2}{\log (2r_{*,L})}\rt)}<\infty$ and  such that for every dyadic $\ell$ with $2 r_{*,L}\le \ell\le L$,
 \begin{equation*}
  \frac{1}{\ell^4}W_2^2\lt(\mu_\ell, \frac{\mu(Q_\ell)}{\ell^2}\rt)\le \frac{ \log\bra{\frac{\ell}{r_{*,L}}}}{\lt(\frac{\ell}{r_{*,L}}\rt)^2}.
 \end{equation*}
\end{theorem}

The proof of this result relies on an adaptation of the concentration argument put forward in \cite[Remark 4.7]{AmStTr16}. One of the   differences between this article and \cite{AmStTr16}
is that in \cite{AmStTr16} the more classical version of \eqref{eq:Poisson matching}, namely the matching of the empirical measure of $n$ iid uniformly distributed points $X_1,\ldots,X_n$ on the cube $Q=[-\frac{1}{2},\frac{1}{2})^d$ to their reference measure is considered:
\begin{align}\label{eq:matching problem}
C_{n,d}:=\EE \sqa{ W_2^2\bra{\chi_Q, \frac1n \sum_{i=1}^n \delta_{X_i}} } = \frac{1}{n^{\frac2d}} \EE \sqa{ W_2^2\bra{\frac1n \chi_{Q_{n^{\frac1d}}}, \frac1n \sum_{i=1}^n \delta_{n^{\frac1d}X_i}} } .
\end{align}
Since the typical distance between nearby points $X_i$ and $X_j$ is of order $n^{-\frac1d}$ it is expected that $C_{n,d}\sim n^{-\frac2d}$. However, it turns out that this is only true in $d\geq 3.$ Since the seminal work \cite{AKT84}, it is known that in dimension two an extra logarithmic factor appears. In dimension one the correct scaling is of order $\frac1n$ so that we can summarize 
\begin{equation*}
 C_{n,d} \sim \begin{cases}
               \frac1n, &d=1, \;\; \text{cf.\ }  \cite{BoLe16}\\
\frac{\log n}{n}, &d=2, \;\; \text{cf.\ }\cite{AKT84}\\
\frac{1}{n^{\frac2d}}, &d\geq 3,\;\; \text{cf.\ }\cite{BaBo13, Le17}.
              \end{cases}
\end{equation*}
Based on a linearization ansatz of the Monge-Amp\`ere equation suggested by \cite{CaLuPaSi14} in the physics literature, \cite{AmStTr16} significantly strengthened the two-dimensional case to
$$ \lim_{n\to\infty} \frac{n}{\log n} C_{n,2}=\frac{1}{4\pi}.$$
Additionally, it is  remarked  in \cite[Remark 4.7]{AmStTr16} that the mass concentrates around the mean.
Combining this concentration argument with  conditioning on the number of points of $\mu$ in $Q_L$ and a Borel-Cantelli argument we show Theorem \ref{theo:stoch intro}.

\subsubsection{The deterministic part}

As already alluded to, the deterministic ingredient is  one step of a Campanato scheme for solutions of the Monge-Amp\`ere equation with arbitrary right-hand side \eqref{eq:MAintro}. 
Since we believe that this far-reaching generalization of  \cite[Proposition 4.7]{GO} (see also \cite{GO18}) could have a large range of applications we prove it for arbitrary dimension $d\ge 2$. 
Let us point out that while \cite{GO} gives an alternative proof of the  partial regularity result for the Monge-Amp\`ere equation with ''regular'' data  previously obtained in \cite{FigKim,DePFig} (see \cite{DePFigIndam} for a nice informal presentation of this approach),
it is unclear  if  that other approach based on maximum principles could also be used in our context.

Given   for some $R>0$ a bounded set $\Omega\supset B_{2R}$  and an arbitrary measure $\mu$, denote by $T$ the optimal transport map between $\chi_\Omega$ and $\mu$  and let  $O\supset B_{2R}$ be an open set.
We have the following result: 
\begin{theorem}\label{theo:det intro}
 For every $0<\tau\ll1$, there exist positive constants  $\eps(\tau)$, $C(\tau)$,  and $0<\theta<1$ such that if  
 \begin{equation}\label{hypsmall}
  \frac{1}{R^{d+2}}\int_{B_{2R}}|T-x|^2 + \frac{1}{R^{d+2}} W_2^2\lt(\mu\LL O, \frac{\mu(O)}{|O|}\rt)\le \eps(\tau), 
 \end{equation}
  then there exists a symmetric matrix $B$ and a vector $b\in \R^d$ such that 
 \begin{equation*}
  |B-Id|^2+\frac{1}{R^2}|b|^2\les  \frac{1}{R^{d+2}}\int_{B_{2R}}|T-x|^2,
 \end{equation*}
and letting $\hat x:=B^{-1} x$, $\hat \Omega:= B^{-1} \Omega$ and then 
\begin{equation*}
 \hat{T}(\hat x):=B(T(x)-b) \qquad \textrm{and} \qquad \hat \mu:=\hat{T}\# \chi_{\hat \Omega} d \hat{x},
\end{equation*}
so that $\hat{T}$ is the optimal transport map between $\chi_{\hat\Omega}$ and $\hat \mu$, we have  
\begin{equation}\label{eq:mainestimateintro}
  \frac{1}{(\theta R)^{d+2}}\int_{B_{2\theta R}}|\hat{T}-\hat{x}|^2\le   \frac{\tau}{R^{d+2}}\int_{B_{2R}}|T-x|^2+ \frac{C(\tau)}{R^{d+2}} W_2^2\lt(\mu\LL O, \frac{\mu(O)}{|O|}\rt).
\end{equation}
\end{theorem}
The only reason for not taking $O=B_{2R}$ is that in our application, the control on the data term $\frac{1}{R^{d+2}} W_2^2\lt(\mu\LL O, \frac{\mu(O)}{|O|}\rt)$ will be given by Theorem \ref{theo:stoch intro} which is stated for cubes.
Let us stress that since  $\int_{B_{2R}}|T-x|^2$ behaves like a squared $H^1$ norm in terms of the potentials and since the squared Wasserstein distance behaves like a squared $H^{-1}$ norm (cf.\ \cite[Theorem 7.26]{Viltop}), 
 all quantities occur in the estimate \eqref{eq:mainestimateintro} as if we were dealing with a second order linear elliptic equation and looking at squared $L^2$-based quantities.

Since the estimates in Theorem \ref{theo:det intro} are scale-invariant, it is enough to prove it for $R=1$. We then let for notational simplicity
\[
 E:= \int_{B_2} |T-x|^2 \qquad \textrm{and} \qquad  D:= W_2^2\lt(\mu\LL O, \frac{\mu(O)}{|O|}\rt).
\]

The main ingredient for the proof of Theorem \ref{theo:det intro} is the following result, which is the counterpart of \cite[Proposition 4.6]{GO} and which states that if $E+D\ll1$, that is, if the  energy is small and if the data is close to a constant in the natural $W^2_2-$topology, then $T-x$ is quantitatively close to the gradient $\nabla\phi$ of a solution to a Poisson equation. This quantifies the well-known fact that
the Monge-Amp\`ere equation linearizes to the Poisson equation around the constant density. 

\begin{proposition}\label{prop:detintro}
 For every $0<\tau\ll1$, there exist positive constants $\eps(\tau)$ and  $C(\tau)$ such that if $E+D\le \eps(\tau)$, then there exists a function $\phi$ with harmonic gradient in $B_{\frac{1}{4}}$ and such that 
 \begin{equation*}
  \int_{B_{\frac{1}{4}}}|T-(x+\nabla \phi)|^2\les \tau E+ C(\tau) D
 \end{equation*}
and 
\begin{equation*}
  \int_{B_{\frac{1}{4}}}|\nabla \phi|^2\les E.
\end{equation*}
\end{proposition}
As in \cite{GO}, Proposition \ref{prop:detintro} is actually proven at the Eulerian (or Benamou-Brenier) level. That is, if we let for $t\in[0,1]$, $T_t:= (1-t)Id+tT$, $\rho:=T_t\# \chi_{\Omega}$, and $j:= T_t\# (T-x)\chi_{\Omega}$, the couple $(\rho,j)$ solves
\begin{equation}\label{eq:BBintro}
\min_{(\rho,j)} \lt\{ \int_{\Omega} \int_0^1 \frac{1}{\rho}|j|^2 \ : \ \partial_t \rho+\nabla \cdot j=0, \ \rho_0=\chi_{\Omega}, \ \rho_1=\mu \rt\} 
\end{equation}
and we show
\begin{proposition}\label{prop:intromaineulerian}
For every $0<\tau\ll1$, there exist positive constants $\eps(\tau)$ and  $C(\tau)$ such that if $E+D\le \eps(\tau)$, then there exists a function $\phi$ with harmonic gradient in $B_{1}$ such that  
 \begin{equation}\label{eq:introdistjphiprop3}
  \int_{B_{\frac{1}{2}}}\int_0^1\frac{1}{\rho}|j-\rho\nabla \phi|^2\les \tau E+ C(\tau) D
 \end{equation}
and
\begin{equation*}
  \int_{B_{1}}|\nabla \phi|^2\les E.
\end{equation*}
\end{proposition}
As in \cite{GO}, this is proven by first choosing a good radius where the flux of $j$ is well controlled in order to define $\phi$, then obtaining the almost-orthogonality estimate 
\begin{equation}\label{eq:introalmorth}
 \int_{B_{\frac{1}{2}}}\int_{0}^1 \frac{1}{\rho}|j-\rho \nabla \phi|^2 \les \lt(\int_{B_1}\int_0^1 \frac{1}{\rho} |j|^2-\int_{B_1} |\nabla \phi|^2\rt) +\tau E +C(\tau) D
\end{equation}
and finally constructing a competitor and using the minimality of $(\rho,j)$ for \eqref{eq:BBintro} in order to estimate the term inside the brackets in \eqref{eq:introalmorth}. However, each of these steps is considerably harder than in \cite{GO}.
This becomes quite clear considering that by analogy with \cite{GO}, letting for $R>0,$ $\Bf:=\int_0^1 j\cdot \nu$ where $\nu$ is the outward normal to $B_R$, one would like to define $\phi$ as a solution of 
\[
 \begin{cases}
   \Delta \phi= 1-\mu & \textrm{in } B_R\\
   \frac{\partial \phi}{\partial \nu} = \Bf &\textrm{on } \partial B_R
 \end{cases}
\]
for some well chosen $R\in (\frac{1}{2},\frac{3}{2})$.  However if $\mu$ is a singular measure $\nabla \phi$ will typically not be in $L^2$ since this  would require $L^2$  bounds (actually $H^{-\frac12}$ would be enough) on $\Bf$ which cannot be obtained from the energy through a Fubini type argument since  
the $L^\infty$ norm of $\rho_t$ typically blows up as $t\to 1$. Similar issues were tackled in \cite{AmStTr16, Le17}  by mollification of $\mu$ with smooth kernels (the heat and the Mehler kernel respectively). 
Here instead we introduce a small time-like parameter $\tau$ and work separately in $(0,1-\tau)$ and in the terminal layer $(1-\tau,1)$.\\
In $(0,1-\tau)$, we take care of the flux going through $\partial B_R$ in that time interval. We first modify the definition of $\Bf$ and let $\Bf:=\int_0^{1-\tau} j\cdot \nu$, and  then change $\phi$ so that it connects in $B_R$ the constant density equal to $1$  to the constant density equal to $1-\frac{1}{|B_R|} \int_{\partial B_R}  \Bf$, i.e. $\phi$ solves
\begin{equation}\label{defphi1}
 \Delta \phi = \frac{1}{|B_R|} \int_{\partial B_R}  \Bf \qquad \textrm{ in }B_R.
\end{equation}
Regarding the Neumann boundary conditions, we face here the problem that even though $\rho\in L^\infty(B_2\times(0,1-\tau))$, its $L^\infty$ bound blows up as $\tau \to 0$. This leads to $L^2$ bounds on $\Bf$ which are not uniform in $\tau$. To overcome this difficulty, we need to replace $\Bf$ by a better behaved density on $\partial B_R$. 
This is the role of Lemma \ref{goodR}. Treating separately the incoming and outgoing fluxes $\Bf_\pm$, we construct densities $\Brho_\pm$ on $\partial B_R$ with
\begin{equation}\label{eq:modifyprho}
 \int_{\partial B_R} \Brho_\pm^2\les E \qquad \textrm{and } \qquad W_2^2(\Brho_\pm,\Bf_\pm)\les E^{\frac{d+3}{d+2}}.
\end{equation}
The densities $\Brho_\pm$ can be seen as rearrangements of $\Bf_\pm$ through projections on $\partial B_R$. Considering the time-dependent version of the Lagrangian problem, since $E\ll 1$ the particles hitting $\partial B_R$ in $(0,1-\tau)$ must come from a small neighborhood of $\partial B_R$ at time $0$. 
A key point in deriving \eqref{eq:modifyprho} is that at time $0$ the density is well-behaved (since it is constant) and thus the number of particles coming from such a small neighborhood of $\partial B_R$ is under  control. The estimate in \eqref{eq:modifyprho} on the Wasserstein distance between
$\Brho_\pm$ and $\Bf_\pm$ is important in view of the construction of a competitor. Indeed, this indicates that if we know how to construct a
competitor having $\Brho_\pm$ as boundary fluxes, we will then be able to modify it into a competitor with the correct $\Bf_\pm$ boundary conditions (see in particular Lemma \ref{lem:defboundaryflux}).   We then complement \eqref{defphi1} with the boundary conditions
\[
 \frac{\partial \phi}{\partial \nu} =\Brho_+-\Brho_- \qquad \textrm{on } \partial B_R.
\]
The almost-orthogonality estimate \eqref{eq:introalmorth} is proven in Proposition \ref{prop:almostorth}. It is readily seen that assuming for simplicity that $R=1$ is the good radius,
\begin{multline}\label{eq:almostorthsketch}
 \int_{B_{\frac{1}{2}}}\int_{0}^1 \frac{1}{\rho}|j-\rho \nabla \phi|^2 \les \lt(\int_{B_1}\int_0^1 \frac{1}{\rho} |j|^2-\int_{B_1} |\nabla \phi|^2\rt)\\
 + \int_{B_1} \lt(\int_0^{1-\tau} \rho -1\rt) |\nabla \phi|^2 +\int_{B_1} \phi \rho_{1-\tau} +\int_{\partial B_1} \phi \lt[ \Bf -(\Brho_+-\Brho_-)\rt]+ \tau E.
\end{multline}
While in \cite{GO} the first term in the second line was easily estimated since in that case, up to a small error we had $\rho\le 1$,  we need here a more delicate argument. In order to estimate the second term, we use that, again up to choosing a good radius, we may assume that 
\begin{equation}\label{eq:estimrho1moinstau}
 W_2^2\lt(\rho_{1-\tau}\LL B_1, \frac{\rho_{1-\tau}(B_1)}{|B_1|} \chi_{B_1}\rt)\les \tau^2E+D,
\end{equation}
see Lemma \ref{lem:distdata}. Ignoring issues coming from the flux through $\partial B_1$, and thus assuming that $\frac{\rho_{1-\tau}(B_1)}{|B_1|} = \frac{\mu(B_1)}{|B_1|}=\frac{\mu(O)}{|O|}$ (where $O\supset B_2$ is the open set in the definition of  $D$), 
\eqref{eq:estimrho1moinstau} follows from $W_2^2(\rho_{1-\tau}\LL B_1, \mu\LL B_1)\les \tau^2 E$ by displacement interpolation, $W_2^2(\mu\LL B_1, \frac{\mu(B_1)}{|B_1|})\les D$ by definition and triangle inequality. 
The last term in \eqref{eq:almostorthsketch} is  estimated thanks to the $W_2^2$ estimate given in \eqref{eq:modifyprho}.\\
Let us finally describe the construction of the competitor given in Proposition \ref{prop:distjphiprop2}. As explained in the beginning of this discussion, we employ a different strategy for the time intervals $(0,1-\tau)$ and $(1-\tau,1)$. 
In $(0,1-\tau)$, forgetting the issue of connecting $\Brho_{\pm}$ to $\Bf_\pm$, we mostly take as competitor
\[
 (\tilde \rho,\tilde j)= \lt(1-\frac{t}{1-\tau}\frac{1}{|B_1|}\int_{\partial B_1} \Bf   ,\nabla \phi\rt)+ (s,q),
\]
where $(s,q)$ are supported in  the annulus $B_1\backslash B_{1-r} \times(0,1-\tau)$ for some boundary layer size $r\ll 1$ and satisfy the continuity equation,  $s_0=s_{1-\tau}=0$ and $q\cdot \nu =j\cdot \nu -\Bf$ on $\partial B_1\times (0,1-\tau)$. The existence of such a couple $(s,q)$ satisfying the appropriate energy estimates is given by \cite[Lemma 3.4]{GO} which in turn is inspired by a similar construction from \cite{ACO}. 
In the terminal time layer $(1-\tau,1)$, we connect the constant $1- \frac{1}{|B_1|}\int_{\partial B_1} \Bf$ to the measure $\mu$. The outgoing flux is easily treated by pre-placing on $\partial B_1$ the particles which should leave the domain in $(1-\tau,1)$. For the incoming flux, we use the corresponding part of $(\rho,j)$ as competitor. Finally, the remaining part of the measure
$\mu'\le \mu$ which was not coming from particles entering $\partial B_1$ in $(1-\tau,1)$ is connected to $1- \frac{1}{|B_1|}\int_{\partial B_1} \Bf$ thanks to the estimate
\[
 W_2^2\lt(\mu',\frac{\mu'(B_1)}{|B_1|}\rt)\les \tau^2E+D,
\]
which is obtained as \eqref{eq:estimrho1moinstau} in Lemma \ref{lem:distdata}.

\subsection{Outline}
The plan of the paper is the following. In Section \ref{sec:prelim} we first set up some notation and then prove a few more or less standard elliptic estimates. 
We then write in Section \ref{sec:optimaltransp} a quick reminder on optimal transportation.  In Section \ref{sec:Poi}, we give the definition and first properties of the Poisson point process and then prove our main concentration estimates (see Theorem \ref{theo:estimDrstar} and Theorem \ref{theo:estimDrstarper}).  
Section \ref{sec:main} is the central part of the paper and contains the proof of Theorem \ref{theo:det intro}. We first explain in Section \ref{sec:goodR} how to choose a good radius before proving Proposition \ref{prop:intromaineulerian}. This is a consequence of Proposition \ref{prop:almostorth} and Proposition \ref{prop:distjphiprop2} which contain the proof of the
almost-orthogonality property \eqref{eq:introalmorth} and the construction of a competitor  respectively. In  Section \ref{sec:application}, we combine the stochastic and deterministic ingredients to perform the Campanato iteration and prove both the quantitative bounds of Theorem \ref{theo:mainresult intro} and perform the construction
of the locally optimal coupling between Lebesgue and Poisson given in Theorem \ref{theo:limit intro}.

\subsection*{Acknowledgements}
 MH gratefully acknowledges partial support by the DFG through the CRC 1060 ``The Mathematics of Emerging Effects'' and by the Hausdorff Center for Mathematics. MG and MH thank the Max Planck Institute MIS for its warm hospitality.

\section{Preliminaries}\label{sec:prelim}
\subsection{Notation}

In this paper we will use the following notation. The symbols $\sim$, $\ges$, $\les$ indicate estimates that hold up to a global constant $C$,
which typically only depends on the dimension $d$.
For instance, $f\les g$ means that there exists such a constant with $f\le Cg$,
$f\sim g$ means $f\les g$ and $g\les f$. An assumption of the form $f\ll1$ means that there exists $\eps>0$, typically only
depending on dimension, such that if $f\le\eps$, 
then the conclusion holds.  

We write $\log$ for the natural logarithm. We denote by $\H^k$ the $k-$dimensional Hausdorff measure. 
For a set $E$,  $\nu_E$ will always denote the external normal to $E$. When clear from the context we will drop the explicit dependence on the set. 
We write $|E|$ for the  Lebesgue measure of a set $E$ and $\chi_E$ for the indicator function of $E$.
When no confusion is possible, we will drop the integration measures in the integrals. Similarly, we will often identify, if possible, measures with their densities with respect to the Lebesgue measure.
For $R>0$ and $x_0\in \R^d$, $B_R(x_0)$ denotes the ball of radius $R$ centered in $x_0$. 
When $x_0=0$, we will simply write $B_R$ for $B_R(0)$. We denote the gradient (resp.\ the Laplace-Beltrami operator) on $\partial B_R$ by $\nabbdr$ (resp.\ $\Delta^{\rm bdr}$).
For $L>0$, we denote by $Q_L:=\lt[-\frac{L}{2},\frac{L}{2}\rt)^d$ the cube of side length $L$.

\subsection{Elliptic estimates}
We start by collecting a few more or less standard elliptic estimates which we will need later on.
\begin{lemma}
 For $f\in L^2(\partial B_1)$  let $\phi$ be the (unique) solution of 
 \[\begin{cases}
   \Delta \phi=\frac{1}{|B_1|}\int_{\partial B_1} f &\textrm{in } B_1\\[8pt]
   \frac{\partial \phi}{\partial \nu}= f &\textrm{on } \partial B_1,
  \end{cases}\]
  with $\int_{\partial B_1} \phi=0$, then letting $p:=\frac{2d}{d-1}$,
\begin{equation}\label{CalZyg}
 \lt(\int_{B_1} |\nabla \phi|^p\rt)^{\frac1p}\les \lt(\int_{\partial B_1} f^2\rt)^{\frac12}.
\end{equation}
Moreover, for every $0<r\le 1$,
\begin{equation}\label{estimphiAr}
 \int_{B_1\backslash B_{1-r}} |\nabla \phi|^2\les r\int_{\partial B_1} f^2.
\end{equation}

\end{lemma}
\begin{proof}
Replacing $\phi$ by $\phi -\frac{|x|^2}{2d|B_1|}\int_{\partial B_1} f$, we may assume that $\int_{\partial B_1} f=0$. \\
 By Pohozaev's identity (see \cite{GO}) and Poincar\'e's inequality, we have that $\phi\in H^{1}(\partial B_1)$ with 
 \begin{align}\label{ao02}
 \int_{\partial B_1} |\nabla \phi|^2\les \int_{\partial B_1} f^2.
 \end{align}
Estimate \eqref{estimphiAr} follows from \eqref{ao02} together with the sub-harmonicity of $|\nabla \phi|^2$ in the form 
 \[
  \int_{\partial B_r} |\nabla \phi|^2\le \int_{\partial B_1} |\nabla \phi|^2 \qquad \textrm{for } r\le 1.
 \]

 We are just left to prove (\ref{CalZyg}). Since by (\ref{ao02}) we have
$\int_{\partial B_1}|\nabla\phi|^2\lesssim\int_{\partial B_1}f^2$, and since $\nabla \phi$ is harmonic, it suffices to show that for every harmonic function $\phi$ with $\int_{\partial B_1}\phi=0$ and thus $\int_{B_1}\phi=0$,
\begin{align}\label{ao03}
\bra{\int_{B_1}|\phi|^p}^\frac{1}{p}\lesssim\bra{\int_{\partial B_1}\phi^2}^\frac{1}{2}.
\end{align}
The argument for (\ref{ao03}) roughly goes as follows: By $L^2$-based regularity theory,
the $H^\frac{1}{2}(B_1)$-norm of $\phi$ is estimated by the $L^2(\partial B_1)$-norm of its Dirichlet data,
so that (\ref{ao03}) reduces to a  fractional Sobolev inequality on $B_1$. If one 
wants to avoid fractional Sobolev norms, in view of their various definitions on bounded domains,
one needs to construct an extension $\bar\phi$ of $\phi$ to the 
(semi-infinite) cylinder $B_1\times(0,\infty)$, preserving $\int_{B_1}\bar\phi=0$
and with the $H^1(B_1\times(0,\infty))$-norm of $\bar\phi$ estimated by the 
$H^\frac{1}{2}(B_1)$-norm of $\phi$, which combines to
\begin{align}\label{ao01}
\bra{\int_{B_1}\int_0^\infty|\nabla\bar\phi|^2+(\partial_t\bar\phi)^2}^\frac{1}{2}
\lesssim\bra{\int_{\partial B_1}\phi^2}^\frac{1}{2}.
\end{align}
It then remains to appeal to the (Sobolev-type) trace estimate 
\begin{align}\label{ao07}
\bra{\int_{B_1}|\phi|^p}^\frac{1}{p}\lesssim
\bra{\int_{B_1}\int_0^\infty|\nabla\bar\phi|^2+(\partial_t\bar\phi)^2}^\frac{1}{2}.
\end{align}

\medskip

For the sake of completeness, we now give the arguments for (\ref{ao01}) and (\ref{ao07}). 
Starting with (\ref{ao01}), we first argue that
it suffices to consider the case where $\phi_{|\partial B_1}$ is an eigenfunction of
the Laplace-Beltrami operator $-\Delta^{\rm bdr}$, say for eigenvalue $\lambda\ge 0$. Then we have
\begin{align}\label{ao05}
\phi(x)=r^\alpha\phi(\hat x)\quad\mbox{with}\;r:=|x|,\;\hat x:=\frac{x}{r},\;\alpha(d-2+\alpha)=\lambda,
\end{align}
which follows from $\Delta$ $=\frac{1}{r^{d-1}}\partial_rr^{d-1}\partial_r+\frac{1}{r^2}\Delta^{\rm bdr}$,
and define the extension
\begin{align*}
\bar\phi(x,t)=\exp(-\alpha t)\phi(x).
\end{align*}
With $\bar\phi'$ being another function of this form we have with  $\nabbdr \phi$ denoting the tangential part of the gradient of $\phi$
\begin{align*}
\lefteqn{(\nabla\bar\phi\cdot\nabla\bar\phi'+\partial_t\bar\phi\partial_t\bar\phi')(x,t)}\nonumber\\
&=\exp(-(\alpha+\alpha')t)r^{\alpha+\alpha'-2}
\big(\nabla^{\rm bdr}\phi\cdot\nabla^{\rm bdr}\phi'+\alpha\alpha'(1+r^2)\phi\phi'\big)(\hat x),
\end{align*}
and thus by integration by parts on $\partial B_1$ we have for every $r, t$
\begin{align*}
\lefteqn{\int_{\partial B_r\times\{t\}}\nabla\bar\phi\cdot\nabla\bar\phi'+\partial_t\bar\phi\partial_t\bar\phi'}\nonumber\\
&=\exp(-(\alpha+\alpha')t)r^{\alpha+\alpha'-2}(\lambda+\alpha\alpha'(1+r^2))\int_{\partial B_1}\phi\phi'.
\end{align*}
From this we learn that the $L^2(\partial B_1)$-orthogonality of the eigenspaces transmits to the extensions; hence
we may indeed restrict to an eigenfunction. Integrating in $(r,t)$ the last identity for $\phi'=\phi$ we obtain
\begin{align*}
\int_{B_1}\int_0^\infty|\nabla\bar\phi|^2+(\partial_t\bar\phi)^2
=\frac{4\alpha+\lambda+\frac{\lambda}{2\alpha}}{4\alpha^2-1}\int_{\partial B_1}\phi^2.
\end{align*}
Since (\ref{ao05}) implies that   $\lambda\les \alpha^2 +1$, we get (\ref{ao01}).

\medskip

We now turn to (\ref{ao07}), which we deduce from the  (mean-value zero) Sobolev estimate
\begin{equation}\label{sobolev}
\bra{\int_{B_1}\int_0^\infty|\bar\phi|^q}^\frac{1}{q}
\lesssim \bra{\int_{B_1}\int_0^\infty|\nabla\bar\phi|^2+(\partial_t\bar\phi)^2}^\frac{1}{2}\quad
\mbox{where}\;q=\frac{2(d+1)}{d-1},
\end{equation}
which holds because the analogue estimate holds on every $B_1\times(n-1,n)$, $n\in\mathbb{N}$, 
since $\int_{B_1\times(n-1,n)}\bar\phi=0$. From $|\frac{d}{dt}\int_{B_1}|\bar\phi|^p|$ $\le p\int_{B_1}|\bar\phi|^{p-1}|\partial_t\bar\phi|$ and Cauchy-Schwarz's inequality (using $2(p-1)=q$)
we have
\begin{align*}
\int_{0}^{\infty}\left|\frac{d}{dt}\int_{B_1}|\bar\phi|^p\right|
\le p\bra{\int_{B_1}\int_0^\infty|\bar\phi|^q}^\frac{1}{2}
\bra{\int_{B_1}\int_0^\infty(\partial_t\bar\phi)^2}^\frac{1}{2}.
\end{align*}
Therefore, using that $\int_{B_1}|\bar\phi|^p\to0$ as $t\to \infty$, we get
\begin{multline*}
 \lt(\int_{B_1}|\phi|^p\rt)^{\frac{1}{p}}\les \bra{\int_{B_1}\int_0^\infty|\bar\phi|^q}^\frac{1}{2p}
\bra{\int_{B_1}\int_0^\infty(\partial_t\bar\phi)^2}^\frac{1}{2p}\\
\stackrel{\eqref{sobolev}}{\les}  \bra{\int_{B_1}\int_0^\infty|\nabla\bar\phi|^2+(\partial_t\bar\phi)^2}^{\frac{q}{4p}+\frac{1}{2p}}= \bra{\int_{B_1}\int_0^\infty|\nabla\bar\phi|^2+(\partial_t\bar\phi)^2}^{\frac{1}{2}},
\end{multline*}
that is \eqref{ao07}.

\end{proof}

For the choice of a good radius (see Lemma \ref{lem:distdata} below) we will need the following not totally standard elliptic estimate.
\begin{lemma}\label{lem:elliptic}
 For every $c>0$ and  every $(z,f)$, with $0\le z \le c$ and $\spt z \subset \overline{B}_1\backslash B_{\frac{1}{2}}$, the   unique solution $\phi$  (up to additive constants) of 
 \[\begin{cases}
   \Delta \phi=z -\frac{1}{|B_1|}\lt(\int_{B_1} z-\int_{\partial B_1} f\rt) &\textrm{in } B_1\\[8pt]
    \frac{\partial \phi}{\partial \nu}= f &\textrm{on } \partial B_1
  \end{cases}\]
  satisfies
  \[
   \int_{B_1}|\nabla \phi|^2\les \int_{\partial B_1} f^2+ c \int_{B_1} (1-|x|) z.
  \]

\end{lemma}
\begin{proof}
 Without loss of generality, we may assume that $\int_{B_1} \phi=0$ and that $\int_{\partial B_1} f^2+ \int_{B_1} (1-|x|) z<\infty$, otherwise there is nothing to prove. Moreover, by scaling we may assume that $c=1$.
 Using integration by parts, the trace inequality for Sobolev functions together with the Poincar\'e inequality for functions of mean zero,
 \begin{align*}
  \int_{B_1}|\nabla \phi|^2&=\int_{\partial B_1} \phi f -\int_{B_1} \phi \Delta \phi\\
  &\le \lt(\int_{\partial B_1} \phi^2\rt)^{\frac12} \lt(\int_{\partial B_1} f^2\rt)^{\frac12} +\lt|\int_{B_1} z \phi\rt|\\
  &\les \lt(\int_{ B_1} |\nabla \phi|^2\rt)^{\frac12} \lt(\int_{\partial B_1} f^2\rt)^{\frac12} +\lt|\int_{B_1} z \phi\rt|.
 \end{align*}
Using Young's inequality, it is  now enough to prove that 
\begin{equation}\label{eq:toproveelliptic}
 \lt|\int_{B_1} z \phi\rt|\les\lt(\int_{ B_1} |\nabla \phi|^2\rt)^{\frac12} \lt(\int_{B_1} (1-|x|) z\rt)^{\frac12}. 
\end{equation}
For $\omega\in \partial B_1$, letting  
\[
 \overline{z}(\omega):=\int_0^{1} z(r \omega) r^{d-1} dr,
\]
 we claim that 
\begin{equation}\label{claim:overz}
 \int_{\partial B_1} \overline{z}^2\les \int_{B_1} (1-|x|)z.
\end{equation}
Indeed, momentarily fixing $\omega\in \partial B_1$ and setting $\psi(r):= r^{d-1} z(r\omega)$ for $r\in [0,1]$, we have $0\le \psi\le 1$ and 
\[
 \int_{0}^1 \psi = \overline{z}(\omega),
\]
so that for almost every $\omega\in \partial B_1$,
\[
 \int_{0}^1 (1-r) r^{d-1} z(r\omega)\ge \min_{\stackrel{0\le \tilde\psi\le 1}{\int \tilde\psi =\overline{z}(\omega)}} \int_{0}^1 (1-r) \tilde\psi(r) \ges \overline{z}^2(\omega),
\]
where the last inequality follows since the minimizer of 
\[
 \min_{\stackrel{0\le \tilde\psi\le 1}{\int \tilde\psi =\overline{z}(\omega)}} \int_{0}^1 (1-r) \tilde\psi(r)
\]
is given  by the characteristic function of $(1-\overline{z}(\omega),1)$.
Using that by hypothesis $\spt z \subset \overline{B}_1\backslash B_{\frac{1}{2}}$, we can thus write 
\begin{align*}
 \lt|\int_{B_1} z \phi\rt| &\le \lt|\int_{\partial B_1} \phi \overline{z}\rt|+\lt|\int_{\partial B_1} \int_{\frac{1}{2}}^1 (\phi(r \omega)-\phi(\omega)) z(r\omega) r^{d-1} dr d\omega\rt|\\
 &\les \lt(\int_{\partial B_1} \phi^2\rt)^{\frac12} \lt(\int_{\partial B_1} \overline{z}^2\rt)^{\frac12} +\int_{\partial B_1} \int_{\frac{1}{2}}^1 |\phi(r \omega)-\phi(\omega)| z(r\omega) r^{d-1} dr d\omega\\
 &\stackrel{\eqref{claim:overz}}{\les} \lt(\int_{ B_1} |\nabla \phi|^2\rt)^{\frac12} \lt(\int_{B_1} (1-|x|) z\rt)^{\frac12}\\
 & \qquad \qquad+\int_{\partial B_1} \int_{\frac{1}{2}}^1 |\phi(r \omega)-\phi(\omega)| z(r\omega) r^{d-1} dr d\omega,
\end{align*}
where in the last line we  used once more  that $\int_{\partial B_1} \phi^2\les \int_{B_1} |\nabla \phi|^2$. Since for $r\in(\frac{1}{2},1)$
\[
 |\phi(r\omega)-\phi(\omega)|\le \lt(\int_{\frac{1}{2}}^1 |\partial_r \phi(s \omega)|^2  ds \rt)^{\frac12}\les  \lt(\int_{0}^1 |\partial_r \phi(s \omega)|^2 s^{d-1} ds \rt)^{\frac12},
\]
estimate \eqref{eq:toproveelliptic} follows from 
\begin{align*}
& \int_{\partial B_1} \int_{\frac{1}{2}}^1 |\phi(r \omega)-\phi(\omega)| z(r\omega) r^{d-1} dr d\omega\\
&\les \int_{\partial B_1} \int_{\frac{1}{2}}^1 \lt(\int_{0}^1 |\partial_r \phi(s \omega)|^2 s^{d-1} ds\rt)^{\frac12}z(r\omega) r^{d-1} dr d\omega\\
 &=\int_{\partial B_1}  \lt(\int_{0}^1 |\partial_r \phi(s \omega)|^2 s^{d-1}ds \rt)^{\frac12} \overline{z}(\omega) d\omega\\
 &\le \lt(\int_{\partial B_1}\int_{0}^1 |\partial_r \phi(s \omega)|^2 s^{d-1}ds\rt)^{\frac12} \lt(\int_{\partial B_1}\overline{z}^2\rt)^{\frac12}\\
 &\stackrel{\mathclap{\eqref{claim:overz}}}{\les} \lt(\int_{B_1}|\nabla \phi|^2\rt)^{\frac12} \lt(\int_{B_1} (1-|x|) z\rt)^{\frac12}.
\end{align*}
\end{proof}

\subsection{The optimal transport problem}\label{sec:optimaltransp}
In order to set up notation, let us quickly recall some well known facts about optimal transportation. Much more can be found for instance in the books \cite{Viltop,VilOandN,AGS,Santam} to name just a few.
We will always work here with transportation between (multiples) of characteristic functions and arbitrary measures so that we restrict our presentation to this setting.

For a measure $\Pi$ on $\R^d\times\R^d$ we denote its marginals by $\Pi_1$ and $\Pi_2$, i.e.\ $\Pi_1(A)=\Pi(A\times\R^d), \Pi_2(A)=\Pi(\R^d\times A).$
For a given bounded set $\Omega$, a positive constant $\Lambda$ and a measure $\mu$
with compact support and such that $\mu(\R^d)=\Lambda |\Omega|$ any measure  $\Pi$ on $\R^d\times\R^d$ with marginals $\Pi_1=\Lambda\chi_\Omega$ and $\Pi_2=\mu$ is called a transport plan or coupling between $\Lambda\chi_\Omega$ and $\mu$. We define the Wasserstein distance between $\Lambda \chi_\Omega$ and $\mu$ as 
\begin{equation}\label{euclwass}
 W_2^2(\Lambda \chi_{\Omega}, \mu):=\min_{ \Pi_1=\Lambda \chi_\Omega, \, \Pi_2=\mu}\int_{\R^d\times \R^d} |x-y|^2 d\Pi=\min_{T\# \Lambda\chi_\Omega=\mu} \int_{\Omega}  |T-x|^2 \Lambda dx.
\end{equation}
By Brenier's Theorem \cite[Theorem 2.12]{Viltop}, the minimizer of the right-hand side of \eqref{euclwass} exists, is called optimal transport map, and is uniquely defined a.e.\ on $\Omega$ as the gradient of a convex map $\psi$. Conversely, for every convex map $\psi$, every $\Lambda>0$ and 
every  bounded set $\Omega$, $\nabla \psi$ is the solution of \eqref{euclwass} for $\mu:= \nabla \psi \# \Lambda \chi_{\Omega}$ (see \cite[Theorem 2.12]{Viltop} again). \\

By \cite[Theorem 5.5]{Viltop}, we have the time-dependent representation of optimal transport
\begin{equation}\label{dynwass}
 W_2^2(\Lambda \chi_{\Omega}, \mu)=\min_{X} \lt\{ \int_{\Omega}\int_0^1 |\dot{X}(x,t)|^2 \Lambda dx \ : \ X(x,0)=x, \ X(\cdot,1)\# \Lambda \chi_\Omega=\mu\rt\}
\end{equation}
and  if $T$ is the solution of \eqref{euclwass}, the minimizer of \eqref{dynwass} is given by $X(x,t)=T_t(x):= (1-t)x+tT(x)$ which are straight lines. We will often drop the argument $x$ and write $X(t)=X(x,t)$.

As in \cite{GO}, a central point for our analysis is the Eulerian version of optimal transportation, also known as the Benamou-Brenier formulation (see for instance \cite[Theorem 8.1]{Viltop} or \cite[Chapter 8]{AGS}). It states that 
\begin{equation}\label{BBwass}
 W_2^2(\Lambda \chi_{\Omega}, \mu)=\min_{(\rho,j)} \lt\{  \int_{\R^d}\int_0^1 \frac{1}{\rho}|j|^2 \ : \ \partial_t \rho+ \Div j=0, \ \rho(0)=\Lambda \chi_\Omega \textrm{ and } \rho(1)=\mu\rt\},
\end{equation}
where the continuity equation and the boundary data are understood in the distributional sense, i.e.\ for every $\zeta\in C^1_c(\R^d\times [0,1])$,
\begin{equation}\label{conteqprel}
 \int_{\R^d}\int_0^1 \partial_t \zeta \rho +\nabla \zeta \cdot j=\int_{\R^d} \zeta (x,1) d\mu- \int_{\R^d} \zeta(x,0) \Lambda \chi_{\Omega}(x) dx, 
\end{equation}
and where 
\[
 \int_{\R^d}\int_0^1 \frac{1}{\rho}|j|^2= \int_{\R^d}\int_0^1 |v|^2 d\rho
\]
if $j\ll \rho$ with $\frac{dj}{d\rho}=v$ and infinity otherwise (see \cite[Theorem 2.34]{AFP}). Let us point out that in particular, the admissible measures for \eqref{BBwass} are allowed to
contain singular parts with respect to the Lebesgue measure. We also note that if $K\subset \R^d$ is a compact set and if $(\rho,j)$ are measures on $K\times[0,1]$, then we have the equality (see \cite[Proposition 5.18]{Santam}) 
\begin{equation}\label{dualBB}
 \int_{K}\int_{0}^1 \frac{1}{2\rho} |j|^2=\sup_{\xi\in C^0(K\times[0,1],\R^d)} \ \int_{K}\int_0^1 \xi\cdot j-\frac{|\xi|^2}{2} \rho.
\end{equation}

If we let for $t\in [0,1]$, 
\begin{equation}\label{defrhojt}
 \rho_t:=T_t\# \Lambda \chi_{\Omega} \qquad \textrm{and} \qquad j_t:= T_t\#[(T-Id)\Lambda \chi_{\Omega}],
\end{equation}
then $j$ is absolutely continuous with respect to $\rho$ and $(\rho,j)$ is the minimizer of \eqref{BBwass} (see \cite[Theorem 8.1]{Viltop} or \cite[Chapter 8]{AGS} and \cite[Proposition 5.32]{Santam} for the uniqueness). Notice that by \cite[Proposition 5.9]{Viltop}, for $t\in [0,1)$, $\rho_t$ and $j_t$ are absolutely continuous with respect to the Lebesgue measure
and $ \frac{1}{\rho}|j|^2$ agrees with its pointwise definition. By Alexandrov's Theorem \cite[Theorem 14.25]{VilOandN}, $T$ is differentiable a.e. and by \cite[Theorem 4.8]{Viltop} for $t\in[0,1)$, the Jacobian equation
\begin{equation}\label{jacobian}
 \rho_t(T_t(x))\det \nabla T_t(x)=\Lambda \chi_\Omega(x)
\end{equation}
holds a.e.  We say that a map $T$ is monotone if for a.e. $(x,y)$, $(T(x)-T(y))\cdot(x-y)\ge 0$. In particular, since the optimal transport map for \eqref{euclwass} is the gradient of a convex function, it is a monotone map. 
 Let us recall the following $L^\infty$
bound for monotone maps proven in\footnote{This bound was proven there for optimal transport maps but a quick inspection of the proof shows that only monotony is used. Similarly, it is stated there for $R=1$ but a simple rescaling gives the present version of the estimate.} \cite[Lemma 4.1]{GO}.
\begin{lemma}\label{lem:Linftybound}
 Let $T$ be a monotone map. Let $R>0$ be such that  $\frac{1}{R^{d+2}}\int_{B_{2R}} |T-x|^2\ll1$. Then
 \begin{equation}\label{LinftyboundT}
  \sup_{B_{\frac{7R}{4}}} |T-x|+ \sup_{B_{\frac{3R}{2}}}\, \dist(y,T^{-1}(y))\les R \lt( \frac{1}{R^{d+2}}\int_{B_{2R}} |T-x|^2\rt)^{\frac1{d+2}}.
 \end{equation}
Moreover, letting for $t\in [0,1]$, $T_t= (1-t) Id +t T$, 
\begin{equation}\label{inclTt}
 T_t^{-1}(B_{\frac{3R}{2}})\subset B_{2R}.
\end{equation}

\end{lemma}
Let us show how together with displacement convexity this implies  an $L^\infty$ bound for $(\rho_t,j_t)$.
\begin{lemma}\label{lem:Linftyboundslice}
 Let $\Lambda=1$ and assume that $B_2\subset \Omega$ and  $E:=\int_{B_2} |T-x|^2\ll 1$, where $T$ is the optimal transport map for \eqref{euclwass}. Then, for a.e.  $0<t<1$,  if $(\rho_t,j_t)$ is given by \eqref{defrhojt},
\begin{equation}\label{Linfty}
 \sup_{B_{\frac{3}{2}}}\rho_t\le \frac{1}{(1-t)^d} \qquad\textrm{and}\qquad  \sup_{B_{\frac{3}{2}}}|j_t|\les E^{\frac1{d+2}} \frac{1}{(1-t)^d}.
\end{equation}
Moreover,
\begin{equation}\label{Linftyboundslice}
 \int_{B_{\frac{3}{2}}} \frac{1}{\rho_t}|j_t|^2\le E \qquad \textrm{and} \qquad \int_{B_{\frac{3}{2}}} \rho_t\les 1.
\end{equation}
\end{lemma}
\begin{proof}
We start by proving \eqref{Linfty}. The estimate on $\rho$ is a direct consequence of displacement convexity: By concavity of $\mathrm{det}^{\frac1d}$ on positive symmetric matrices,
\begin{equation}\label{detconcave}
 \mathrm{det}^{\frac1d} (\nabla T_t)\ge (1-t) \mathrm{det}^{\frac1d} Id + t\mathrm{det}^{\frac1d} \nabla T\ge (1-t)
\end{equation}
and thus by \eqref{jacobian},
\begin{equation}\label{eq:intermed}
 \rho_t(x)=\frac{1}{\det \nabla T_t(T_t^{-1}(x))}\le \frac{1}{(1-t)^d}.
\end{equation}
We turn to  the  estimate on $j$. For $\xi\in C_c(B_{\frac{3}{2}},\R^d)$, and $t\in(0,1)$
\begin{align*}
 \int_{B_{\frac{3}{2}}} \xi\cdot  j_t &\stackrel{\eqref{defrhojt}}{=}\int_{T_t^{-1}(B_{\frac{3}{2}})} \xi(T_t)\cdot (T-x)\\
 &\le \sup_{T_t^{-1}(B_{\frac{3}{2}})} |T-x| \int_{T_t^{-1}(B_{\frac{3}{2}})} |\xi( T_t)|\\
 &\stackrel{\eqref{inclTt}\& \eqref{LinftyboundT}}{\les}  E^{\frac1{d+2}} \int_{B_{\frac{3}{2}}} |\xi| \rho_t \\
 &\stackrel{\eqref{eq:intermed}}{\les} E^{\frac1{d+2}} \frac{1}{(1-t)^d}\int_{B_{\frac{3}{2}}} |\xi|.
\end{align*}
 Estimate \eqref{Linftyboundslice} then follows from \eqref{inclTt}:
\begin{multline*}
 \int_{B_{\frac{3}{2}}} \frac{1}{\rho_t}|j_t|^2\stackrel{\eqref{defrhojt}}{=}\int_{T_t^{-1}(B_{\frac{3}{2}})}|T-x|^2\le \int_{B_{2}}|T-x|^2=E \\
 \textrm{and} \qquad \int_{B_{\frac{3}{2}}} \rho_t =|T_t^{-1}(B_{\frac{3}{2}})|\les 1.
\end{multline*}
\end{proof}

Before closing this section, let us spend a few words about optimal transportation on the torus and on the sphere since both problems will appear later on. Let us start with the periodic setting. For $L>0$, we let $Q_L:=[-\frac{L}{2},\frac{L}{2})^d$
be the centered cube of side length $L$.  We say that a measure $\mu$ on $\R^d$ is $Q_L-$periodic\footnote{when it is clear from the context we will simply call them periodic measures} if for every $z\in (L\Z)^d$, and every measurable set $A$, 
\[
 \mu(A+z)=\mu(A)
\]
 so that we may identify measures on the flat torus of size $L>0$ with $Q_L-$periodic measures on $\R^d$. If $\mu$ is a  $Q_L-$periodic measure  and $\Lambda:=\frac{\mu(Q_L)}{L^d}$, then we can define
\begin{equation}\label{perwass}
\Wper(\Lambda,\mu):=  \min_{T\# \Lambda=\mu} \int_{Q_L}|T-x|_{\per}^2\, \Lambda dx,
\end{equation}
where $|\cdot|_\per$ denotes the distance on $\T_L$ i.e. $|x-y|_\per=\min_{z\in (L\Z)^d} |x-y+z|$. 

 By \cite[Theorem 2.47]{Viltop} (see also \cite{cordero,ACDF} for a simpler proof), 
there exists a unique (up to  additive constants) convex function $\psi$ on $\R^d$ such that if $T$ is  the unique solution  of \eqref{perwass}, then for $x\in Q_L$, $T(x)=\nabla \psi(x)$ and
for $(x,z)\in \R^d\times (L\Z)^d$,
\begin{equation}\label{Cordero}
  \nabla \psi(x+z)=\nabla \psi(x)+z.
\end{equation}
\begin{remark}\label{rem:selection}
We will often identify  $T$ and $\nabla \psi$. Let us point out that although $T$ is defined  only Lebesgue a.e., it will be sometimes useful
 to consider a pointwise defined map which we then take to be an arbitrary but fixed measurable selection of the subgradient $\partial \psi$ of $\psi$.
\end{remark}

 Notice that of course for every $Q_L-$periodic measure $\mu$, 
\[
 \Wper(\Lambda,\mu)\le W_2^2\lt(\Lambda \chi_{Q_L}, \mu \restr Q_L\rt).
\]

Finally for $f_1$ and $f_2$ two non-negative densities on $\partial B_1$ with $\int_{\partial B_1} f_1=\int_{\partial B_1} f_2$, we define 
\begin{equation}\label{spherewass}
 W_{\partial B_1}^2(f_1,f_2):=\min_{T\# f_1=f_2}\int_{\partial B_1} d_{\partial B_1}^2(T(x),x) df_1,
\end{equation}
where $d_{\partial B_1}$ is the geodesic distance on $\partial B_1$. Let us point out that a minimizer exists by McCann's extension of Brenier's Theorem \cite[Theorem 2.47]{Viltop}.
Notice that $W_{\partial B_1}^2(f_1,f_2)$ is comparable to the Wasserstein distance in $\R^d$ between $f_1 \H^{d-1}\LL \partial B_1$ and $f_2 \H^{d-1}\LL \partial B_1$, that is 
\begin{equation}\label{W2compare}
 W_2^2(f_1,f_2)\le W_{\partial B_1}^2(f_1,f_2)\les W_2^2(f_1,f_2).
\end{equation}
As for \eqref{dynwass}, we have the time-dependent formulation \cite[Theorem 5.6]{Viltop}
\begin{equation}\label{dynsphere}
 W_{\partial B_1}^2(f_1,f_2)=\min_{X}  \lt\{ \int_{\partial B_1}\int_0^1 |\dot{X}(x,t)|^2 df_1 \ : \ X(x,0)=x, \ X(\cdot,1)\# f_1=f_2\rt\},
\end{equation}
and the Benamou-Brenier formulation \cite[Theorem 13.8]{VilOandN}
\begin{equation}\label{BBsphere}
 W_{\partial B_1}^2(f_1,f_2)=\min_{(\rho,j)} \lt\{ \int_{\partial B_1}\int_{0}^1 \frac{1}{\rho} |j|^2 \ : \ \partial_t \rho+ \nabla^\bdr \cdot j=0, \ \rho(0)= f_1 \textrm{ and } \rho(1)=f_2\rt\},
\end{equation}
where we stress that $j$ is tangent to $\partial B_1$. Even though it is more delicate than in the Euclidean case, the analog of \eqref{jacobian} also holds in this case.
Indeed, by  \cite[Theorem 13.8]{VilOandN} and \cite[Theorem 11.1]{VilOandN} (see also \cite[Theorem 4.2]{CorderoMCannS}  and \cite[Lemma 6.1]{CorderoMCannS} for more details) 
if\footnote{here $\exp_x$ denotes the exponential map on $\partial B_1$} $T(x)=\exp_x(\nabla \psi(x))$ is the minimizer of \eqref{spherewass}, letting for $t\in [0,1]$, 
$T_t(x):=\exp_x(t\nabla \psi(x))$  and then $\rho_t:= T_t\# f_1$, we have that $\rho_t$ is a minimizer of \eqref{BBsphere} and 
for every  $t\in[0,1]$ and a.e. $x\in \partial B_1$, the Jacobian equation
\begin{equation}\label{jacobiansphere}
 \rho_t(T_t(x)) J_t(x)= f_1(x)
\end{equation}
holds, with $J_t$ the Jacobian determinant (see for instance \cite[Lemma 6.1]{CorderoMCannS} for its definition).  The main point for us is that similarly to \eqref{detconcave}, it satisfies by \cite[Theorem 14.20]{VilOandN} (see also \cite[Lemma 6.1]{CorderoMCannS} for a  statement closer to ours)
\begin{equation}\label{displconvsphere}
 J_t^{\frac{1}{d-1}}\ge (1-t) + t J_1^{\frac{1}{d-1}},
\end{equation}
where we used  the fact that since  the sphere has positive Ricci curvature, its volume distortion coefficients are larger than one.

We finally prove a simple lemma which will be useful in the construction of competitors in Proposition \ref{prop:distjphiprop2} below. 
\begin{lemma}\label{lem:defboundaryflux}
 Let $f_1$ and $f_2$ be two non-negative densities on $\partial B_1$ of equal mass. For every $0\le a< b\le 1$ and $0\le c\le d\le 1$ with $a<c$ and $b<d$, there exists $(\rho,j)$ supported on $\partial B_1 \times [a,d]$ such that for every $\zeta\in C^1(\partial B_1\times[0,1])$
 \begin{equation}\label{conteqbdrconstruct}
  \int_{ \partial B_1}\int_0^1 \partial_t \zeta \rho+\nabbdr \zeta\cdot j= \frac{1}{d-c}\int_{\partial B_1}\int_{c}^d  \zeta df_2-\frac{1}{b-a} \int_{\partial B_1}\int_a^b \zeta df_1
 \end{equation}
and 
\begin{equation}\label{estimW2construct}
\int_{\partial B_1}\int_0^1 \frac{1}{\rho}|j|^2\les \frac{W^2_{\partial B_1}(f_1,f_2)}{(d-b)-(c-a)}\log \frac{d-b}{c-a},  
\end{equation}
with the understanding that for $c=d$,
\[
 \frac{1}{d-c}\int_{\partial B_1}\int_{c}^d  \zeta df_2=\int_{\partial B_1} \zeta(\cdot,c) df_2
\]
and for $d-b=c-a$,
\begin{equation}\label{dmoinsb}
 \frac{1}{(d-b)-(c-a)}\log \frac{d-b}{c-a}=\frac{1}{c-a}.
\end{equation}
\end{lemma}
\begin{proof}
 For $x\in \partial B_1$ and $t\in[0,1]$, let $X(x,t)\in \partial B_1$ be the minimizer of \eqref{dynsphere}, i.e. for $\zeta\in C^0(\partial B_1)$, 
\begin{equation}\label{def:X}
 \int_{\partial B_1} \zeta(X(x,1)) d f_1= \int_{\partial B_1} \zeta d f_2 \qquad \textrm{and} \qquad W_{\partial B_1}^2(f_1,f_2)=\int_{\partial B_1}\int_0^1 |\dot{X}|^2 df_1.
\end{equation}
Let $\psi$ be the affine function defined on $[a,b]$ through
\[
 \psi(a)=c \qquad \textrm{and } \qquad \psi(b)=d
\]
and let then for $t\in [s,\psi(s)]$
\[
 X_s(x,t):=X\lt(x,\frac{t-s}{\psi(s)-s}\rt),
\]
so that 
\begin{equation}\label{boundaryX}
 X_s(x,s)=x \qquad \textrm{and }\qquad X_s(x,\psi(s))=X(x,1).
\end{equation}
We then let $\rho$ be the non-negative measure on $\partial B_1\times[0,1]$ defined for $\zeta\in C^0(\partial B_1\times[0,1])$ by 
\[
 \int_{\partial B_1}\int_0^1 \zeta d \rho:=\frac{1}{b-a}\int_{\partial B_1} \int_{a}^{b} \int_{s}^{\psi(s)} \zeta(X_s(x,t),t) dt ds df_1
\]
and $j$ the $\R^d-$valued measure defined for $\xi\in C^0(\partial B_1\times[0,1],\R^d)$  by 
\[
 \int_{\partial B_1}\int_0^1 \xi\cdot d j:=\frac{1}{b-a}\int_{\partial B_1} \int_{a}^{b} \int_{s}^{\psi(s)} \xi(X_s(x,t),t)\cdot \dot{X_s}(x,t) dt ds df_1.
\]
It is readily seen that with this definition $j\ll \rho$. Let us establish \eqref{conteqbdrconstruct}. For $\zeta\in C^1(\partial B_1\times [0,1])$ a test function,
\begin{align*}
 \int_{\partial B_1}\int_0^{1} \partial_t \zeta \rho+ \nabbdr \zeta\cdot j&= \frac{1}{b-a}\int_{\partial B_1} \int_{a}^{b} \int_{s}^{\psi(s)} \partial_t \zeta(X_s,t)+ \nabla \zeta(X_s,t)\cdot \dot{X_s}dt ds df_1 \\
 &=\frac{1}{b-a}\int_{\partial B_1} \int_{a}^{b} \zeta(X_s(x,\psi(s)),\psi(s))-\zeta(X_s(x,s),s) ds df_1\\
 &\stackrel{\eqref{boundaryX}}{=}\frac{1}{b-a}\int_{\partial B_1} \int_{a}^{b} \zeta(X(x,1),\psi(s))-\zeta(x,s)  dsdf_1\\
 &\stackrel{\eqref{def:X}}{=}\int_{\partial B_1} \int_{a}^{b}\frac{1}{b-a} \zeta(x,\psi(s)) df_2 ds -\int_{\partial B_1} \int_{a}^{b} \frac{1}{b-a} \zeta(x,s) dsdf_1\\
 &=\int_{\partial B_1} \int_{c}^{d}\frac{1}{d-c} \zeta d f_2 d \hat{s} -\int_{\partial B_1} \int_{a}^{b} \frac{1}{b-a}\zeta  dsdf_1,
\end{align*}
where we made the change of variables $\hat{s}=\psi(s)$ in the  last equality.\\
We now turn to \eqref{estimW2construct}. Using \eqref{dualBB} and the definition of $(\rho,j)$,
\begin{align*}
 \int_{\partial B_1}\int_0^1 \frac{1}{2\rho}|j|^2&=\sup_{\xi\in C^0(\partial B_1\times[0,1],\R^d)} \int_{\partial B_1} \int_0^1 \xi\cdot j-\frac{|\xi|^2}{2}\rho\\
 &=\sup_{\xi\in C^0(\partial B_1\times[0,1],\R^d)}\frac{1}{b-a}\int_{\partial B_1}\int_{a}^{b} \int_s^{\psi(s)} \xi(X_s(x,t),t)\cdot \dot{X}_s(x,t)\\
 &\hspace{7cm} -\frac{|\xi(X_s(x,t),t)|^2}{2}dt ds df_1\\
 &\le \frac{1}{b-a}\int_{\partial B_1}\int_{a}^{b} \int_s^{\psi(s)}\sup_{\xi\in \R^d} \lt( \xi\cdot\dot{X}_s(x,t)-\frac{1}{2} |\xi|^2 \rt) dt ds df_1\\
 &= \frac{1}{b-a}\int_{\partial B_1}\int_{a}^{b} \int_s^{\psi(s)} \frac{1}{2}|\dot{X_s}(x,t)|^2 dt ds df_1\\
 &\stackrel{\eqref{boundaryX}}{=}\frac{1}{b-a}\int_{\partial B_1}\int_{a}^{b} \int_s^{\psi(s)} \frac{1}{(\psi(s)-s)^2}\lt|\dot{X}\lt(x,\frac{t-s}{\psi(s)-s}\rt)\rt|^2dt ds df_1\\
 &\stackrel{\hat{t}=\frac{t-s}{\psi(s)-s}}{=}\frac{1}{b-a}\int_{\partial B_1}\int_{a}^{b}\int_0^1 \frac{1}{\psi(s)-s}|\dot{X}(x,\hat t)|^2 d\hat t ds df_1\\
 &=\frac{W^2_{\partial B_1}(f_1,f_2)}{b-a}\int_{a}^{b} \frac{ds}{\psi(s)-s}.
\end{align*}
Let us finally estimate $\frac{1}{b-a}\int_{a}^{b} \frac{1}{\psi(s)-s}$. By definition of $\psi$,
\begin{align*}
 \frac{1}{b-a}\int_a^b \frac{1}{\psi(s)-s}&=\frac{1}{b-a}\int_a^b \frac{1}{c-s+\frac{d-c}{b-a}(s-a)}\\
 &\stackrel{s=(1-t)a+tb}{=}\int_0^1 \frac{1}{(1-t)(c-a)+t(d-b)}\\
 &=\frac{1}{(d-b)-(c-a)}\log \frac{d-b}{c-a},
\end{align*}
where we take as convention \eqref{dmoinsb} if $d-b=c-a$.

\end{proof}

\subsection{Poisson process, optimal matching, and concentration}\label{sec:Poi}
\subsubsection{The Euclidean problem}\label{sec:PoisEuc}
Let  $\Gamma$ be the set of all locally finite counting measures on $\R^d$
$$\Gamma=\cur{\mu : \mu=\sum_i \delta_{y_i}, y_i\in \R^d, \mu(K)<\infty, \forall \, K \text{ compact }}, $$
where $\Gamma$ is equipped with the $\sigma$-field $\F$ generated by the mappings $\mu\mapsto \mu(A)$ for Borel sets $A\subset \R^d$. We say that $(\mu, \PP)$ (or simply $\mu$) 
is a Poisson point process with intensity measure Lebesgue (or simply a Poisson point process) if $\PP$ is a probability measure on $\Gamma$ such that 
\begin{itemize}
 \item[(i)] If $A_1,\ldots,A_k$ are disjoint Borel sets, then $\mu(A_1),\ldots,\mu(A_k)$ are independent integer valued random variables;
\item[(ii)] for any  Borel set $A$ with $|A|<\infty$, the random variable $\mu(A)$ has a Poisson distribution with parameter $|A|$ i.e. for every $n\in \N$,
\begin{equation}\label{def:Poisson}
 \PP\lt[ \mu(A)=n\rt]= \exp(-|A|) \frac{|A|^n}{n!}.
\end{equation}

\end{itemize}
For a set $\Omega\subset \R^d$, we define the Poisson point process on $\Omega$ as the restriction of the Poisson point process on $\R^d$ to $\Omega$. It could be equivalently defined through properties $(i)$ and $(ii)$ above restricted to subsets of $\Omega$.\\

We let $\theta : \R^d\times \Gamma\to \Gamma$ be the shift operator, that is for $(z,\mu)\in \R^d\times \Gamma$ and a Borel set $A\subset \R^d$,
\begin{equation}\label{eq:equivariance}
\theta_z\mu(A):=\mu(A+z),\end{equation}
which we write shortly as $\theta_z\mu(\cdot)=\mu(\cdot + z)$. Moreover, we note that $\PP$ is stationary in the sense that it is invariant under the action of $\theta,$ i.e. $\PP\circ \theta = \PP$.
\\

For $\ell>0$, we recall that $Q_\ell=[-\frac{\ell}{2},\frac{\ell}{2})^d$. The optimal matching problem consists in understanding the behavior as $\ell\to \infty$ of\footnote{in order to have shift-invariance properties, we will actually consider periodic variants of \eqref{tocontrol}, see below.} 
\begin{equation}\label{tocontrol}
 W_2^2\lt(\mu\restr Q_\ell, \frac{\mu(Q_\ell)}{\ell^d}\chi_{Q_\ell}\rt)
\end{equation}
together with the properties of the corresponding optimal transport maps. We will use as shorthand notation $\mu_\ell:=\mu\restr Q_\ell$ and when it is clear from the context we will identify a constant $\Lambda>0$ with the measure $\Lambda \chi_{Q_\ell}$. As explained in the introduction, it is known that $ \EE\sqa{W_2^2\lt(\mu_\ell, \frac{\mu(Q_\ell)}{\ell^d}\rt)}\sim \ell^d$ for $d\geq 3$ whereas in dimension two there is an extra logarithmic 
factor in the scaling of \eqref{tocontrol} resulting from larger shifts of the mass. 
In particular, the transport cost per unit volume diverges  logarithmically. Our aim is to investigate the behavior of the associated transport maps. Although our techniques also allow us to say something for the case of $d\geq 3$ 
we focus on the  case $d=2$ where an additional renormalization of the maps is needed in order to be able to pass to the limit. Hence, we assume from now on $d=2$.
\\
The main stochastic ingredient is a control at every scale  $1\ll\ell< \infty$ of  \eqref{tocontrol}.  
This estimate is a quite direct consequence  of a result of Ambrosio, Stra and Trevisan \cite{AmStTr16} which we now recall.
  
  Since the results of \cite{AmStTr16} are not stated for the Poisson point process but rather for a deterministic number $n\to \infty$ of uniform iid random variables $X_i$ on a given domain $Q_\ell$, we need to introduce some more notation. For a given $n\in \N$, we let the probability $\PP_n$ on $\Gamma$   be defined as 
  \[
   \PP_n\sqa{F}:= \frac{\PP\sqa{F\cap\{ \mu_\ell(Q_\ell)=n\}}}{\PP\sqa{\mu_\ell(Q_\ell)=n}}, 
  \]
and let $\EE_n$ be the associated expectation. Note that by \eqref{def:Poisson}, we have
\begin{align}\label{eq:pn}
 p_n:=\PP\sqa{\mu_\ell(Q_\ell)=n}=\exp(-\ell^2)\frac{\ell^{2n}}{n!}.
\end{align}
Equipped with this probability measure, $\mu_\ell$ can be identified with $n$ uniformly iid random variables $X_i$ on $Q_\ell$. 
A simple rescaling shows that
\begin{align*}
 \frac{1}{\ell^2\log n} \EE_n\sqa{W_2^2\bra{\mu_\ell,\frac{\mu(Q_\ell)}{\ell^2} }}
 \end{align*}
is independent of $\ell$ and \cite[Theorem 1.1]{AmStTr16} states that,
 \begin{align}\label{eq:AST16}
 \lim_{n\to\infty} \frac{1}{\ell^2\log n} \EE_n\sqa{W_2^2\bra{\mu_\ell,\frac{\mu(Q_\ell)}{\ell^2} }} = \frac1{{4}\pi}.
 \end{align}
Arguing as in \cite[Remark 4.7]{AmStTr16} and using the fact that the uniform measure on $[0,1]$ satisfies a log-Sobolev inequality to replace exponential bounds by Gaussian bounds,
this gives that  there exists $c_0>0$ such that for every $M\ge 1$ (since we pass from a deviation to a tail estimate) and $n$ large enough uniformly\footnote{Notice that the left-hand side of \eqref{eq:concAmStTr} actually does not depend on $\ell$.} in $\ell$,

\begin{equation}\label{eq:concAmStTr}
 \PP_n\sqa{\frac{1}{\ell^2 \log n} W_2^2\lt(\mu_\ell, \frac{\mu(Q_\ell)}{\ell^2}\rt)\ge M}\le \exp(-c_0 M \log n).
\end{equation}
Let us now show how \eqref{eq:AST16} and \eqref{eq:concAmStTr} translate into our setting. 
\begin{proposition}\label{prop:concentration}
Let $\mu$ be a Poisson point process on $\R^2$. Then,
\begin{equation}\label{eq:convexpectation}
 \lim_{\ell\to \infty} \frac{1}{\ell^2 \log \ell}\EE\sqa{W_2^2\lt(\mu_\ell, \frac{\mu(Q_\ell)}{\ell^2}\rt)}= \frac{1}{2\pi}
\end{equation}
and there exists a universal constant $c$ independent of $\ell$ and $M$ such that for $\ell\ge 2, M \ge 1$,
\begin{align}\label{eq:concentration}
 \PP\sqa{\frac{1}{\ell^2 \log \ell} W_2^2\left(\mu_\ell, \frac{\mu(Q_\ell)}{\ell^2}\right) \geq M} \leq \exp(-c M \log \ell)\ . 
\end{align}
\end{proposition}

\begin{proof} 
We first start by noting that by Cram\'er-Chernoff's bounds for the Poisson distribution with intensity $\ell^2$ (see \cite{BouLuMa}), there exists a constant $c$ such that
\begin{equation}\label{chernoff}
 \PP\sqa{\frac{\mu(Q_\ell)}{\ell^2}\notin \sqa{\frac{1}{2},2}}\le \exp(-c \ell^2),
\end{equation}
and for $M\gg 1$
\begin{equation}\label{chernoff2}
 \PP\sqa{\frac{\mu(Q_\ell)}{\ell^2}\ge M}\le \exp(-c \ell^2 M).
\end{equation}
Let us now prove \eqref{eq:convexpectation}. 
Recall $p_n$ from \eqref{eq:pn}.
 By definition of  $\EE_n$ we have  
\begin{multline}\label{convexpectstep1}
 \frac{1}{\ell^2 \log \ell}\EE\sqa{W_2^2\lt(\mu_\ell, \frac{\mu(Q_\ell)}{\ell^2}\rt)}=\frac{1}{\ell^2 \log \ell} \sum_{n\notin[\ell^2/2, 2 \ell^2]} p_n \EE_n\sqa{W_2^2\lt(\mu_\ell, \frac{\mu(Q_\ell)}{\ell^2}\rt)}\\
  +\frac{1}{\ell^2 \log \ell} \sum_{n\in[\ell^2/2,2 \ell^2]} p_n \EE_n\sqa{W_2^2\lt(\mu_\ell, \frac{\mu(Q_\ell)}{\ell^2}\rt)}.
\end{multline}
Using the crude transport estimate $W_2^2\lt(\mu_\ell, \frac{\mu(Q_\ell)}{\ell^2}\rt)\le  \ell^2 \mu(Q_\ell)$ together with \eqref{eq:pn} we can estimate the first term as 
\begin{align*}
 \frac{1}{\ell^2 \log \ell} \sum_{n\notin[\ell^2/2, 2 \ell^2]} p_n \EE_n\sqa{W_2^2\lt(\mu_\ell, \frac{\mu(Q_\ell)}{\ell^2}\rt)}&\le \frac{1}{ \log \ell} \sum_{n\notin[\ell^2/2, 2 \ell^2]} n p_n\\
 &=  \frac{1}{ \log \ell} \sum_{n\notin[\ell^2/2, 2 \ell^2]} \exp(-\ell^2) \frac{\ell^{2n}}{(n-1)!}\\
 &\les \frac{\ell^2}{\log \ell} \PP\sqa{\frac{\mu(Q_\ell)}{\ell^2}\notin[\frac{1}{2},2]}\\
 &\stackrel{\eqref{chernoff}}{\les} \frac{\ell^2 \exp(-c \ell^2)}{\log \ell}, 
\end{align*}
which goes to zero as $\ell\to \infty$. The second term in \eqref{convexpectstep1} can be rewritten as 
\begin{multline*}
 \frac{1}{\ell^2 \log \ell} \sum_{n\in[\ell^2/2,2 \ell^2]} p_n \EE_n\sqa{W_2^2\lt(\mu_\ell, \frac{\mu(Q_\ell)}{\ell^2}\rt)}\\
 = \sum_{n\in[\ell^2/2,2 \ell^2]} p_n \frac{\log n}{\log \ell} \lt( \frac{1}{\ell^2 \log n}\EE_n\sqa{W_2^2\lt(\mu_\ell, \frac{\mu(Q_\ell)}{\ell^2}\rt)}\rt). 
\end{multline*}
Since by \eqref{eq:AST16}, $\frac{1}{\ell^2 \log n}\EE_n\sqa{W_2^2\lt(\mu_\ell, \frac{\mu(Q_\ell)}{\ell^2}\rt)}$ converges uniformly in $\ell$ to $\frac{1}{4\pi}$, it is enough to show that 
\begin{equation*}\label{toproveconvexpect}
 \lim_{\ell\to \infty}  \sum_{n\in[\ell^2/2,2 \ell^2]} p_n \frac{\log n}{\log \ell}=2.
\end{equation*}
This is a simple consequence of \eqref{chernoff} and the fact that 
\[
 \frac{ 2\log \ell -\log 2}{\log \ell} \sum_{n\in[\ell^2/2,2 \ell^2]} p_n\le \sum_{n\in[\ell^2/2,2 \ell^2]} p_n \frac{\log n}{\log \ell}\le \frac{ 2\log \ell +\log 2}{\log \ell} \sum_{n\in[\ell^2/2,2 \ell^2]} p_n.
\]
We now turn to \eqref{eq:concentration}. For $1\le M\les \frac{\ell^2}{\log \ell}$, by definition of $\PP_n$ 
\begin{align*}
 \lefteqn{\PP\sqa{\frac{1}{\ell^2 \log \ell} W_2^2\left(\mu_\ell, \frac{\mu(Q_\ell)}{\ell^2}\right) \geq M}}\\
 &= \sum_{n\notin [\ell^2/2,2 \ell^2]} p_n \PP_n\sqa{\frac{1}{\ell^2 \log \ell} W_2^2\left(\mu_\ell, \frac{\mu(Q_\ell)}{\ell^2}\right) \geq M}\\
 &\qquad \qquad +\sum_{n\in[\ell^2/2,2 \ell^2]} p_n \PP_n\sqa{\frac{1}{\ell^2 \log \ell} W_2^2\left(\mu_\ell, \frac{\mu(Q_\ell)}{\ell^2}\right) \geq M}\\
 &\le \PP\sqa{\frac{\mu(Q_\ell)}{\ell^2}\notin\sqa{\frac{1}{2},2}}+\sum_{n\in[\ell^2/2,2 \ell^2]} p_n \PP_n\sqa{\frac{1}{\ell^2 \log \ell} W_2^2\left(\mu_\ell, \frac{\mu(Q_\ell)}{\ell^2}\right) \geq M}\\
 &\stackrel{\eqref{chernoff}}{\le}  \exp(-c\ell^2) + \sum_{n\in[\ell^2/2,2 \ell^2]} p_n\PP_n\sqa{\frac{1}{\ell^2 \log n} W_2^2\left(\mu_\ell, \frac{\mu(Q_\ell)}{\ell^2}\right) \geq \frac{\log \ell}{ 2\log \ell +2}M}\\
 &\stackrel{\eqref{eq:concAmStTr}}{\le}\exp(-c\ell^2) + \sum_{n\in[\ell^2/2,2 \ell^2]} p_n \exp(-\frac{c_0}{4} M \log n)\\
 &\le \exp(-c\ell^2) +\exp(-\frac{c_0}{4} M \log \ell) \le \exp(-c_1 M \log \ell),
\end{align*}
while for $M\gg \frac{\ell^2}{\log \ell}$, using once more the estimate $W_2^2\lt(\mu_\ell, \frac{\mu(Q_\ell)}{\ell^2}\rt)\le \ell^2 \mu(Q_\ell)$ together with \eqref{chernoff2}, we obtain
\[
 \PP\sqa{\frac{1}{\ell^2 \log \ell} W_2^2\left(\mu_\ell, \frac{\mu(Q_\ell)}{\ell^2}\right) \geq M}\le \PP\sqa{\frac{\mu(Q_\ell)}{\ell^2}\ge  M \frac{\log \ell}{\ell^2}}\le \exp(-c M \log \ell).
\]

\end{proof}
By a Borel-Cantelli  argument we can now strengthen \eqref{eq:concentration} into a supremum bound.

\begin{theorem}\label{theo:estimDrstar}
 Let $\mu$ be a Poisson point process on $\R^2$. Then, there exist a universal constant $c$  and a  random variable\footnote{notice that we keep implicit here the dependence on $\mu$} $r_*=r_*(\mu)\ge 1$ with $\textstyle\EE\sqa{\exp\bra{\frac{c r_*^2}{\log(2r_*)}}}<\infty$ such that for every  dyadic $\ell$ with $ 2 r_*\le \ell$,
 \begin{equation}\label{eq:concentrationrstar}
  \frac{1}{\ell^4}W_2^2\lt(\mu_\ell, \frac{\mu(Q_\ell)}{\ell^2}\rt)\le \frac{ \log\bra{\frac{\ell}{r_*}}}{\lt(\frac{\ell}{r_*}\rt)^2}.
 \end{equation}
\end{theorem}
\begin{proof}
 We first prove that there exist a constant $\bar c>0$ and a    random variable $\Theta$  with $\textstyle\EE\sqa{\exp(\bar c\Theta)}<\infty$
  such that for all dyadic   $\ell$ with $ \ell\gg 1$,
\begin{equation}\label{eq:concentrationtheta}
  \frac{1}{\ell^2 \log \ell}  W_2^2\left(\mu_\ell, \frac{\mu(Q_\ell)}{\ell^2}\right)\leq \Theta  .
\end{equation}
For  $ k\ge 1$, let $\ell_k:= 2^k$ and put 
\begin{equation}\label{def:theta}
\Theta_k:= \frac{1}{\ell_k^2 \log \ell_k} W_2^2\left(\mu_{\ell_k}, \frac{\mu(Q_{\ell_k})}{\ell_k^2}\chi_{Q_{\ell_k}}\right), \quad \Theta:=\sup_{ k\geq 1}\Theta_k.
\end{equation}
We claim that the exponential moments of $\Theta_k$ given by Proposition \ref{prop:concentration} translate into exponential moments for $\Theta$. Indeed fix  $1\gg \bar c>0$. Then, we estimate
\begin{align*}
 \EE[\exp(\bar c \Theta)] &\leq \exp(\bar c) + \sum_{ M\in\N}  \PP[\Theta\geq M]\exp(\bar c (M+1)) \\
& \leq \exp(\bar c ) + \sum_{M\in \N} \exp(\bar c (M+1)) \sum_{k\geq 1} \PP[\Theta_k \geq M] \\
& \stackrel{\eqref{eq:concentration}}{\leq} \exp(\bar c ) +  \exp(\bar c)\sum_{k\geq 1} \sum_{M\in\N}  \exp(-M(c\log \ell_k-\bar c))\\
& \stackrel{\ell_k=2^k \& \, \bar c\ll1}{\lesssim} \exp(\bar c) + \exp(\bar c) \sum_{k\geq 1} \exp( -  c  k) <\infty.
\end{align*}
Hence, $\Theta$ has exponential moments and \eqref{eq:concentrationtheta} is satisfied for every large enough dyadic $\ell$.\\

Define $r_* \ge  1$ via the equation 
\begin{equation}\label{def:rstar} \frac{r_*^2}{\log \lt( 2r_*\rt)}  = \frac{\Theta}{\log 2} ,\end{equation}
which has a solution since $r_*\mapsto r_*^2/\log \lt( 2r_*\rt)$ is monotone on $(e/2,\infty)$. 
Since $\ell\mapsto \frac{ \log \lt(\frac{\ell}{r_*}\rt)}{\log \ell}$ is an increasing function, we have for $\ell\ge 2 r_*$,
\[
  \log \ell\le \frac{ \log (2 r_*)}{ \log 2 } \log \lt(\frac{\ell}{r_*}\rt)
\]
which together with \eqref{def:rstar} gives for every dyadic $\ell\ge 2 r_*$,
\[
 \Theta  \log \ell\le r_*^2 \log \lt(\frac{\ell}{r_*}\rt),
\]
from which we see that \eqref{eq:concentrationtheta} implies \eqref{eq:concentrationrstar}.
\end{proof}

We remark that  $r_*$ inherits all stationarity properties of the Poisson process as a measurable function of the Poisson process (similarly  for $r_{*,L}$ in  Theorem \ref{theo:estimDrstarper}). We will not explicitly mention this in the sequel.

\subsubsection{The periodic problem}
Since our aim is to construct a coupling on $\R^2$ between Lebesgue and Poisson which keeps some of the shift-invariance properties of the Poisson point process,
it is more convenient to work for finite-size cubes also with a shift-invariant point process. For $L>1$ let us introduce the $Q_L-$periodic Poisson point process (which can be identified with the Poisson point process on the flat torus of size $L$). For $\mu\in \Gamma$, 
we let $\mu^\per_L$ be the $Q_L-$periodic extension of $\mu\restr Q_L$ and then 
\begin{equation}\label{def:PL}
 \PP_L:= \mu^\per_L\# \PP.
\end{equation}
We denote by $\EE_L$ the expectation with respect to $\PP_L$. We then call $(\mu,\PP_L)$ (or simply $\mu$ when it is clear from the context) a $Q_L-$periodic Poisson point process (or Poisson point process on the torus). Since $\PP$ is invariant under $\theta$, so is $\PP_L$. Notice that for $\ell\le L$ the restriction of a $Q_L-$periodic
Poisson point process to $Q_\ell$ is a Poisson point process on $Q_\ell$ in the sense of Section \ref{sec:PoisEuc}. 

For $\mu\in \Gamma$ and $L>1$, our main focus will be to understand at every scale $1\ll \ell\le L$ the structure of the optimal transport map $T_{\mu,L}$ on the torus
 between $\frac{\mu(Q_L)}{L^2}$ and $\mu$, i.e. $T_{\mu,L}$ is the unique minimizer of \eqref{perwass}. We will often identify $T_{\mu,L}$ with $\nabla \psi_{\mu,L}$ where $\psi_{\mu,L}$ is the convex potential given in \eqref{Cordero}. 
 When it is clear from the context, we drop the dependence of $T_{\mu,L}$ and $\psi_{\mu,L}$ on either $\mu$, $L$ or both.  Uniqueness of the optimal transport map solving \eqref{perwass}
 implies that  $T_\mu$ is covariant in the sense that   

\begin{equation}\label{eq:covariantliftT}
 T_\mu(x+z)= T_{\theta_z \mu}(x)+z \qquad x,z\in \R^2.
\end{equation}
For future reference, let us prove a corresponding stationarity property of the potentials.
\begin{lemma}\label{lem:stationa}
 Let $\mu$ be a $Q_L-$periodic Poisson point process and let $T_\mu=\nabla \psi_\mu$ be the optimal transport map between $\frac{\mu(Q_L)}{L^2}$ and $\mu$ on the torus. Then, for $z\in \R^2$,
 \begin{equation}\label{eq:statpsi}
  \psi_{\theta_z \mu}(x)=\psi_{\theta_z \mu}(0)-\psi_\mu(z)+\psi_\mu(x+z)-z\cdot x \qquad \forall x\in \R^2
 \end{equation}
and if $\psi^*$ is the convex conjugate of $\psi$, 
\begin{equation}\label{eq:statpsistar}
 \psi^*_{\theta_z \mu}(y)=\psi^*_{\theta_z \mu}(0)-\psi^*_\mu(z)+\psi^*_\mu(y+z)-z\cdot y \qquad \forall y\in \R^2.
\end{equation}
As a consequence, if we let for $h\in \R^2$, $D^2_h \psi(x):= \psi(x+h)+\psi(x-h)-2\psi(x)$,
\begin{equation}\label{D2hpsi}
 D^2_h \psi_{\theta_z \mu} (x)= D^2_h \psi_\mu(x+z) \qquad \textrm{and} \qquad  D^2_h \psi^*_{\theta_z \mu} (y)= D^2_h \psi^*_\mu(y+z).
\end{equation}

\end{lemma}
\begin{proof}
 Equation \eqref{eq:statpsi} is a direct consequence of \eqref{eq:covariantliftT} so that we just need to prove that it implies \eqref{eq:statpsistar}. By definition,
 \begin{align*}
  \psi^*_{\theta_z \mu}(y)&=\sup_x \lt[x\cdot y-\psi_{\theta_z \mu}(x)\rt]\\
  &\stackrel{\eqref{eq:statpsi}}{=} \sup_x \lt[x\cdot y-\psi_{\theta_z \mu}(0)+\psi_\mu(z)-\psi_\mu(x+z)+ z\cdot x\rt]\\
  &= -\psi_{\theta_z \mu}(0)+\psi_\mu(z)-y\cdot z-|z|^2+\sup_x \lt[(x+z)\cdot(y+z)-\psi_\mu(x+z)\rt]\\
  &=-\psi_{\theta_z \mu}(0)+\psi_\mu(z)-y\cdot z-|z|^2+ \psi^*_\mu(y+z).
 \end{align*}
 Applying this to $y=0$, we obtain
 \[
  \psi^*_{\theta_z \mu}(0)-\psi^*_\mu(z)=-\psi_{\theta_z \mu}(0)+\psi_\mu(z)-|z|^2,
 \]
 so that \eqref{eq:statpsistar} follows.
\end{proof}

Let us finally translate the result of Theorem \ref{theo:estimDrstar} into the periodic setting.  In particular, the following result contains Theorem \ref{theo:stoch intro}.
\begin{theorem}\label{theo:estimDrstarper}
 There exists a universal constant $c$ such that for $L=2^k$, $k\in \N$, dyadic and $\mu$  a $Q_L-$periodic Poisson point process, there exists a family of   random variables $r_{*,L}\ge 1$
 with $\sup_L \EE_L\sqa{\exp\lt(\frac{c r_{*,L}^2}{\log (2r_{*,L})}\rt)}<\infty$ such that if $ 2r_{*,L}\le L$ we have
 \begin{equation}\label{eq:goodboundmu}
  \frac{\mu(Q_L)}{L^2}\in \lt[\frac{1}{2},\, 2\rt], \qquad \spt \mu \cap B_{r_{*,L}}\neq \emptyset
 \end{equation}
and for every dyadic $\ell$ with $2 r_{*,L}\le \ell\le L$,
 \begin{equation}\label{eq:concentrationrstarper}
  \frac{1}{\ell^4}\Wper\lt(\mu_\ell, \frac{\mu(Q_\ell)}{\ell^2}\rt)\le \frac{1}{\ell^4}W_2^2\lt(\mu_\ell, \frac{\mu(Q_\ell)}{\ell^2}\rt)\le \frac{ \log\bra{\frac{\ell}{r_{*,L}}}}{\lt(\frac{\ell}{r_{*,L}}\rt)^2}.
 \end{equation}
\end{theorem}
\begin{proof}
 Let first { $\tilde{r}_{*,L}=\tilde{r}_{*,L}(\mu)$ be defined by  } 
 \[
  \tilde{r}_{*,L}:=\inf\lt\{  r \, : \, \frac{1}{\ell^4}W_2^2\lt(\mu_\ell, \frac{\mu(Q_\ell)}{\ell^2}\rt)\le \frac{ \log\bra{\frac{\ell}{r}}}{\lt(\frac{\ell}{r}\rt)^2}, \quad \textrm{for all dyadic $\ell$ with }   2r\le \ell\le L\rt\},
 \]
where we take the convention that $\tilde{r}_{*,L}=L/2$ if the set on the right-hand side is empty. We then { define $r_{*,L}=r_{*,L}(\mu)$ by}
\[
 r_{*,L}:=\begin{cases}
            \frac{L}{2}  \qquad  & \textrm{if } \frac{\mu(Q_L)}{L^2}\notin \lt[\frac{1}{2}, \, 2\rt]  \\
            \max(\tilde{r}_{*,L}, \min_{\spt \mu} |y|) \qquad  & \textrm{otherwise}.
          \end{cases}
\]
Since $\PP_L\sqa{\frac{\mu(Q_L)}{L^2}\notin \lt[\frac{1}{2}, \, 2\rt]}\le \exp(-c_1L^2)$ for some universal $c_1>0$ (cf.\ \eqref{chernoff}) and  for $r\le L/2$, 
$\PP_L( \spt \mu\cap B_{r}=\emptyset)=\exp(- r^2)$, it is enough to prove that $\tilde{r}_{*,L}$ satisfies $\sup_L \EE_L\sqa{\exp\lt(\frac{c_2 \tilde{r}_{*,L}^2}{\log (2\tilde{r}_{*,L})}\rt)}<\infty$ for some $c_2>0$. Now this follows directly
from \eqref{eq:concentrationrstar} and the fact that  for general potentially non-periodic $\mu\in \Gamma$, $\tilde{r}_{*,L}(\mu^{\per}_L)\le r_*(\mu)$. 
The first inequality in \eqref{eq:concentrationrstarper} follows from the fact that $\Wper\le W_2^2$. 
\end{proof}

\begin{remark}
 To avoid confusion between periodic and Euclidean objects, we would like to stress a few things which will be important in Section \ref{sec:application}. While the map $T_{\mu}$ is defined through the periodic optimal transport problem, estimate \eqref{eq:concentrationrstarper} gives a bound on the  Euclidean Wasserstein distance between the
 restrictions $\mu_\ell$ and the corresponding multiples of the Lebesgue measure on $Q_\ell$. The reason for the two different transport problems is that  on the one hand we want
 to work with a map which has good stationarity properties and on the other hand, for the iteration argument below,  it is more natural to have bounds on the Euclidean Wasserstein distances between $\mu_\ell$ and $\frac{\mu(Q_\ell)}{\ell^2}\chi_{Q_\ell}$.\\
Presently, conditions \eqref{eq:goodboundmu} come out of the blue but they will be very useful in Section \ref{sec:application}. Similarly, the first inequality in \eqref{eq:concentrationrstarper} will allow us to initialize the iteration argument in Theorem \ref{theo:mainresult intro}.
\end{remark}

\section{The main regularity result}\label{sec:main}
In this section we prove our main regularity result which states that  for every dimension $d\ge 2$, given an optimal transport map $T$ between a bounded set $\Omega$ and a measure $\mu$, if at some scale $R>0$ both the excess energy
\begin{equation}\label{def:E}
 E(\mu,T,R):=\frac{1}{R^{d+2}}\int_{B_{2R}} |T-x|^2
\end{equation}
and the local squared Wasserstein distance of $\mu$  in $O\supset B_{2R}$  to a constant 
\begin{equation}\label{def:D}
 D(\mu,O,R):=\frac{1}{R^{d+2}} W_2^2\lt(\mu\restr O, \frac{\mu(O)}{|O|}\chi_O\rt)
\end{equation}
are small, then on $B_R$, $T$ is quantitatively close (in terms of  $E$ and $D$) to  an harmonic gradient field.
This is similar to \cite[Proposition 4.6]{GO} with the major difference that here the measure $\mu$ is arbitrary and can be in particular singular. Let us point out that
we allow for $O\neq B_{2R}$ only because of the application we have in mind to the optimal matching problem where
we have good control on cubes instead of balls  (see Theorem \ref{theo:estimDrstar} and  the proof of Theorem \ref{theo:mainresult intro}). \\
By scaling we will mostly
work here with $R=1$ and will use the notation $E$ for $E(\mu,T,1)$ and similarly for $D$.  The global strategy is similar
to the one used in \cite{GO} and goes through an estimate at the Eulerian level \eqref{BBwass}. However, as opposed  to \cite{GO}, if $(\rho,j)$
is the minimizer of \eqref{BBwass} it does not satisfy a global $L^\infty$ bound (see Lemma \ref{lem:Linftyboundslice}). We will thus need to introduce a terminal layer. 
In \cite{AmStTr16}, regularization by the heat flow is used as an alternative approach to tackle  this issue.  \\
For $\tau>0$, let 
\[
 \Brho:=\int_0^{1-\tau} \rho_t \, dt.
\]
By \eqref{Linfty}, we have
\begin{equation}\label{Linftyrhobar}
  \Brho\les \tau^{-(d-1)} \qquad \textrm{in }  \ B_{\frac{3}{2}}.
\end{equation}
As in \cite{GO}, we would like to use the flux of $j$ as boundary data for the solution of the Poisson equation we will consider. This requires choosing a good radius  $R$ for which $j$
satisfies good estimates on $\partial B_R$. In our setting, this is much more complex than in \cite{GO} and is the purpose of the next section.
Let us point out that since the estimates we want to use are on the $L^2$ scale, we would need that the flux of $j$ through $\partial B_R$ is well controlled in $L^2$. Since this
is in general not the case, we will also need to replace this flux by a more regular one (see Lemma \ref{goodR} below).   

\subsection{Choice of a good radius}\label{sec:goodR}
 
Let  $X(x,t)=T_t(x)$ be the solution of the time-dependent version of optimal transport \eqref{dynwass}. Let us recall that the corresponding trajectories $t\to T_t(x)$ are straight segments and that we  
often drop the dependence in $x$ when it is not necessary to specify it. For $R>0$, and a given trajectory $X$ passing through $B_R$ we let (see Figure \ref{fig:fbar})
\begin{figure}\begin{center}
 \resizebox{9.5cm}{!}{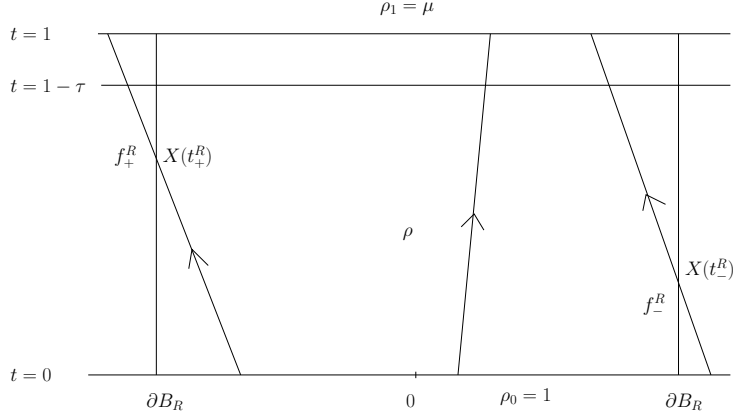} 
   \caption{The definition of $f_\pm^R$. } \label{fig:fbar}
 \end{center}
 \end{figure}
\[
 t_-^R:=\min \{t\in[0,1] \ : \ X(t)\in \overline{B}_R\} \qquad \textrm{and} \qquad t^R_+:=\max \{t\in[0,1] \ : \ X(t)\in \overline{B}_R\}
\]
be the entrance  and exit times. If $X(t)$ does not intersect $\overline{B}_R$, we set $t_-^R=1$ and $t_+^R=0$. Notice that for $t_-^R<1$ we have $X(t_-^R)\in \partial B_R$ and likewise $t_+^R>0$ implies $X(t_+^R)\in \partial B_R$. We now define the flux $f^R$ of $j$
through $\partial B_R$ (formally $f^R=j\cdot \nu^{B_R}$, where $\nu^{B_R}$
denotes the outward normal to $B_R$) by its action on functions $\zeta\in C^0_c(\R^d\times[0,1])$ as
\begin{equation}\label{def:fR}
  \int_{\R^d}\int_0^1 \zeta d f^R:=\int_{\Omega} \chi_{0\le t_-^R<t^R_+<1}(X) \zeta(X(t^R_+),t^R_+)-\int_{\Omega} \chi_{0 <t_-^R<t_+^R\le 1}(X) \zeta(X(t_-^R),t_-^R).
\end{equation}
Note that the measure $f^R$ is supported on $\partial B_R\times [0,1]$. The integration in \eqref{def:fR} is with respect to the Lebesgue measure $dx$ (the integrand depends on $x$ since $X$ and $t_\pm^R$ depend on $x$).
Our first lemma states that $f^R$ really acts like boundary values for $(\rho,j)$.
\begin{lemma}
 Let $(\rho,j)$ be a minimizer of \eqref{BBwass} and let $f^R$ be defined by \eqref{def:fR}. Then, for every $\zeta\in C^1_c(\R^d\times [0,1])$,
\begin{equation}\label{conteqloc}
 \int_{B_R}\int_{0}^1 \partial_t \zeta \rho +\nabla \zeta \cdot j=\int_{B_R} \zeta_1 d\mu -\int_{B_R}\zeta_0 +\int_{\R^d}\int_0^1  \zeta df^R.
\end{equation}
As a consequence, $(\rho,j)$ is a local minimizer of \eqref{BBwass} in the sense that for every $(\tilde \rho, \tilde j)$ with $\spt (\tilde \rho, \tilde j)\subset \overline{B}_R\times[0,1]$ and  satisfying \eqref{conteqloc}, 
\begin{equation}\label{eq:locmin}
 \int_{B_R}\int_0^1  \frac{1}{\rho}|j|^2\le \int_{\overline{B}_R}\int_0^1 \frac{1}{\tilde \rho} |\tilde j|^2.
\end{equation}

\end{lemma}
 \begin{proof}
 Once \eqref{conteqloc} is established, local minimality of $(\rho,j)$ follows from the fact that $(\hat{\rho},\hat{j})$ defined as $(\tilde \rho,\tilde j)$ in $\overline{B}_R\times[0,1]$ and
 $(\rho,j)$ outside, is admissible for \eqref{BBwass}. We thus only need to prove that \eqref{conteqloc} holds.
 By \eqref{def:fR}, for every $\zeta\in  C^1_c(\R^d\times [0,1])$,
 \begin{align*}
   \int_{\R^d}\int_0^1  \zeta df^R&=\int_{\Omega} \chi_{0\le t_-^R<t^R_+<1}(X) \zeta(X(t^R_+),t^R_+)-\int_{\Omega} \chi_{0<t_-^R< t_+^R\le 1}(X) \zeta(X(t_-^R),t_-^R)\\
    &=\int_{\Omega} \lt[\int_{t_-^R}^{t_+^R} \frac{d}{dt} \zeta(X,t)\rt] +\chi_{ t_-^R=0}(X) \zeta(X(t_-^R),t_-^R)-\chi_{ t_+^R=1}(X) \zeta(X(t_+^R),t_+^R)\\
  &= \int_{\Omega} \int_0^1\chi_{B_R}(X)\lt[ \partial_t \zeta(X,t) +\nabla \zeta(X,t)\cdot \dot{X}\rt]\\
   &\qquad +\int_{\Omega}\chi_{B_R}(X(0))\zeta(X(0),0)-\chi_{B_R}(X(1))\zeta(X(1),1).
  \end{align*}
Since $\rho_t= T_t\# \chi_{\Omega}$ and $j_t= T_t\# (T-x)\chi_{\Omega}$, and since $X(x,t)= T_t(x)=(1-t)x+tT(x)$, for every $\zeta\in C^1_c(\R^d\times[0,1])$,
  \[
   \int_{B_R}\int_0^1 \partial_t \zeta d\rho= \int_{\Omega}\int_0^1 \chi_{B_R}(X)\partial_t\zeta(X,t)
  \] and 
  \[
    \int_{B_R}\int_0^1 \nabla \zeta \cdot dj= \int_{\Omega}\int_0^1 \chi_{B_R}(X)\nabla \zeta(X,t)\cdot \dot{X}.
  \]
This together with $X(0)=Id$ and $X(1)=T$ concludes the proof.
 \end{proof}
 We now prove that $f^R$ satisfies a bound analog to \eqref{Linfty}.
 \begin{lemma}
  The measure $f^R$ is absolutely continuous with respect to the measure $\H^{d-1}\restr \partial B_R\otimes dt $ and for $t\in(0,1)$, 
  \begin{equation}\label{Linftyf}
   \sup_{\partial B_R} |f^R_t|\les E^{\frac{1}{d+2}}\frac{1}{(1-t)^d}. 
  \end{equation}
 \end{lemma}
\begin{proof}
 Let $\zeta\in C^1_c( \R^d\times (0,1))$ be fixed and for $0<r\ll R$, let $\eta_r$ be a smooth radial function such that  $\eta_r(x)=0$ 
 if $|x|\le R-r$, $\eta_r(x)=1$ for $|x|\ge R$ and $\sup |\nabla \eta_r|\les r^{-1}$. Testing \eqref{conteqloc} with $\zeta \eta_r$, we obtain since $f^R$ is supported  on $ \partial B_R\times [0,1]$,
 \begin{align*}
  \lt|\int_{\R^d}\int_0^1  \zeta df^R\rt|&\les \int_{B_R\backslash B_{R-r}}\int_0^1 |\partial_t \zeta| \rho + |\nabla \zeta| |j| + r^{-1}|\zeta| |j| \\
  &\stackrel{\eqref{Linfty}}{\les} \int_{B_R\backslash B_{R-r}}\int_0^1 (|\partial_t \zeta|+ E^{\frac{1}{d+2}}|\nabla \zeta|) \frac{1}{(1-t)^d} + r^{-1} E^{\frac{1}{d+2}} \frac{| \zeta|}{(1-t)^d}.
 \end{align*}
 Letting $r\to 0$, we obtain 
 \[
  \lt|\int_{\R^d}\int_0^1  \zeta df^R\rt|\les \int_{\partial B_R}\int_0^1 E^{\frac{1}{d+2}} \frac{| \zeta|}{(1-t)^d},
 \]
from which the claim follows.
\end{proof}

We define the outgoing and incoming fluxes as  (see Figure \ref{fig:fbar})
 \begin{multline}\label{def:fpm}
  \int_{\R^d}\int_0^1 \zeta d f_+^R:=\int_{\Omega} \chi_{0\le t_-^R<t^R_+<1}(X) \zeta(X(t^R_+),t^R_+) \\
  \textrm{and }  \int_{\R^d}\int_0^1 \zeta d f_-^R:=\int_{\Omega} \chi_{0<t_-^R< t_+^R\le 1}(X) \zeta(X(t_-^R),t_-^R),
 \end{multline}
so that $f^R=f^R_+-f^R_-$. Now for a given layer size $0<\tau<1$, we define the cumulated fluxes
 \begin{multline}\label{def:fbarpm}
 \int_{\R^d} \zeta d\Bf^R_{+}:=\int_{\Omega} \chi_{0\le t^R_-<t^R_{+}<1-\tau}(X)\,  \zeta(X(t^R_+)) \\
  \textrm{and }  \int_{\R^d} \zeta d\Bf^R_{-}:=\int_{\Omega} \chi_{t_-^R<t_+^R}(X)\chi_{0<t^R_{-}<1-\tau}(X)\,  \zeta(X(t^R_-)) ,
 \end{multline}
and then let 
\begin{equation}\label{defBfR}
 \Bf^R(x):=\int_0^{1-\tau} f^R(x,t) dt \qquad \textrm{so that } \qquad \Bf^R=\Bf^R_+-\Bf^R_-.
\end{equation}
\begin{lemma}\label{lem:linftybarf}
 There holds $\Bf^R_+ \perp \Bf^R_-$ and
 \begin{equation}\label{LinftyBarf}
 \sup_{\partial B_R} \Bf^R_{\pm}\les E^{\frac{1}{d+2}} \tau^{-(d-1)}.
\end{equation}
\end{lemma}
\begin{proof}
 Let us prove that $f^R_+ \perp f^R_-$. Estimate \eqref{LinftyBarf} will then follow from \eqref{Linftyf}.  

To this end we consider the space-time points on the cylinder through which particles exit 
\[
 A:=\{ (y,t)\in \partial B_R\times(0,1) \ : \ \exists x\in \Omega \ \textrm{such that } t_-^R<t_+^R \textrm{ and } (X(t_+^R),t_+^R)=(y,t)\}
\]
and the original positions of particles that enter at the same space-time where another particle exits
\[
 B:=\lt\{x\in \Omega \ :  0<t_-^R< t_+^R, \,   \exists \tilde{x}\in \Omega \textrm{ with }  \tilde{t}_-^R<\tilde{t}_+^R=t_-^R \textrm{ and } X(x,t_-^R)=X(\tilde{x},\tilde{t}_+^R)\rt\}.
\]
We claim that $B=\emptyset$.  

Recall (cf.\ Remark \ref{rem:selection}) that $T$ is given as a measurable selection of the subgradient of a convex function.
In particular, we have $(T(x)-T(\tilde x))\cdot(x-\tilde x)\ge 0$ for all $x,\tilde x$. Hence, if $x$ is such that $0<t_-^R< t_+^R$ and 
$\tilde x$ such that $\tilde{t}_-^R<\tilde{t}_+^R=t_-^R$ we cannot have $X(x,t_-^R)=T_{t_-^R}(x)=T_{t_{-}^R}(\tilde x)=X(\tilde x, \tilde t_+^R)$ and thus $B=\emptyset$.\\
Now since 
\[
 f^R_+(A^c)\stackrel{\eqref{def:fpm}}{=}\int_{\Omega} \chi_{0\le t_-^R<t^R_+<1}(X) \chi_{A^c}(X(t^R_+),t^R_+),
\]
we  have $f^R_+(A^c)=0$ by definition of $A$. Similarly,
\[
 f^R_-(A)\stackrel{\eqref{def:fpm}}{=} \int_{\Omega} \chi_{0<t_-^R< t_+^R\le 1}(X) \chi_{A}(X(t_-^R),t_-^R)=|B|=0,   
\]
which proves that $f^R_+ \perp f^R_-$.
\end{proof}

We will also need the analog of $\Bf^R_{+}$ for the outgoing flux in $(1-\tau,1)$. Let $\Bf^{R,\lay}_{+}$ be the measures defined for $\zeta\in C^{0}_c(\R^d)$, by
 \begin{equation}\label{defBflay+}
 \int_{\R^d} \zeta d\Bf^{R,\lay}_{+}:=\int_{\Omega} \chi_{t^R_-<1-\tau<t_+^R<1}(X)\,  \zeta(X(t^R_+)).
 \end{equation}
Notice that we   consider only  the particles which leave  $B_R$ after $t=1-\tau$ but were already inside $B_R$ at time $1-\tau$. For later use, let also
 \begin{equation}\label{defflay+}
 \int_{\R^d}\int_0^1 \zeta df^{R,\lay}_{+}:=\int_{\Omega} \chi_{t^R_-<1-\tau<t_+^R<1}(X)\,  \zeta(X(t^R_+),t^R_+).
 \end{equation}
We can now show that there exists a good radius $R$.
\begin{lemma}\label{goodR}
Assume that $E+D\ll1$, then there exists $R\in(\frac{1}{2},\frac{3}{2})$ such that
 \begin{equation}\label{L2boundf}
  \int_{\partial B_R} \int_0^{1-\tau} (f^R)^2\les \tau^{-d} E
 \end{equation}
and there exist densities $\Brho_\pm^R$ and $\Brhop^{R,\lay}$ on $\partial B_R$ such that 
 \begin{equation}\label{L2boundrho}
  \int_{\partial B_R} (\Brho_\pm^R)^2\les E \qquad \textrm{and} \qquad W_{\partial B_1}^2(\Brho^R_\pm, \Bf^R_{\pm})\les E^{ \frac{d+3}{d+2}},
 \end{equation}
 and 
  \begin{equation}\label{L2boundrholay}
  \int_{\partial B_R} (\Brhop^{R,\lay})^2\les \tau^2 E+D \qquad \textrm{and} \qquad W_{\partial B_1}^2(\Brhop^{R,\lay}, \Bf_+^{R,\lay})\les \tau^3 E^{\frac{d+3}{d+2}} +\tau E^{\frac{1}{d+2}}D.
 \end{equation}
Moreover,
\begin{equation}\label{Linftyboundrho}
 \sup_{\partial B_R}\, \Brho_\pm^R\, \les E^{\frac{1}{d+2}}.
\end{equation}

\end{lemma}
\begin{proof}
Let us start by \eqref{L2boundf}. For this, given $\zeta\in C^1_c( \R^d\times (0,1-\tau))$, integrating \eqref{conteqloc} in $R\in(\frac{1}{2},\frac{3}{2})$, we obtain
\[
 \int_{\frac{1}{2}}^{\frac{3}{2}} \int_{\R^d}\int_0^{1-\tau} \zeta df^R= \int_{\frac{1}{2}}^{\frac{3}{2}} \int_{B_R}\int_0^{1-\tau} \partial_t \zeta \rho+ \nabla \zeta \cdot j.
\]
Letting $\omega(x):=\int_{\frac{1}{2}}^{\frac{3}{2}} \chi_{B_R}(x) dR$ and using Fubini, we obtain 
\begin{align*}
 \int_{\frac{1}{2}}^{\frac{3}{2}} \int_{\R^d}\int_0^{1-\tau} \zeta df^R&= \int_{\R^d}\int_0^{1-\tau} \omega \lt( \partial_t \zeta \rho+ \nabla \zeta \cdot j\rt) \\
 &= \int_{\R^d}\int_0^{1-\tau} \zeta \nabla \omega \cdot j, 
\end{align*}
where in the second line we used the fact that $(\rho,j)$ satisfies the continuity equation on $\R^d\times (0,1)$. By the Cauchy-Schwarz inequality together with the estimate
on $\rho$ given by \eqref{Linfty} and by \eqref{Linftyboundslice}, we thus obtain 
\begin{align*}
\lt|\int_{\frac{1}{2}}^{\frac{3}{2}} \int_{\R^d}\int_0^{1-\tau} \zeta df^R\rt|&\les \lt(\int_{B_{\frac{3}{2}}}\int_0^{1-\tau} \rho \zeta^2\rt)^{\frac12} \lt( \int_{B_{\frac{3}{2}}} \int_0^{1-\tau}\frac{1}{\rho}|j|^2\rt)^{\frac12}\\
&\les \tau^{-\frac{d}{2}} E^{\frac12} \lt(\int_{B_{\frac{3}{2}}}\int_0^{1-\tau} \zeta^2\rt)^{\frac12},
\end{align*}
from which we obtain by duality
\begin{equation}\label{L2boundfR}
 \int_{\frac{1}{2}}^{\frac{3}{2}}  \int_{\partial B_R} \int_0^{1-\tau}(f^R)^2\les \tau^{-d}E.
\end{equation}
 We now turn to \eqref{L2boundrho} and \eqref{Linftyboundrho}. Notice that by \eqref{W2compare}, it is enough to prove \eqref{L2boundrho} for $W_2^2$ instead of $W_{\partial B_1}^2$.
 Let $\Brho_{\pm}^R$ be the measures supported on $\partial B_R$ such that for $\zeta\in C^{0}_c(\R^d)$ (see Figure \ref{fig:rhobar})
 \begin{figure}\begin{center}
 \resizebox{6.cm}{!}{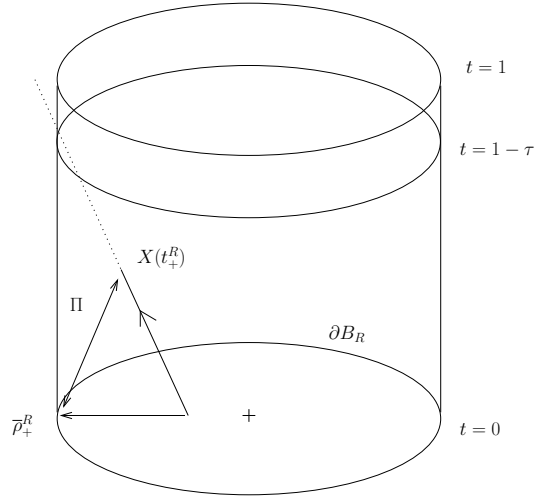} 
   \caption{The definition of $\Brho_+^R$.} \label{fig:rhobar}
 \end{center}
 \end{figure}
\begin{multline}\label{defrhopm}
 \int_{\partial B_R} \zeta d\Brho_{+}^R:= \int_{\Omega}  \chi_{t_-^R<t^R_{+}<1-\tau}(X) \zeta\lt(R \frac{X(0)}{|X(0)|}\rt) 
 \\ \textrm{and}\quad  \int_{\partial B_R} \zeta d\Brhom^R:= \int_{\Omega}   \chi_{t_-^R<t^R_+}(X)\chi_{0<t^R_{-}<1-\tau}(X) \zeta\lt(R \frac{X(0)}{|X(0)|}\rt).
\end{multline}
 Since the proofs are almost identical for $\Brhom^R$, we focus for brevity on $\Brhop^R$. We start with the $L^2$ bound. We introduce the measure $\rho^R_+$ of all original particles that spend time in $B_R$ but exit before $1-\tau$, that is
 \begin{equation}\label{defrho+}
  \int_{\R^d} \zeta d\rho^R_+:=\int_{\Omega} \chi_{t_-^R<t^R_+<1-\tau}(X) \zeta\lt( X(0)\rt).
 \end{equation}
Let us point out that  $\rho^R_+\le \chi_\Omega$  and note that $\Brhop^R$ is nothing else than the push-forward of $\Brho_+^R$ under the map $x\to R \frac{x}{|x|}$ so that on the level of densities, we have for $x\in \partial B_R$ 
\[
 \Brhop^R(x)=\int_0^{+\infty}\lt(\frac{r}{R}\rt)^{d-1} \rho^R_+\lt(r \frac{x}{|x|}\rt) dr.
\]
Notice that because of the $L^\infty$ bound  \eqref{LinftyboundT} on $T-x$, the integral above can be restricted to $(R-C E^{\frac{1}{d+2}},R+C E^{\frac{1}{d+2}})\subset (R/2, 4R/3)$. This directly implies \eqref{Linftyboundrho}. Arguing as for \eqref{claim:overz}, we obtain 
\[
 \int_{\partial B_R} (\Brhop^R)^2\les \int_{\R^d} |R-|x|| d\rho^R_+,
\]
so that  we are just left to prove that 
\begin{equation}\label{toproveL2}
 \int_0^{\frac{3}{2}} \int_{\R^d} |R-|x|| d\rho^R_+ \les E.
\end{equation}
Note that since $X$ are straight lines, $\rho^R_+$ a.s.\ we have $|R-|X(0)||\le|X(1)-X(0)|$ so that
\begin{align*}
 \int_0^{\frac{3}{2}} \int_{\R^d} |R-|x|| d\rho^R_+ &\stackrel{\eqref{defrho+}}{=}\int_0^{\frac{3}{2}}\int_{\Omega}  \chi_{t_-^R<t^R_+<1-\tau}(X) |R-|X(0)|| \\
 &\le \int_0^{\frac{3}{2}}\int_{\Omega}  \chi_{t_-^R<t^R_+<1-\tau}(X) |X(1)-X(0)|\\
 &=\int_{\Omega}|X(1)-X(0)|\int_0^{\frac{3}{2}} \chi_{t_-^R<t^R_+<1-\tau}(X)\\
 &\le \int_{\Omega} \chi_{|X(0)|< 2} |X(1)-X(0)|^2\stackrel{\eqref{def:E}}{\les} E,
\end{align*}
where we used again the $L^\infty$ bound \eqref{LinftyboundT} for $T-x$ and the fact that for every $x$ such that $X(0)\neq X(1)$ and every $t_1,t_2\in[0,1]$, 
\begin{equation}\label{estimH1}
 \H^1(R \ : \ \exists t\in [t_1,t_2] \textrm{ with } X(t)\in \partial B_R)\le |X(t_1)-X(t_2)|,
\end{equation}
to obtain that for given $x$ such that $X(1)\neq X(0)$,
\begin{multline*}
 \int_0^{\frac{3}{2}}  \chi_{t_-^R<t^R_+<1-\tau}(X)\le  \chi_{|X(0)|<2}  \H^1(R \ : \ \exists t\in(0,1-\tau) \textrm{ with } X(t)\in \partial B_R)\\
 \le    \chi_{|X(0)|<2} |X(0)-X(1)|
\end{multline*}
This shows \eqref{toproveL2} and thus
\begin{equation}\label{L2boundrhoR}
 \int_{\frac{1}{2}}^{\frac{3}{2}}\int_{\partial B_R} (\Brhop^R)^2\les E.
\end{equation}
Let us now turn to the $W_2^2$ estimate in \eqref{L2boundrho}. We consider the coupling $\Pi$ between $\Bf^R_+$ (recall \eqref{def:fbarpm}) and $\Brhop^R$ defined for $\zeta\in C^0_c(\R^d\times \R^d)$ by  (see Figure \ref{fig:rhobar})
\[
 \int_{\R^d\times \R^d} \zeta d\Pi:=\int_{\Omega} \chi_{t_-^R<t_+^R<1-\tau}(X) \zeta\lt(X(t^R_+), R \frac{X(0)}{|X(0)|}\rt).
\]
Using that for $|X|\ges R$ (which $\Pi$ a.e. is the case by the $L^\infty$ bound on $T-x$) the radial projection on $\partial B_R$ is Lipschitz continuous and thus $|X(t^R_+)-  R \frac{X(0)}{|X(0)|}|\les |X(t^R_+)-X(0)|$, we get
\begin{align*}
 W_2^2(\Brhop^R,\Bf^R_+)&\le \int_{\Omega} \chi_{t^+_R<1-\tau}(X) \lt|X(t^R_+)-  R \frac{X(0)}{|X(0)|}\rt|^2\\
 &\les\int_{\Omega} \chi_{t^+_R<1-\tau}(X)|X(t^R_+)-X(0)|^2\\
 &\le \int_{\Omega} \chi_{t^+_R<1-\tau}(X)|X(1)-X(0)|^2,
\end{align*}
where in the last step  we used once again that $X$ is a straight line. Integrating in $R$ and arguing as above, we obtain
\begin{align}
 \int_{0}^{\frac{3}{2}} W_2^2(\Brhop^R,\Bf^R_+)&\les\int_{\Omega} \chi_{|X(0)|< {\frac{7}{4}}} |X(1)-X(0)|^3\nonumber\\
 &\les \sup_{ B_{\frac{7}{4}}} |T-x| \int_{\Omega}  \chi_{|X(0)|<2} |X(1)-X(0)|^2\nonumber\\
 &\les E^{ \frac{d+3}{d+2}},\label{L2boundrhoRbis}
\end{align}
where in the last step we have used once more Lemma \ref{lem:Linftybound} and definition \eqref{def:E}.\\
We finally turn to \eqref{L2boundrholay}, the  proof of which  is similar to the one of \eqref{L2boundrho}.   As above,  thanks to  \eqref{W2compare}, it is enough to prove \eqref{L2boundrholay} for $W_2^2$ instead of $W_{\partial B_1}^2$. The definition of  $\Brho_+^{R,\lay}$ is a little bit more complex than the one of $\Brho_+^R$,
since we need to couple the trajectories $X$ to the ones given by the optimal coupling for
$D=W_2^2(\mu\restr O,\frac{\mu(O)}{|O|}\chi_{O})$. Consider first the coupling $\Pi_{12}(x,y):= \chi_{O}(y) (Id\times T)\# \chi_{\Omega}(x,y)$ so that for $\zeta\in C^0_c(\Omega\times O)$
\[
 \int_{\Omega\times O} \zeta d\Pi_{12}=\int_{\Omega} \chi_{O}(T(x))\zeta(x,T(x)).
\]
In particular, the second marginal of $\Pi_{12}$ is equal to $\mu\restr O$. Let then $\Pi_{23}$ be the optimal coupling between $\mu\restr O$ and $\frac{\mu(O)}{|O|} \chi_{O}$. By the Gluing Lemma (see \cite[Lemma 7.6]{Viltop}), there exists a measure $\Pi$ on $\Omega\times O\times O$ with marginals $\Pi_{12}$ on $\Omega\times O$ and $\Pi_{23}$ on $O\times O$.
We now define $\Brho_+^{R,\lay}$ in analogy to \eqref{defrhopm} by (see Figure \ref{fig:rhobarlay})
 \begin{figure}\begin{center}
 \resizebox{13.cm}{!}{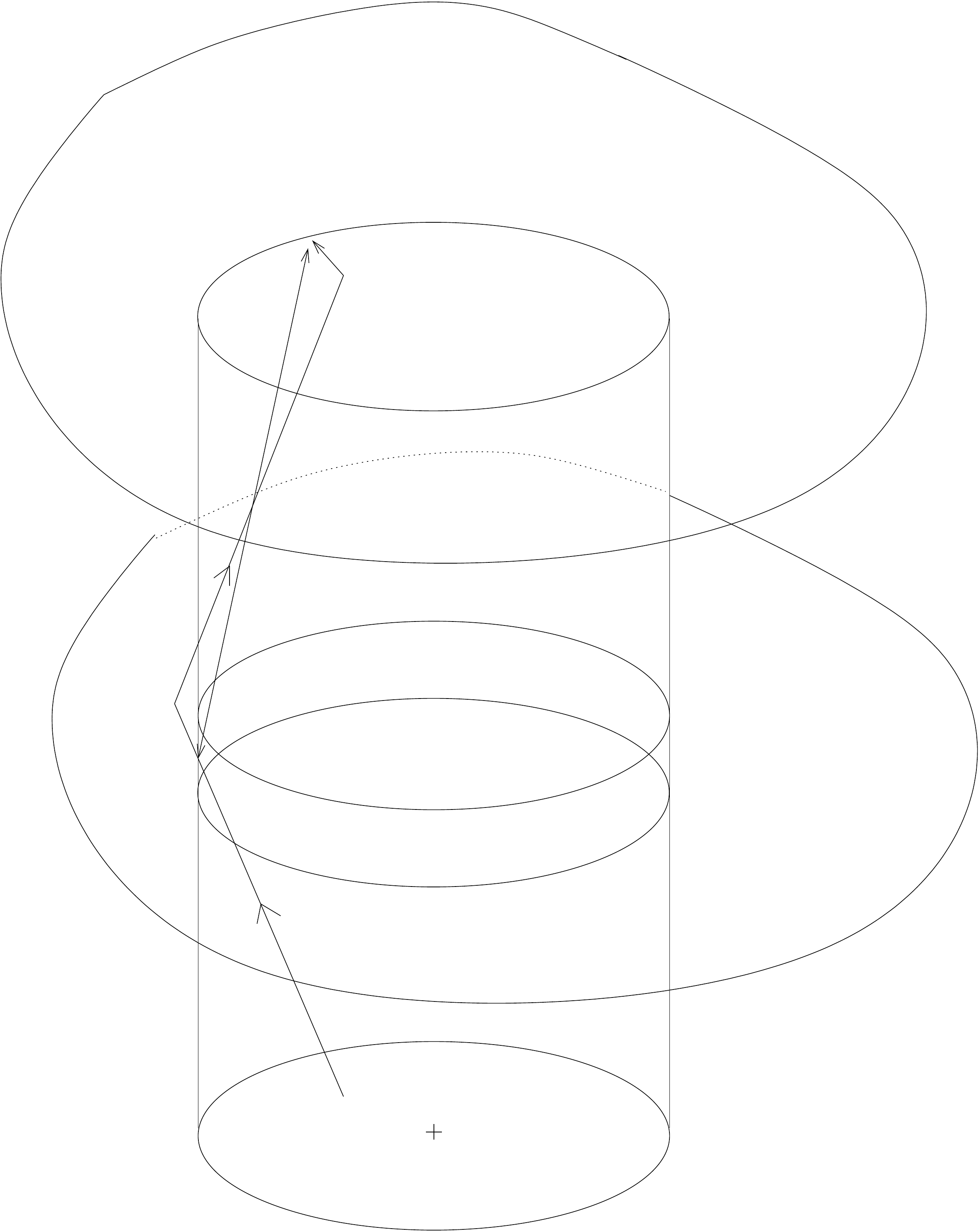} 
   \caption{The definition of $\Brho_+^{R,\lay}$.} \label{fig:rhobarlay}
 \end{center}
 \end{figure}
\[
  \int_{\R^d} \zeta d\Brhop^{R,\lay}:= \int_{\Omega\times O\times O}   \chi_{t^R_-<1-\tau<t^R_+<1}(X) \zeta\lt(R \frac{z}{|z|}\rt) d\Pi(x,y,z).
\]
We also define the unprojected density
\[
  \int_{\R^d} \zeta d\rho_+^{R,\lay}:= \int_{\Omega\times O\times O}   \chi_{t^R_-<1-\tau<t^R_+<1}(X) \zeta\lt( z\rt) d\Pi(x,y,z).
\]
 By the  argument used for \eqref{claim:overz}, 
we have
\begin{align*}
\int_0^{\frac{3}{2}} \int_{\partial B_R} \lt(\Brho_+^{R,\lay}\rt)^2&\les \int_0^{\frac{3}{2}} \int_{\R^d} |R-|x|| d\rho^{R,\lay}_+\,\\
&= \int_0^{\frac{3}{2}}\int_{\Omega\times O\times O}  \chi_{t^R_-<1-\tau<t^R_+<1}(X)  |R-|z|| d\Pi \\
 &\le \int_0^{\frac{3}{2}}\int_{\Omega\times O\times O}\chi_{t^R_-<1-\tau<t^R_+<1}(X)  \lt(|X(t^R_+)-y|+|y-z|\rt) d\Pi. 
 \end{align*}
 By definition of $\Pi$ and since the trajectories of $X$ are straight lines,  we have 
 \begin{multline*}
   \int_{\Omega\times O\times O}\chi_{t^R_-<1-\tau<t^R_+<1}(X)  |X(t^R_+)-y| d\Pi=\int_{\Omega\times O\times O}\chi_{t^R_-<1-\tau<t^R_+<1}(X)  |X(t^R_+)-X(1)| d\Pi\\
   \le  \int_{\Omega} \chi_{t^R_-<1-\tau<t^R_+<1}(X)  |X(1-\tau)-X(1)|,
 \end{multline*}
which  by \eqref{estimH1} and \eqref{LinftyboundT} leads to
 \[
  \int_0^{\frac{3}{2}}\int_{\Omega\times O\times O}\chi_{t^R_-<1-\tau<t^R_+<1}(X)  |X(t^R_+)-y| d\Pi\le \int_{\Omega} \chi_{|X(0)|<2} |X(1-\tau)-X(1)|^2.
 \]
Since  \eqref{estimH1} and \eqref{LinftyboundT} also yield
\begin{align*}
 \int_0^{\frac{3}{2}}\int_{\Omega\times O\times O}\chi_{t^R_-<1-\tau<t^R_+<1}(X)|y-z| d\Pi &= \int_{\Omega\times O\times O}|y-z| \lt[ \int_0^{\frac{3}{2}} \chi_{t^R_-<1-\tau<t^R_+<1}(X)\rt] d\Pi\\
 &\le \int_{\Omega\times O\times O}|y-z| \chi_{|X(0)|<2} |X(1-\tau)-X(1)| d\Pi\\
 &\stackrel{\textrm{Young}}{\les} \int_{\Omega} \chi_{|X(0)|<2}|X(1-\tau)-X(1)|^2\\
 &\qquad +\int_{\Omega\times O\times O}|y-z|^2 d\Pi\\
 &= \int_{\Omega} \chi_{|X(0)|<2}|X(1-\tau)-X(1)|^2\\
 &\qquad +\int_{O\times O}|y-z|^2 d\Pi_{23},
\end{align*}
 we have using that $|X(t)-X(1)|=(1-t)|X(0)-X(1)|$,
\begin{align}
\int_0^{\frac{3}{2}} \int_{\partial B_R} \lt(\Brho_+^{R,\lay}\rt)^2 
 &\les \tau^2  \int_{\Omega} \chi_{|X(0)|<2}|X(0)-X(1)|^2 +\int_{O\times O} |y-z|^2 d\Pi_{23}\nonumber\\
 &\le \tau^2E +D. \label{L2boundrholayR}
\end{align}
 In order to obtain the second estimate in \eqref{L2boundrholay}, we consider the coupling $\widehat{\Pi}$ between $\Bf_+^{R,\lay}$ (recall \eqref{defBflay+}) and $\Brhop^{R,\lay}$ given by
 \[
  \int_{\R^d\times\R^d} \zeta d\widehat{\Pi}:= \int_{\Omega\times O\times O}   \chi_{t^R_-<1-\tau<t^R_+<1}(X) \zeta\lt( X(t_+^R),R \frac{z}{|z|}\rt) d\Pi.
 \]
 This is indeed a coupling between $\Bf_+^{R,\lay}$ and $\Brhop^{R,\lay}$ since for $x\in \Omega$ such that $t^R_-<1-\tau<t^R_+<1$, we have $X(1)\in B_2\subset O$ and therefore,
 \begin{multline*}
  \int_{\R^d\times\R^d} \zeta(x) d\widehat{\Pi}= \int_{\Omega\times O\times O}   \chi_{t^R_-<1-\tau<t^R_+<1}(X) \zeta\lt( X(t_+^R)\rt) d\Pi\\
  =\int_{\Omega} \chi_{t^R_-<1-\tau<t^R_+<1}(X) \chi_{O}(X(1))\zeta\lt( X(t_+^R)\rt)=\int_{\R^d} \zeta d \Bf_+^{R,\lay}.
 \end{multline*}
Arguing as for \eqref{L2boundrhoRbis} we have
\begin{align*}
 \int_0^{\frac{3}{2}} W_2^2(\Bf_+^{R,\lay},\Brhop^{R,\lay})&\le \int_0^{\frac{3}{2}} \int_{\Omega\times O\times O}   \chi_{t^R_-<1-\tau<t^R_+<1}(X) \lt| X(t_+^R)-R \frac{z}{|z|}\rt|^2 d\Pi\\
 &\le \int_0^{\frac{3}{2}} \int_{\Omega\times O\times O}   \chi_{t^R_-<1-\tau<t^R_+<1}(X) \lt| X(t_+^R)-z\rt|^2 d\Pi\\
 &\les \int_0^{\frac{3}{2}} \int_{\Omega} \chi_{t^R_-<1-\tau<t^R_+<1}(X) \lt| X(t_+^R)-X(1)\rt|^2\\
 &\qquad + \int_0^{\frac{3}{2}} \int_{\Omega\times O\times O}   \chi_{t^R_-<1-\tau<t^R_+<1}(X) \lt| y-z\rt|^2 d\Pi\\
 &\les \int_0^{\frac{3}{2}} \int_{\Omega} \chi_{t^R_-<1-\tau<t^R_+<1}(X) \lt| X(1-\tau)-X(1)\rt|^2\\
 &\qquad + \int_0^{\frac{3}{2}} \int_{\Omega\times O\times O}   \chi_{t^R_-<1-\tau<t^R_+<1}(X) \lt| y-z\rt|^2 d\Pi\\
 &\les \int_{\Omega} \chi_{|X(0)|<\frac{7}{4}}\lt|X(1-\tau)-X(1)\rt|^3\\
& \qquad + \int_{\Omega\times O\times O} \chi_{|X(0)|<\frac{7}{4}} \lt| y-z\rt|^2 |X(1-\tau)-X(1)|d\Pi\\
 &\les \tau^3  \sup_{B_{\frac{7}{4}}} |T-x| \int_{\Omega} \chi_{|X(0)|< 2} |X(1)-X(0)|^2 \\
 &\qquad + \tau \sup_{B_{\frac{7}{4}}}|T-x| \int_{ O\times O}  \lt| y-z\rt|^2 d\Pi_{23} \\
 &\les \tau^3 E^{\frac{d+3}{d+2}} +\tau E^{\frac{1}{d+2}} D. 
\end{align*}
Putting this together with \eqref{L2boundfR}, \eqref{L2boundrhoR}, \eqref{L2boundrhoRbis} and \eqref{L2boundrholayR}, we see that we may choose $R\in (\frac{1}{2},\frac{3}{2})$ such that \eqref{L2boundf}, \eqref{L2boundrho} and \eqref{L2boundrholay} hold.
\end{proof}
Let $\mu'_R$ be the part of $\mu\restr B_R$ coming from trajectories which were inside $B_R$ before the time $1-\tau$. That is, for $\zeta\in C_c^0(\R^d)$ 
\begin{equation}\label{defmu'}
 \int_{\R^d} \zeta d\mu'_R:=\int_{\Omega} \chi_{t_-\le 1-\tau}(X)\chi_{t_+=1}(X) \zeta(X(1)).
\end{equation}
We then have
\begin{lemma}\label{lem:distdata}
 Assume that $E+D\ll 1$, then there exists $R\in (\frac{1}{2},\frac{3}{2})$ such that the conclusions of Lemma \ref{goodR} hold and 
 \begin{equation}\label{distdata}
  W_2^2\lt(\rho_{1-\tau}\restr B_R, \frac{\rho_{1-\tau}(B_R)}{|B_R|} \chi_{B_R}\rt)+W_2^2\lt(\mu'_R, \frac{\mu'_R(B_R)}{|B_R|} \chi_{B_R}\rt)\les \tau^2 E +D.
 \end{equation}

\end{lemma}
\begin{proof}
 We are only going to show that 
 \begin{equation}\label{distdata1}
  \int_{\frac{1}{2}}^{\frac{3}{2}}W_2^2\lt(\rho_{1-\tau}\restr B_R, \frac{\rho_{1-\tau}(B_R)}{|B_R|} \chi_{B_R}\rt)\les \tau^2 E +D,
 \end{equation}
 since the estimate for $W_2^2\lt(\mu'_R, \frac{\mu'_R(B_R)}{|B_R|} \chi_{B_R}\rt)$ is similarly obtained.
For notational simplicity, { in this proof we will often drop the  $R$ dependence in our notation.}  Put $\Lambda:= \frac{\rho_{1-\tau}(B_R)}{|B_R|}$ and $\Gamma:=\frac{\mu(O)}{|O|}$. We will not distinguish between $\Lambda$ and the function $\Lambda\chi_{B_R}$ and similarly for $\Gamma$. 
 Since $E\ll1$,  \eqref{LinftyboundT} and \eqref{inclTt} imply that $\Lambda\sim \Gamma\sim 1$.
Let $X$ be the optimal trajectories for $W_2^2(\chi_{\Omega},\mu)$ and $\Pi_{23}$ be the optimal coupling for $D=W_2^2(\mu\restr O,\Gamma \chi_{O})$. Let us recall that  $\Pi_{12}$ is the measure defined on $\Omega\times O$ by
\[
 \int_{\Omega\times O} \zeta d\Pi_{12}:= \int_\Omega \chi_{O}(T(x)) \zeta(x,T(x))
\]
and let $\tilde{\mu}\le \mu$ be the measure defined by (see Figure \ref{fig:muprime})
 \begin{figure}\begin{center}
 \resizebox{10.cm}{!}{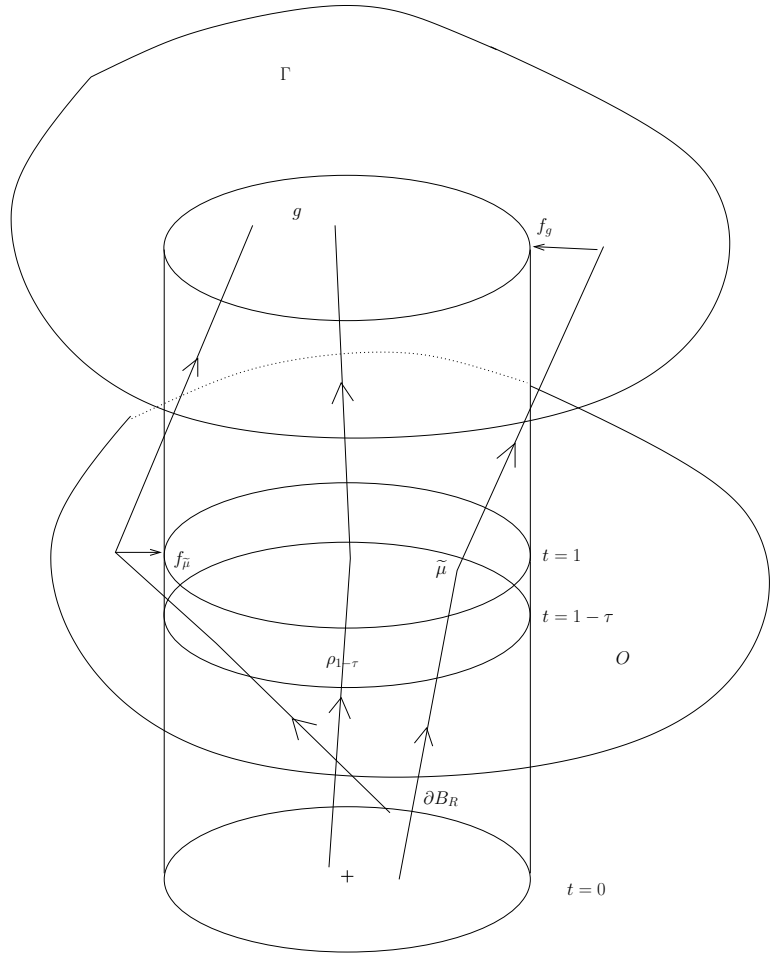} 
   \caption{The definition of $\tilde{\mu}$, $g$, $f_{\tilde{\mu}}$ and $f_{g}$.} \label{fig:muprime}
 \end{center}
 \end{figure}
\[
 \int_{\R^d} \zeta d\tilde{\mu}:=\int_{\Omega} \chi_{|X(1-\tau)|\le R}(X) \zeta(X(1))=\int_{\Omega\times O} \chi_{|X(1-\tau)|\le R}(X) \zeta(y) d\Pi_{12},
\]
or in words, $\tilde{\mu}$ is the part of $\mu$ which originates from $\rho_{1-\tau}\restr B_R$ along $X$. Notice that $\tilde{\mu}\LL B_R=\mu'_R$ (recall \eqref{defmu'}). Recall that if $\Pi$ is the coupling obtained by the Gluing Lemma applied to $\Pi_{12}$ and $\Pi_{23}$, 
we can then define $g\le \Gamma \chi_{O}$ by
\[
 \int_{\R^d} \zeta dg:= \int_{\Omega\times O\times O}  \chi_{|X(1-\tau)|\le R}(X)  \zeta(z) d\Pi(x,y,z),
\]
which is the part of $\Gamma \chi_{O}$ which originates from $\tilde{\mu}$ through $\Pi_{23}$. We then project the parts of $\tilde \mu$ and $g$ outside $B_R$ onto $\partial B_R$: 
\[
 \int_{\R^d} \zeta df_{\tilde{\mu}}:= \int_{\R^d} \chi_{B_R^{c}}(y)\zeta\lt(R \frac{y}{|y|}\rt)  d\tilde{\mu}= \int_{\Omega\times O} \chi_{|X(1-\tau)|\le R}(X) \chi_{B_R^c}(y) \, \zeta\lt(R \frac{y}{|y|}\rt) d\Pi_{12}
 \]
 and 
 \begin{multline*}
 \int_{\R^d} \zeta df_{g}:=\int_{O}  \chi_{B_R^c}(z)\, \zeta\lt(R\frac{z}{|z|}\rt) d g
 =\int_{\Omega\times O\times O}  \chi_{|X(1-\tau)|\le R}(X) \chi_{B_R^c}(z) \zeta\lt(R\frac{z}{|z|}\rt) d\Pi.
\end{multline*}
 Since $g\le \Gamma\les 1$ we can argue as for  \eqref{claim:overz} to obtain\footnote{notice that we cannot assert the same thing for $f_{\tilde{\mu}}$}
 \begin{align}
 \int_0^{\frac{3}{2}}\int _{\partial B_R} f_{g}^2 &\les \int_0^{\frac{3}{2}} \int_{\Omega\times O \times O} |R-|z|| \chi_{|X(1-\tau)|\leq R}(X) \chi_{B_R^c}(z) d\Pi \nonumber\\
&\stackrel{\eqref{estimH1}}{\le} \int_{\Omega\times O \times O} |X(1-\tau)-z|^2 \chi_{|X(0)|<2 }(X)  d\Pi \nonumber\\
& \les  \int_{\Omega\times O \times O} |X(1-\tau)-y|^2 \chi_{|X(0)|<2 }(X)  d\Pi +  \int_{\Omega\times O \times O} |z-y|^2 d\Pi\nonumber\\
&\le \tau^2 E+D \label{L2boundfz'}.
\end{align}
We then let 
\[\hat{\mu}:= \tilde{\mu}\restr B_R +f_{\tilde \mu}\qquad \textrm{ and } \qquad \hat{g}:= g\restr B_R+ f_{g}.\]
Since projecting from outside $B_R$ reduces the distances
\begin{multline*}
 W^2_2(\rho_{1-\tau} \restr B_R,\hat{\mu})\le W^2_2(\rho_{1-\tau}\restr{B_R},\tilde \mu)\\
 \le \int_{\Omega} \chi_{|X(0)|<2} |X(1-\tau)-X(1)|^2=\tau^2\int_{B_2} |T-x|^2= \tau^2 E.
\end{multline*}
For the same reason, we also have
\[
 W^2_2(\hat{\mu},\hat{g})\le W^2_2(\tilde \mu,g)\le D. 
\]
Therefore by triangle inequality
\begin{equation}\label{triangleineq}
 W_2^2(\rho_{1-\tau} \restr B_R, \Lambda)\les W_2^2(\rho_{1-\tau} \restr B_R, \hat{\mu})+ W^2_2(\hat{\mu},\hat{g})+W_2^2(\hat{g}, \Lambda)\les \tau^2E+D +W_2^2(\hat{g}, \Lambda).
\end{equation}
We are thus left with the estimate of $W_2^2(\hat{g}, \Lambda)$. For this we first claim that 
\begin{equation*}
 W_2(\hat{g},\Lambda)\les W_2\lt(\frac{1}{2}(\hat{g}+\Lambda),\Lambda\rt).
\end{equation*}
Indeed, by triangle inequality and monotonicity of the transport cost
 \begin{align*}
 W_2(\Lambda,s)&\le W_2\lt(\Lambda, \frac{1}{2}(\Lambda+s)\rt)+ W_2\lt( \frac{1}{2}(\Lambda+s),s\rt)\\
 &\le W_2\lt(\Lambda, \frac{1}{2}(\Lambda+s)\rt) +W_2\lt( \frac{1}{2}\Lambda,\frac{1}{2}s\rt)\\
 &= W_2\lt(\Lambda, \frac{1}{2}(\Lambda+s)\rt) + \frac{1}{\sqrt{2}} W_2(\Lambda,s).
\end{align*}
Now let $\phi^{g}$ be the solution of 
\begin{equation}\label{eq:defphiz'}\begin{cases}
   \Delta \phi^{g} =\Lambda-g &\textrm{in } B_R\\[8pt]
   \frac{\partial \phi^{g}}{\partial \nu}= f_{g} &\textrm{on } \partial B_R,
  \end{cases}\end{equation}
 with $\int_{B_R} \phi^{g}=0$. Notice that since by definition of $\Lambda$ and $g$,  $\Lambda=\frac{1}{|B_R|} g(\R^d)$ so that by definition of $f_g$,
 \[
  \int_{B_R} (\Lambda-g)= g(B_R^c)=\int_{\partial B_R} f_{g},
 \]
so that this equation is indeed solvable. Let  
\[
 \tilde\rho:= (1-t)\Lambda+t \frac{1}{2} (\Lambda+ \hat{g}) \qquad \textrm{and} \qquad \tilde j:= \frac{1}{2} \nabla \phi^{g}.
\]
The pair $(\tilde \rho,\tilde j)$ is admissible for the Benamou-Brenier formulation  \eqref{BBwass} of $W^2_2(\Lambda, \frac{1}{2}(\Lambda+\hat{g}))$ since \eqref{eq:defphiz'} implies in a  distributional sense
\[
 \nabla\cdot \tilde j=\frac{1}{2}\lt(\Lambda-g-f_{g}\rt)=\frac{1}{2}\lt(\Lambda-\hat{g}\rt) \qquad \textrm{in } \R^d,
\]
where we think of $\tilde j$ as being extended by zero from $B_R$ to $\R^d$. Hence, as desired,
\[
 \partial_t \tilde\rho+ \nabla\cdot \tilde j=0 \qquad \textrm{in } \R^d\times(0,1) 
\]
in a distributional sense.
Noticing that 
\[
 \tilde\rho\ge\frac{1}{2} \Lambda,
\]
we thus have
\begin{equation}\label{estimW2Lambdaphi}
 W_2^2(\Lambda, \frac{1}{2}(\Lambda+\hat{g}))\les \frac{1}{\Lambda}\int_{B_R} |\nabla \phi^{g}|^2.
\end{equation}
Let $g_-:=(\Gamma-g)\chi_{B_R}$ so that by definition of $g$,
\[
 \int_{B_R} \zeta dg_-=\int_{\Omega\times O\times O } \chi_{ |X(1-\tau)|>R }(x)\chi_{B_R} (z) \zeta(z) d\Pi .
\]
Thanks to the $L^\infty$ bound  \eqref{LinftyboundT} on the transport, we have that $\spt g_-\subset \overline{B}_R\backslash B_{R/2}$. We can rewrite $\Delta \phi^{g}=\Lambda-g= g_--(\Gamma-\Lambda)$ so that by Lemma \ref{lem:elliptic}, 
\[
 \int_{B_R} |\nabla \phi^{g}|^2\les \int_{\partial B_R} f_g^2+\int_{B_R} (R-|x|) d g_-.
\]
Arguing as for \eqref{toproveL2}, we get
\[
 \int_{\frac{1}{2}}^{\frac{3}{2}}\int_{B_R} (R-|x|) d g_-\les D,
\]
so that using  \eqref{L2boundfz'}, \eqref{triangleineq} and \eqref{estimW2Lambdaphi} we obtain \eqref{distdata1}. From this we see that we may find $R\in (\frac{1}{2},\frac{3}{2})$ such that both the conclusions of Lemma \ref{goodR} and \eqref{distdata} hold. 
\end{proof}

\subsection{The main estimate}

To ease notation, we shall now assume that $R=1$ and we will drop the index $R$. The main goal of this section is to prove Proposition \ref{prop:intromaineulerian} 
which states that for every fixed $\tau\ll1$, there exists a constant $C(\tau)>0$ such that if $E$ and $D$ are small enough, then there exists an harmonic gradient field $\nabla \phi$ in $B_1$ such that 
\begin{equation}\label{aimeulerian}
 \int_{B_{\frac{1}{2}}}\int_0^1 \frac{1}{\rho}|j-\rho \nabla \phi|^2\les \tau E+ C(\tau)D.
\end{equation}
From this Eulerian estimate, Proposition \ref{prop:detintro}  which is the Lagrangian counterpart,  is readily obtained. This in turn leads to the proof of Theorem \ref{theo:det intro}, which is one step in a Campanato iteration scheme.

We now proceed with the definition of $\phi$. Recall $\Brho_\pm$ from Lemma \ref{goodR} and let $\phi$ be the (unique) solution of 
\begin{equation}\label{def:phi}\begin{cases}
   \Delta \phi=\frac{1}{|B_1|}\int_{\partial B_1} (\Brhop-\Brhom) &\textrm{in } B_1\\[8pt]
   \frac{\partial \phi}{\partial \nu}= \Brhop-\Brhom &\textrm{on } \partial B_1,
  \end{cases}\end{equation}
such that $\int_{B_1} \phi=0$. Notice that by \eqref{CalZyg}, the H\"older inequality  and \eqref{L2boundrho}
\begin{equation}\label{L2boundphi}
 \int_{B_1} |\nabla \phi|^2 \les \int_{\partial B_1} \Brhop^2+\Brhom^2\les E.
\end{equation}
Moreover, by Pohozaev, we also have 
\begin{equation}\label{pohozaev}
 \int_{\partial B_1} |\nabla \phi|^2\les E.
\end{equation}

The proof of \eqref{aimeulerian} is divided into two parts. The first is an almost orthogonality property (see \eqref{eq:distjphiprop}) and the second is a construction of a competitor to estimate 
\[
 \int_{B_1}\int_0^1 \frac{1}{\rho}|j|^2-\int_{B_1} |\nabla \phi|^2,
\]
see \eqref{eq:distjphiprop2}. We start with the almost orthogonality property.
\begin{proposition}\label{prop:almostorth}
 For every $0<\tau\ll1$, there exist constants $\eps(\tau)>0$ and $C(\tau)>0$  such that if  $E+D\le \eps(\tau)$, then letting $\phi$ be defined via \eqref{def:phi}, we have 
 \begin{equation}\label{eq:distjphiprop}
  \int_{B_{\frac{1}{2}}}\int_0^1\frac{1}{\rho}|j-\rho\nabla \phi|^2\les \lt(  \int_{B_1}\int_0^1 \frac{1}{\rho} |j|^2-\int_{B_1} |\nabla \phi|^2\rt)+\tau E+ C(\tau) D.
 \end{equation}

\end{proposition}
\begin{proof}
 {\textit Step 1.} Before starting, let us point out that since in $B_{\frac{1}{2}}$ the function $\nabla \phi$ is smooth, the measure $\rho \nabla \phi$ is well defined. Furthermore, since clearly $j-\rho \nabla \phi \ll \rho$ also the left-hand side of \eqref{eq:distjphiprop} is well defined. We start by noting that 
 \[
  \int_{B_{\frac{1}{2}}}\int_0^1\frac{1}{\rho}|j-\rho\nabla \phi|^2= \int_{B_{\frac{1}{2}}}\int_0^{1-\tau}\frac{1}{\rho}|j-\rho\nabla \phi|^2+  \int_{B_{\frac{1}{2}}}\int_{1-\tau}^1 \frac{1}{\rho}|j-\rho\nabla \phi|^2
 \]
and that since by harmonicity of $\nabla \phi$,  $\sup_{B_{\frac{1}{2}}} |\nabla \phi|^2\les \int_{B_1} |\nabla \phi|^2$,
\begin{align*}
 \int_{B_{\frac{1}{2}}}  \int_{1-\tau}^1 \frac{1}{\rho}|j-\rho\nabla \phi|^2&\les  \int_{B_{\frac{1}{2}}} \int_{1-\tau}^1 \frac{1}{\rho}|j|^2 +\sup_{B_{\frac{1}{2}}} |\nabla \phi|^2  \int_{B_{\frac{1}{2}}} \int_{1-\tau}^1 \rho\\
 &\stackrel{\eqref{Linftyboundslice}\& \eqref{L2boundphi}}{\les} \tau E. 
\end{align*}
Therefore,
\begin{equation}\label{eq:split1}
  \int_{B_{\frac{1}{2}}} \int_0^1 \frac{1}{\rho}|j-\rho\nabla \phi|^2\les  \int_{B_{1}} \int_0^{1-\tau} \frac{1}{\rho}|j-\rho\nabla \phi|^2+ \tau E
\end{equation} 
and we are left with bounding the first term on the right-hand side. Notice that since in $(0,1-\tau)$, $\rho$ and $j$ are bounded functions by \eqref{Linfty}, the right-hand side of \eqref{eq:split1} is well defined in a pointwise sense. Recalling that $\Brho=\int_0^{1-\tau} \rho$, we may now compute 
\begin{align}\label{eq:split2}
 \int_{B_{1}} \int_0^{1-\tau}\frac{1}{\rho}|j-\rho\nabla \phi|^2&=\int_{B_1}\int_0^{1-\tau}\frac{1}{\rho}|j|^2-\int_{B_1}|\nabla \phi|^2\\
 &\qquad -2\int_{B_1}\int_0^{1-\tau} \lt(j-\frac{1}{1-\tau}\nabla \phi\rt)\cdot \nabla \phi +\int_{B_1} (\Brho-1)|\nabla \phi|^2. \nonumber
\end{align}
\medskip

{\it Step 2.} In this step we show that 
\begin{equation}\label{eq:firsttermorth}
 \lt|\int_{B_1} (\Brho-1)|\nabla \phi|^2\rt|\les \lt(\tau^{-d(d-1)} \gamma_d(\tau)\rt)^{\frac{1}{d+2}} E^{\frac{d+3}{d+2}} +\tau E,
\end{equation}
where 
\begin{equation}\label{def:gammad}
 \gamma_d(\tau):=\begin{cases}
                1 & \textrm{for } d=2\\
                |\log \tau| & \textrm{for } d=3\\
                \tau^{-(d-3)} & \textrm{otherwise}.
               \end{cases}
\end{equation}
Let the boundary layer size $r\ll 1$ to be fixed later and let $\eta$ be a smooth cut-off function with $\chi_{B_{1-2r}}\le \eta\le\chi_{B_{1-r}}$ and $|\nabla \eta|\les r^{-1}$. We split the integral: 
\begin{equation}\label{firstfirsttermortho}
 \int_{B_1} (\Brho-1)|\nabla \phi|^2=\int_{B_1} (\Brho-1)(1-\eta) |\nabla \phi|^2+ \int_{B_1} (\Brho-1)\eta |\nabla \phi|^2.
\end{equation}
The first term may be estimated as follows
\begin{align}\nonumber
 \lt|\int_{B_1} (\Brho-1)(1-\eta) |\nabla \phi|^2\rt|&\le\int_{B_1\backslash B_{1-2r}} |\Brho-1| |\nabla \phi|^2\\ \nonumber 
 &\stackrel{\eqref{Linftyrhobar}}{\les} \tau^{-(d-1)} \int_{B_1\backslash B_{1-2r}} |\nabla \phi|^2\\ 
 &\stackrel{\eqref{estimphiAr}}{\les} r \tau^{-(d-1)} \int_{\partial B_1} \Brhop^2+\Brhom^2\stackrel{\eqref{L2boundrho}}{\les} r \tau^{-(d-1)} E \label{firsttermortho}.
\end{align}
 We  now turn to  the second term. By
\begin{multline*}
  \lt|\int_{B_1} (\Brho-1)\eta |\nabla \phi|^2\rt|\le  \lt|\int_{B_1} (\Brho-(1-\tau))\eta |\nabla \phi|^2\rt|+ \tau \int_{B_1} |\nabla \phi|^2\\
  \stackrel{\eqref{L2boundphi}}{\les} \lt|\int_{B_1} (\Brho-(1-\tau))\eta |\nabla \phi|^2\rt|+ \tau E, 
\end{multline*}
 it is enough to estimate $\lt|\int_{B_1} (\Brho-(1-\tau))\eta |\nabla \phi|^2\rt|$. To this purpose we give an alternative representation: since $-(1-\tau-t)\eta |\nabla \phi|^2\in C^\infty_c( B_1\times [0,1])$ we can extend it by zero for $t\in [1-\tau,1]$ and test \eqref{conteqloc} with it to obtain
\begin{multline*}
 \int_{B_1} \bar \rho \eta |\nabla \phi|^2= \int_{B_1}\int_0^{1-\tau} \rho \partial_t(-(1-\tau-t) \eta |\nabla \phi|^2)\\
 =\int_{B_1}\int_0^{1-\tau} (1-\tau-t) j\cdot \nabla(\eta |\nabla \phi|^2) +\int_{B_1} (1-\tau) \eta |\nabla \phi|^2.
\end{multline*}
Therefore,
\begin{align}
 \lt|\int_{B_1} (\Brho-(1-\tau))\eta |\nabla \phi|^2\rt|&=\lt|\int_{B_1}\int_0^{1-\tau} (1-\tau-t) j\cdot \nabla (\eta |\nabla \phi|^2)\rt|\nonumber\\
 &\le \lt(\int_{B_1}\int_0^{1-\tau}\frac{1}{\rho}|j|^2\rt)^{\frac12}\lt(\int_{B_1}\int_0^{1-\tau} (1-\tau-t)^2\rho  |\nabla (\eta |\nabla \phi|^2)|^2\rt)^{\frac12}  \nonumber\\
 &\stackrel{\eqref{Linfty}}{\les} E^{\frac12} \lt(\int_0^{1-\tau} \frac{1}{(1-t)^{d-2}} \int_{B_1} |\nabla (\eta |\nabla \phi|^2)|^2\rt)^{\frac12} \label{firsttermsecondorth}\\
 &\les   \gamma_d^{\frac12}(\tau)E^{\frac12} \lt(\int_{B_1} |\nabla (\eta |\nabla \phi|^2)|^2\rt)^{\frac12} \nonumber,
\end{align}
where we recall that $\gamma_d$ is defined in \eqref{def:gammad}. By Leibniz rule and Cauchy-Schwarz we have 
\begin{align*}
 \int_{B_1} |\nabla (\eta |\nabla \phi|^2)|^2&\les \frac{1}{r^2} \int_{B_{1-r}} |\nabla \phi|^4 +\int_{B_{1-r}} |\nabla \phi|^2|\nabla^2 \phi|^2\\
 &\les \frac{1}{r^2} \int_{B_{1-r}} |\nabla \phi|^4 +r^2 \int_{B_{1-r}} |\nabla^2 \phi|^4.
\end{align*}
By the mean value formula for $\nabla \phi$, for every $x\in B_{1-r}$,
\[
  |\nabla^2 \phi|(x)\les \frac{1}{r} \frac{1}{|B_{\frac{r}{2}}|}\int_{B_{\frac{r}{2}}(x)} |\nabla \phi|
\]
so that integrating, using Jensen inequality and Fubini,
\[
 r^2 \int_{B_{1-r}} |\nabla^2 \phi|^4\les \frac{1}{r^2} \int_{B_{1-\frac{r}{2}}} |\nabla \phi|^4
\]
from which the above estimate simplifies to 
\[
 \int_{B_1} |\nabla (\eta |\nabla \phi|^2)|^2\stackrel{\eqref{L2boundrho}}{\les} \frac{1}{r^2} \int_{B_{1-\frac{r}{2}}} |\nabla \phi|^4.
\]
Let $p=\frac{2d}{d-1}$. By the mean value formula for $\nabla \phi$ and Jensen's inequality,
\[\sup_{B_{1-\frac{r}{2}}} |\nabla \phi|\les \lt(\frac{1}{r^{d}}\int_{B_1} |\nabla \phi|^p\rt)^{\frac{1}{p}}\stackrel{\eqref{CalZyg}}{\les} r^{-\frac{d}{p}} \lt(\int_{\partial B_1} \Brhop^2+\Brhom^2\rt)^{\frac12}\les r^{-\frac{d}{p}} E^{\frac12}.\]
We then have
\begin{align*}
  \frac{1}{r^2} \int_{B_{1-\frac{r}{2}}} |\nabla \phi|^4&\le  \frac{1}{r^2} \sup_{B_{1-\frac{r}{2}}} |\nabla \phi|^{4-p} \int_{B_1} |\nabla \phi|^p\\
  &\les r^{-2} \lt(r^{-\frac{d}{p}} E^{\frac12}\rt)^{4-p}E^{\frac{p}{2}}=r^{-d} E^2.
\end{align*}
Collecting all the previous estimates we obtain
\[
 \lt|\int_{B_1} (\Brho-1)\eta |\nabla \phi|^2\rt|\les r^{-\frac{d}{2}} \gamma_d^{\frac12}(\tau)E^{\frac{3}{2}}+\tau E 
\]
and thus plugging this and \eqref{firsttermortho} into \eqref{firstfirsttermortho}, we get
\[
 \lt|\int_{B_1} (\Brho-1)|\nabla \phi|^2\rt|\les r \tau^{-(d-1)} E +r^{-\frac{d}{2}} \gamma_d^{\frac12}(\tau)E^{\frac{3}{2}} +\tau E.
\]
Optimizing in $r$ through $r=\lt(\tau^{2(d-1)}\gamma_d(\tau) E\rt)^{\frac{1}{d+2}}$ and using $\gamma_d^{\frac12}(\tau) E^{\frac12}\ll \tau^{-(d-1)}$ to ensure that $r\ll 1$, we obtain the aimed estimate \eqref{eq:firsttermorth}.

\medskip
{\it Step 3.} We now estimate 
\[
 \int_{B_1}\int_0^{1-\tau} \lt(j-\frac{1}{1-\tau}\nabla \phi\rt)\cdot \nabla \phi.
\]
For this we want to use \eqref{conteqloc}  for $\zeta= \chi_{(0,1-\tau)} \phi$. Notice first that since $\rho$, $j$ and $\Bf_{\pm}$ (recall the definition \eqref{def:fbarpm}) are bounded densities in
$(0,1-\frac{\tau}{2})$ (see Lemma \ref{lem:Linftyboundslice} and Lemma \ref{lem:linftybarf}), by density we can apply \eqref{conteqloc} to $\zeta\in H^{1}(B_1\times (0,1))$ with $\spt \zeta \subset \overline{B}_1\times [0, 1-\tau/2]$. 
 Let  $\phi_\delta\in C^0(B_1)$ be a mollification of $\phi$ so that by continuity of $t\to \rho_t$  in $W_2$, 
\begin{equation}\label{eq:phidelta}
\frac1\eps\int_{1-\tau}^{1-\tau+\eps}\int_{B_1}\phi_\delta \rho_t \to \int_{B_1} \phi_\delta\rho_{1-\tau}. 
\end{equation}
   Then, apply \eqref{conteqloc} to $\eta_\eps(t)\phi_\delta(x)$ where for $\eps >0$ 
\[
 \eta_\eps(t)=\begin{cases}
               1 &\textrm{for } t\in (0,1-\tau]\\
               1- \eps^{-1}(t-(1-\tau)) & \textrm{for } t\in(1-\tau,1-\tau +\eps)\\
               0 &\textrm{for } t\ge  1-\tau +\eps
              \end{cases}
\]
to obtain for $\eps\to 0$ using \eqref{eq:phidelta}
\[
 \int_{B_1}\int_0^{1-\tau}\nabla\phi_\delta\cdot j = \int_{B_1}\phi_\delta \rho_{1-\tau}-\int_{B_1}\phi_\delta + \int_{\R^d}\int_0^{1-\tau} \phi_\delta df.
\]
Letting $\delta \to 0$ using that $\Delta \phi=\textrm{constant}$ and $\int_{B_1} \phi=0$ and recalling the definition of $\phi$ in \eqref{def:phi} and \eqref{defBfR}, we thus obtain
\begin{equation}\label{toproveorthogonal}
 \int_{B_1}\int_0^{1-\tau} \lt(j-\frac{1}{1-\tau}\nabla \phi\rt)\cdot \nabla \phi=  \int_{B_1} \phi \rho_{1-\tau} +\int_{\partial B_1} \phi [(\Bf_+-\Brhop) -(\Bf_- -\Brhom)].
\end{equation}

Let us estimate the first term. Let $(\tilde{\rho},\tilde{j})$ be given by the Benamou-Brenier theorem and such that
\begin{equation}\label{eq:BBtilde}
 \int_{B_1}\int_{0}^1 \frac{1}{\tilde \rho}|\tilde j|^2 =W_2^2\lt( \frac{\rho_{1-\tau}(B_1)}{|B_1|} \chi_{B_1},\rho_{1-\tau}\restr B_1\rt)\stackrel{\eqref{distdata}}{\les} \tau^2 E+D.
\end{equation}
 If $\widetilde{T}$ is the optimal transport map between $ \frac{\rho_{1-\tau}(B_1)}{|B_1|} \chi_{B_1}$ and $\rho_{1-\tau}\restr B_1$, 
\begin{align*}
 \tilde \rho_t^{-\frac1d}&= \lt(\frac{\rho_{1-\tau}(B_1)}{|B_1|}\rt)^{-\frac{1}{d}} \mathrm{det}^{\frac{1}{d}} \nabla \tilde{T}_t(\tilde{T}_t^{-1})\\
 &\stackrel{\eqref{detconcave}}{\ge} \lt(\frac{\rho_{1-\tau}(B_1)}{|B_1|}\rt)^{-\frac{1}{d}} \lt( (1-t)+ t\mathrm{det}^{\frac{1}{d}} \nabla \tilde{T}(\tilde{T}_t^{-1})\rt)\\
 &= (1-t) \lt(\frac{\rho_{1-\tau}(B_1)}{|B_1|}\rt)^{-\frac1d}+ t \rho_{1-\tau}^{-\frac1d}(\tilde{T}_t^{-1})\\
 &\stackrel{\eqref{Linfty}}{\ge} (1-t) \lt(\frac{\rho_{1-\tau}(B_1)}{|B_1|}\rt)^{-\frac1d}+ t \tau
\end{align*}
and thus  since  $\frac{\rho_{1-\tau}(B_1)}{|B_1|}\sim 1$ thanks to \eqref{LinftyboundT},
\begin{equation}\label{argdisplconv}
 \tilde \rho\les ((1-t)+t\tau)^{-d}. 
\end{equation}
We then have because of $\int_{B_1} \phi=0$,
\begin{align}
 \lt|\int_{B_1} \phi \rho_{1-\tau}\rt|&= \lt|\int_{B_1}\int_0^1 \nabla \phi \cdot \tilde j\rt|\nonumber \\
 &\le  \lt(\int_{B_1}\int_0^1 \tilde \rho |\nabla \phi|^2\rt)^{\frac12}\lt(\int_{B_1}\int_0^1 \frac{1}{\tilde \rho} |\tilde j|^2\rt)^{\frac12}\nonumber\\
 &\stackrel{\eqref{eq:BBtilde}}{\les} \lt(\int_{B_1}\int_0^1 \lt(\tilde \rho-\frac{\rho_{1-\tau}(B_1)}{|B_1|}\rt) |\nabla \phi|^2 + \frac{\rho_{1-\tau}(B_1)}{|B_1|}\int_{B_1} |\nabla \phi|^2\rt)^{\frac12} \lt(\tau^2 E+D\rt)^{\frac12}\nonumber\\
 &\stackrel{\eqref{L2boundphi}}{\les} \lt(\int_{B_1}\int_0^1 \lt(\tilde \rho-\frac{\rho_{1-\tau}(B_1)}{|B_1|}\rt) |\nabla \phi|^2 + E\rt)^{\frac12} \lt(\tau^2 E+D\rt)^{\frac12}\nonumber\\
 &\les \tau \int_{B_1}\int_0^1 \lt(\tilde \rho-\frac{\rho_{1-\tau}(B_1)}{|B_1|}\rt) |\nabla \phi|^2 +\tau E +\tau^{-1} D \label{estimphirho},
\end{align}
where in the last line we used Young's inequality.
The term $\int_{B_1}\int_0^1 \lt(\tilde \rho-\frac{\rho_{1-\tau}(B_1)}{|B_1|}\rt) |\nabla \phi|^2$ is estimated as in Step 2. Indeed, choosing for $r\ll1$ a smooth cut-off function 
$\eta$ with $\chi_{B_{1-2r}}\le \eta\le \chi_{B_{1-r}}$, we obtain as in \eqref{firsttermortho} that 
\[
 \lt|\int_{B_1}\int_0^1 \lt(\tilde \rho-\frac{\rho_{1-\tau}(B_1)}{|B_1|}\rt) (1-\eta) |\nabla \phi|^2\rt|\les r \tau^{-(d-1)} E.
\]
Using that 
\[
 \tilde \rho-\frac{\rho_{1-\tau}(B_1)}{|B_1|}=\int_0^1(1-t) \partial_t \tilde{\rho},
\]
we obtain as in \eqref{firsttermsecondorth}
\begin{align*}
 \lt|\int_{B_1}\int_0^1 \lt(\tilde \rho-\frac{\rho_{1-\tau}(B_1)}{|B_1|}\rt) \eta |\nabla \phi|^2\rt|&\les \lt(\int_{B_1} \int_0^1 \frac{1}{\tilde \rho} |\tilde j|^2\rt)^{\frac12} \lt(\int_{B_1}\int_0^1 \frac{(1-t)^2}{ ((1-t)+t\tau)^{d}} |\nabla (\eta|\nabla \phi|^2)|^2\rt)^{\frac12}\\
 &\stackrel{\eqref{eq:BBtilde}}{\les} \gamma_d^{\frac12}(\tau) (\tau^2E+D)^{\frac12} \lt(\int_{B_1} |\nabla (\eta|\nabla \phi|^2)|^2\rt)^{\frac12}\\
 &\les  r^{-\frac{d}{2}} \gamma_d^{\frac12}(\tau) (\tau^2 E+D)^{\frac12} E.
\end{align*}
Optimizing in $r$, we get
\[
 \lt|\int_{B_1}\int_0^1 \lt(\tilde \rho-\frac{\rho_{1-\tau}(B_1)}{|B_1|}\rt) |\nabla \phi|^2\rt|\les \lt(\tau^{-d(d-1)} \gamma_d(\tau)\rt)^{\frac{1}{d+2}} (\tau^2 E+D)^{\frac{1}{d+2}} E.
\]
Plugging this into \eqref{estimphirho} we obtain for some $C(\tau)\gg1$,
\begin{equation}\label{secondtermfirstorth}
 \lt|\int_{B_1} \phi \rho_{1-\tau}\rt|\les  \lt(C(\tau) (\tau^2 E+D)^{\frac{1}{d+2}} +\tau\rt) E +\tau^{-1} D.
\end{equation}
We now turn to the second term in \eqref{toproveorthogonal}. It is enough to bound 
\[
 \int_{\partial B_1} \phi (\Bf_+-\Brhop)
\]
since the other term is treated analogously. Let $(\hat \rho, \hat j)$ be the minimizer of \eqref{BBsphere}, i.e.
\[
 \int_{\partial B_1}\int_0^1 \frac{1}{\hat \rho}|\hat j|^2= W_{\partial B_1}^2(\Brhop,\Bf_+)\stackrel{\eqref{L2boundrho}}{\les}E^{\frac{d+3}{d+2}}.
\]
Arguing as for \eqref{argdisplconv} but using \eqref{jacobiansphere} and \eqref{displconvsphere} together with  \eqref{Linftyboundrho} and  \eqref{LinftyBarf}, we obtain
\[
 \hat \rho^{-\frac{1}{d-1}}\ges  \lt((1-t)+ t\tau\rt) E^{-\frac{1}{(d-1)(d+2)}}
\]
so that 
\[
 \int_0^1 \hat \rho \les \tau^{-(d-2)} E^{\frac{1}{d+2}},
\]
with the convention that for $d=2$, $\tau^{-(d-2)}=|\log \tau|$.
By integration by parts we then have 
\begin{align*}
 \lt|\int_{\partial B_1} \phi (\Bf_+-\Brhop)\rt|&=\lt|\int_{\partial B_1} \int_0^1 \nabbdr \phi \cdot \hat j\rt|\\
 &\le \lt(\int_{\partial B_1} \int_0^1\hat \rho |\nabla \phi|^2\rt)^{\frac12}\lt(\int_{\partial B_1}\int_{0}^1 \frac{1}{\hat \rho}|\hat j|^2\rt)^{\frac12}\\
 &\stackrel{\eqref{pohozaev}}{\les} \lt(\tau^{-(d-2)} E^{\frac{1}{d+2}}\rt)^{\frac12} E^{\frac12} \lt(E^{\frac{d+3}{d+2}}\rt)^{\frac{1}{2}}\\
 &\les \tau^{-\frac{d-2}{2}} E^{\frac{d+3}{d+2}},
\end{align*}
 with the convention that for $d=2$, $\tau^{-\frac{d-2}{2}}=|\log \tau|^{\frac12}$. This estimate together with \eqref{secondtermfirstorth} yields
\begin{equation}\label{eq:secondtermmorth}
 \int_{B_1}\int_0^{1-\tau} \lt(j-\frac{1}{1-\tau}\nabla \phi\rt)\cdot \nabla \phi\les \lt(C(\tau)\lt[(\tau^2 E+D)^{\frac{1}{d+2}} +E^{\frac{1}{d+2}}\rt] +  \tau\rt) E+ \tau^{-1} D.
\end{equation}

Putting together \eqref{eq:split2}, \eqref{eq:firsttermorth} and \eqref{eq:secondtermmorth}, we conclude that
\begin{multline*}
  \int_{B_{\frac{1}{2}}}\int_0^1\frac{1}{\rho}|j-\rho\nabla \phi|^2- \lt(  \int_{B_1}\int_0^1 \frac{1}{\rho} |j|^2-\int_{B_1} |\nabla \phi|^2\rt)\\
  \les \lt(C(\tau)\lt[(\tau^2 E+D)^{\frac{1}{d+2}} +E^{\frac{1}{d+2}}\rt] +  \tau\rt) E+ \tau^{-1} D,
\end{multline*}
so that \eqref{eq:distjphiprop} follows if $E+D\le \eps(\tau)$ for some $\eps(\tau)$ small enough.
\end{proof}

We may now use the minimality of $(\rho,j)$ to estimate $ \int_{B_1}\int_0^1 \frac{1}{\rho} |j|^2-\int_{B_1} |\nabla \phi|^2$.

\begin{proposition}\label{prop:distjphiprop2}
 For every $0<\tau\ll 1$, there exist $\eps(\tau)$ and $C(\tau)$ such that if  $E+D\le \eps(\tau)$, then letting $\phi$ be defined via \eqref{def:phi}, we have
 \begin{equation}\label{eq:distjphiprop2}
    \int_{B_1}\int_0^1 \frac{1}{\rho} |j|^2-\int_{B_1} |\nabla \phi|^2\les \tau E+ C(\tau) D.
 \end{equation}
\end{proposition}

\begin{proof}
Recall the measure $f$ from  \eqref{def:fR}.
We are going to  construct a competitor $(\tilde \rho, \tilde j)$ supported in $\overline{B}_1\times[0,1]$ of the form  
\[
\tilde \rho_t= \rho_t^{\textrm{bdr}} \H^{d-1}\restr \partial B_1 +\rho_t^{\textrm{in}} dx\restr B_1 \quad \textrm{and }\quad \tilde j_t= j_t^{\textrm{bdr}} \H^{d-1}\restr \partial B_1 +j_t^{\textrm{in}} dx\restr B_1,
\]
where $j_t^{\textrm{bdr}}$ is tangent to $\partial B_1$, and such that for every $\zeta\in C^1(\R^d\times[0,1])$
\begin{multline}\label{eq:conteqtildrho}
\int_{\partial B_1} \int_0^1 \partial_t \zeta \rho^{\textrm{bdr}}+\nabbdr \zeta\cdot  j^{\textrm{bdr}}+\int_{B_1}\int_0^1  \partial_t \zeta \rho^{\textrm{in}}+\nabla \zeta\cdot  j^{\textrm{in}}\\
=\int_{B_1} \zeta_1d\mu-\int_{B_1}\zeta_0+\int_{\partial B_1}\int_0^1 \zeta f.
\end{multline}
By \eqref{eq:locmin}, we then have
\begin{equation}\label{eq:locmin2}
 \int_{B_1} \int_{0}^1 \frac{1}{\rho}|j|^2- \int_{B_1} |\nabla \phi|^2\le \int_{\R^d}\int_{0}^1 \frac{1}{\tilde \rho}|\tilde j|^2- \int_{B_1} |\nabla \phi|^2.
\end{equation}

\medskip
 For the construction we will decompose $(\tilde \rho, \tilde j)$ into (see Figures \ref{fig:rhobulk}-\ref{fig:rholay2})
\begin{align*}
\rho^{\textrm{in}}&:=\rho^{\textrm{bulk}} +  \rho^{\textrm{lay}}\\[6pt]
j^\textrm{in}&:= j^\bulk +j^{\textrm{lay}}\\[6pt]
\rho^{\textrm{bdr}}&:=\rho^{\textrm{bdr},\bulk} +\rho^{\textrm{bdr,lay}}\\[6pt]
 j^{\textrm{bdr}}&:=j^{\bdr,\bulk}+j^{\bdr,\lay}.
\end{align*}
The bulk terms will live in the time interval $(0,1-\tau)$ while the layer terms will allow to treat the boundary layer (in time) but will be defined for all $t\in(0,1)$.  
One of the crucial points for the estimate is that 
\begin{equation}\label{Linftyconstraint}
 \frac{1}{4}\le \rho^{\textrm{in}}\le 2 \qquad \textrm{for }\quad  t\in(0,1-\tau).
\end{equation}
Indeed, we will then have (recall that for $t\in (1-\tau,1)$, $\rho^{\textrm{in}}=\rho^{\textrm{lay}}$)
\begin{align*}
 \lefteqn{\int_{B_1}\int_0^{1} \frac{1}{\tilde{\rho}}|\tilde{j}|^2- \int_{B_1} |\nabla \phi|^2}\\
 &=\int_{B_1}\int_0^{1-\tau} \frac{1}{\rho^{\textrm{in}}}|j^{\textrm{in}}|^2-\int_{B_1}\int_0^{1} |\nabla \phi|^2 +\int_{\partial B_1}\int_{0}^1 \frac{1}{\rho^{\bdr}}|j^\bdr|^2 +\int_{B_1}\int_{1-\tau}^1\frac{1}{\rho^{\lay}}|j^{\lay}|^2 \\
 &\le  \int_{B_1}\int_0^{1-\tau} \frac{1}{\rho^{\textrm{in}}}\lt||j^{\textrm{in}}|^2-\rho^\textrm{in}|\nabla \phi|^2\rt| +\int_{\partial B_1}\int_{0}^1 \frac{1}{\rho^{\bdr}}|j^\bdr|^2+\int_{B_1}\int_{1-\tau}^1\frac{1}{\rho^{\lay}}|j^{\lay}|^2 \\
 &\stackrel{\eqref{Linftyconstraint}}{\les}\int_{B_1}\int_0^{1-\tau} \lt||j^\bulk|^2-\rho^\textrm{in}|\nabla \phi|^2\rt| +\lt(\int_{B_1}\int_0^{1-\tau} |j^\bulk|^2\rt)^{\frac12} \lt(\int_{B_1}\int_0^{1-\tau} |j^\lay|^2\rt)^{\frac12} \\
 &\qquad  +\int_{B_1}\int_0^{1-\tau} |j^\lay|^2
 +\int_{\partial B_1}\int_{0}^1 \frac{1}{\rho^{\bdr}}|j^\bdr|^2+\int_{B_1}\int_{1-\tau}^1\frac{1}{\rho^{\lay}}|j^{\lay}|^2.
\end{align*}
Since for $(\rho,j)$ compactly supported (recall \eqref{dualBB})
\[
 \int_{\R^d}\int_0^1 \frac{1}{2\rho}|j|^2=\sup_{\xi\in C^0(\R^d\times[0,1],\R^d)} \int_{\R^d}\int_0^1 \xi\cdot j-\frac{|\xi|^2}{2}\rho
\]
is subadditive, we have
\[
 \int_{\partial B_1}\int_{0}^1 \frac{1}{\rho^{\bdr}}|j^\bdr|^2\le \int_{\partial B_1}\int_{0}^1 \frac{1}{\rho^{\bdr,\bulk}}|j^{\bdr,\bulk}|^2+\int_{\partial B_1}\int_{0}^1 \frac{1}{\rho^{\bdr,\lay}}|j^{\bdr,\lay}|^2
\]
so that 
\begin{align}
 \int_{B_1}\int_0^{1} \frac{1}{\tilde{\rho}}|\tilde{j}|^2- \int_{B_1} |\nabla \phi|^2\les&\int_0^{1-\tau}\int_{B_1} \lt||j^\bulk|^2-\rho^\textrm{in}|\nabla \phi|^2\rt| \nonumber\\
 &+\lt(\int_{B_1}\int_0^{1-\tau} |j^\bulk|^2\rt)^{\frac12} \lt(\int_{B_1}\int_0^{1-\tau} |j^\lay|^2\rt)^{\frac12}  \nonumber\\
 &+\int_{\partial B_1}\int_{0}^1 \frac{1}{\rho^{\bdr,\bulk}}|j^{\bdr,\bulk}|^2+\int_{\partial B_1}\int_{0}^1 \frac{1}{\rho^{\bdr,\lay}}|j^{\bdr,\lay}|^2 \nonumber\\
 &+\int_{B_1}\int_0^{1-\tau} |j^\lay|^2+\int_{B_1}\int_{1-\tau}^1\frac{1}{\rho^{\lay}}|j^{\lay}|^2.\label{splitenergy}
\end{align}
We now define and estimate the various contributions to the energy.
\medskip

{\it Step 1.} We start by constructing and estimating $(\rho^\bulk,j^\bulk)$. The main estimate of this step is
\begin{equation}\label{eq:estimbulk}
 \int_{B_{1}}\int_0^{1-\tau} \lt||j^\bulk|^2-\rho^\textrm{in}|\nabla \phi|^2\rt|\les \tau E +\lt(\tau^{-d} E\rt)^{\frac{d+2}{d+1}}+D.
\end{equation}
Note that the first right-hand side term of \eqref{splitenergy} involves $\rho^\lay$ through $\rho^{\textrm{in}}=\rho^\bulk+\rho^\lay$. However, 
for its estimate in this substep we only need that for  $t\in (0,1-\tau)$, 
\begin{equation}\label{Linftyconstraintrholay}
 \rho^\lay\les \lt(\tau^2 E+D\rt)^{\frac12}\ll 1.
\end{equation}
Recalling the definition \eqref{defBflay+} of  $\Bf^\lay_+$   and that $\Bf_\pm$ are defined in \eqref{def:fbarpm} similarly, we let 
\begin{equation}\label{def:fhat}
 m_+^\lay:=\int_{\partial B_1}\Bf_+^\lay \qquad \textrm{and} \qquad \hat{f}:=\begin{cases}
           0 & \textrm{for } t\in [0,\tau)\\[8pt]
           -\frac{2}{1-2\tau}\Bf_- & \textrm{for } t\in [\tau,\frac{1}{2})\\[8pt]
           \frac{2}{1-2\tau}\Bf_+ & \textrm{for } t\in [\frac{1}{2}, 1-\tau).
          \end{cases}
\end{equation}
Notice that since $\int_{\partial B_1}\Bf_+^\lay=\int_{\partial B_1} \Brhop^\lay$, by Cauchy-Schwarz and \eqref{L2boundrholay},
\begin{equation}\label{boundsmhatf1}
m_+^\lay\les \lt(\tau^2 E+D\rt)^{\frac12}.
\end{equation}
Moreover,  since $\Bf_+ \perp \Bf_-$ (recall Lemma \ref{lem:linftybarf}),
\[
 \int_{\partial B_1}\int_0^{1-\tau} \hat{f}^2\les\int_{\partial B_1} \Bf_+^2+\Bf_-^2 \stackrel{\eqref{defBfR}}{=}\int_{\partial B_1} \Bf^2\les  \int_{\partial B_1}\int_0^{1-\tau} f^2,
\]
where in the last inequality we used Jensen's inequality. This yields
\begin{equation}\label{boundsmhatf2}
  \int_{\partial B_1}\int_0^{1-\tau} (f-\hat{f})^2 \les\int_{\partial B_1}\int_0^{1-\tau} f^2\stackrel{\eqref{L2boundf}}{\les}  \tau^{-d}E.
\end{equation}

For a boundary layer size $r\gg \lt(\tau^{-d}E\rt)^{\frac{1}{d+1}}\ges \lt(\int_{\partial B_1}\int_0^{1-\tau} (f-\hat{f})^2\rt)^{\frac{1}{d+1}}$ let $A_r:=B_1\backslash B_{1-r}$, and let $(s,q)$ be given
by \cite[Lemma 3.4]{GO} applied to $f-\hat{f}$ and to the time interval $(0,1-\tau)$ instead of $(0,1)$. We recall that $(s,q)$ is  such that 
it has support in $\overline{A}_r\times[0,1-\tau]$, $|s|\le \frac{1}{2}$ and for $\zeta\in C^1(\R^d\times[0,1])$, 
\[
\int_{B_1} \int_0^{1} \partial_t \zeta s + \nabla \zeta \cdot q = \int_{A_r}\int_0^{1-\tau} \partial_t \zeta s + \nabla \zeta \cdot q=\int_{\partial B_1}\int_0^{1-\tau} \zeta(f-\hat{f}).
\]
 In addition it satisfies the estimate 
\begin{equation}\label{estimsq}
 \int_{A_r}\int_0^{1-\tau} |q|^2\les r  \int_{\partial B_1}\int_0^{1-\tau} (f-\hat{f})^2\stackrel{\eqref{boundsmhatf2}}{\les} r \tau^{-d}E.
\end{equation}
We then let $(\rho^\bulk,j^\bulk)$, supported in $\overline{B}_1\times [0,1-\tau]$, be defined through (see Figure \ref{fig:rhobulk})
 \begin{figure}\begin{center}
 \resizebox{10.cm}{!}{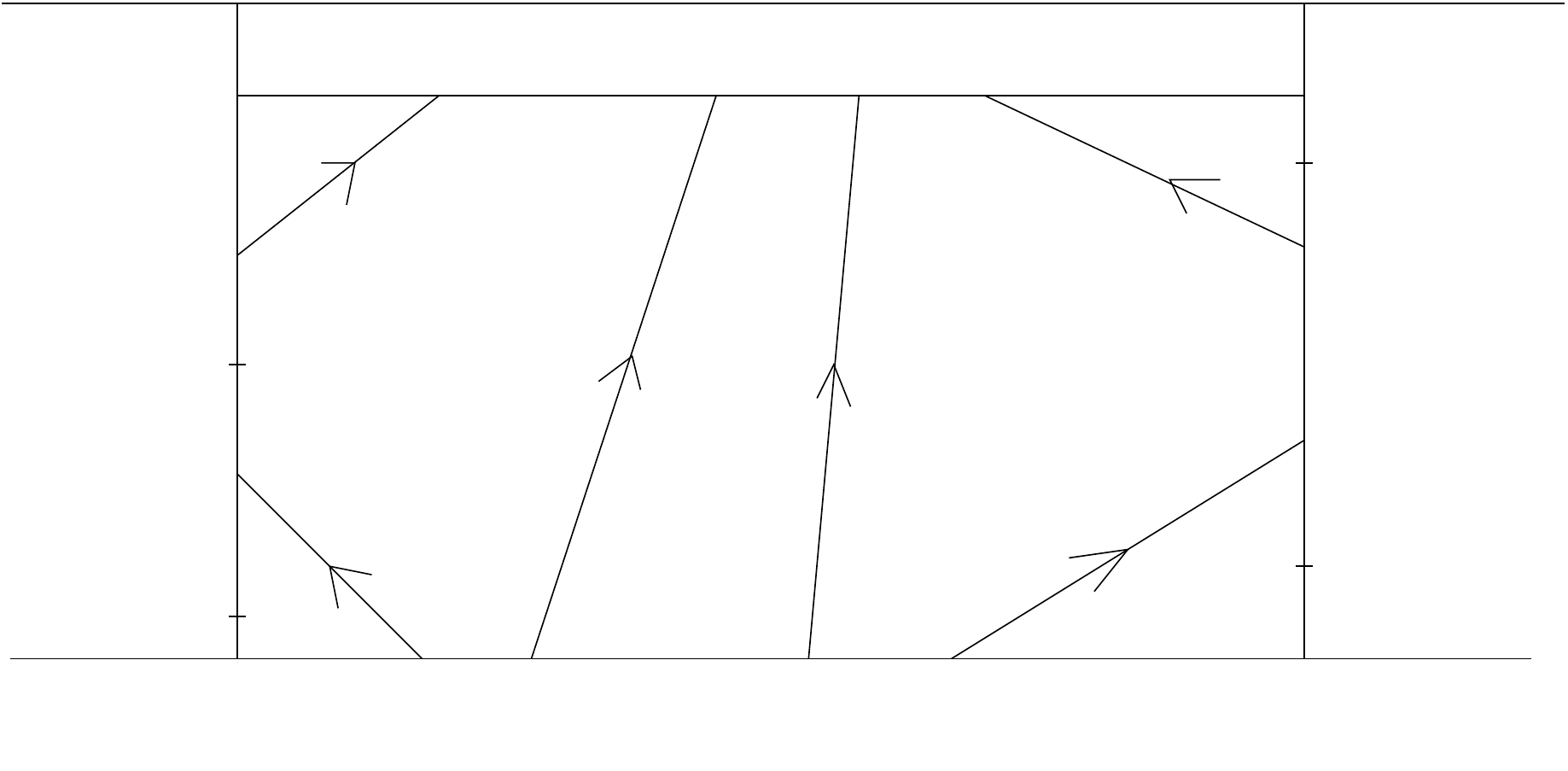} 
   \caption{The definition of $\rho^\bulk$.} \label{fig:rhobulk}
 \end{center}
 \end{figure}

\[
 \rho^\bulk:= 1-\frac{m_+^\lay}{|B_1|}+s -\frac{1}{|B_1|}\int_{\partial B_1} (\Brhop-\Brhom)\times \begin{cases}
                                                                                             0 &\textrm{ for } t\in [0,2\tau]\\
                                                                                             \frac{t-2\tau}{1-4\tau} & \textrm{ for } t\in [2\tau,1-2\tau]\\
                                                                                             1 &\textrm{ for } t\in [1-2\tau,1-\tau]
                                                                                            \end{cases}
\]
and 
\[j^\bulk:=q+ \nabla \phi\times \begin{cases}
                                                                                             0 &\textrm{ for } t\in [0,2\tau]\\
                                                                                             \frac{1}{1-4\tau} & \textrm{ for } t\in [2\tau,1-2\tau]\\
                                                                                             0 &\textrm{ for } t\in [1-2\tau,1-\tau],
                                                                                            \end{cases}\]
so that by definition \eqref{def:phi} of $\phi$
\begin{multline}\label{conteqbulk}
 \int_{B_1}\int_0^{1} \partial_t \zeta \rho^\bulk + \nabla \zeta \cdot j^\bulk \\
 =\int_{B_1} \zeta_{1-\tau} \lt(1-\frac{m_+^\lay}{|B_1|}-\frac{1}{|B_1|}\int_{\partial B_1} (\Brhop-\Brhom)\rt) - \zeta_0\lt(1-\frac{m_+^\lay}{|B_1|}\rt) \\
 +  \int_{\partial B_1}\int_0^{1-\tau}\zeta\lt(\chi_{(2\tau,1-2\tau)}\frac{1}{1-4\tau} (\Brhop-\Brhom) + f-\hat{f}\rt) .
\end{multline}
Notice that thanks to \eqref{Linftyconstraintrholay} also \eqref{Linftyconstraint} is satisfied.\\
Now, we can  start  estimating 
\[
 \int_{B_{1-r}}\int_0^{1-\tau} \lt||j^\bulk|^2-\rho^\textrm{in}|\nabla \phi|^2\rt|.
\]
By definition of $\rho^\bulk$, since $s$ vanishes in $B_{1-r}\times(0,1-\tau)$,  
\begin{multline*}
 |1-\rho^\textrm{in}|\les\rho^\lay+ m_+^{\lay} +\int_{\partial B_1} (\Brhop+\Brhom)
 \stackrel{\eqref{Linftyconstraintrholay}\& \eqref{boundsmhatf1}}{\les} \lt(\tau^2 E +D\rt)^{\frac12} +\lt(\int_{\partial B_1} \Brho_+^2+\Brho_-^2\rt)^{\frac{1}{2}}\\
 \stackrel{\eqref{L2boundrho}}{\les}  D^{\frac12} +E^{\frac12}\ll 1 \qquad \qquad \textrm{in } \quad B_{1-r}\times(0,1-\tau).
\end{multline*}
Therefore, since $q$ also vanishes on $B_{1-r}\times(0,1-\tau)$,
\begin{align*}
  \int_{B_{1-r}} \int_0^{1-\tau}\lt||j^\bulk|^2-\rho^\textrm{in}|\nabla \phi|^2\rt|&\les (D^{\frac12} +E^{\frac12} +\tau) \int_{B_1} |\nabla \phi|^2\\
  &\stackrel{\eqref{L2boundphi}}{\les} E^{\frac{3}{2}}+ D^{\frac12}E +\tau E\les E^{\frac{3}{2}} +\tau E +D,
\end{align*}
where in the last line we used Young's inequality together with the fact that since $E\ll1$, $E^2\les E^{\frac{3}{2}}$.
Choosing $r$ to be a large multiple of $\lt(\tau^{-d} E\rt)^{\frac{1}{d+1}}$,  we have 
\begin{align*}
 \int_{A_r}\int_0^{1-\tau} \lt||j^\bulk|^2-\rho^\textrm{in}|\nabla \phi|^2\rt|&\stackrel{\eqref{Linftyconstraint}}{\les}\int_{A_r} |\nabla \phi|^2+ \int_{A_r}\int_0^{1-\tau} |q|^2\\
 &\stackrel{\eqref{estimphiAr}\&\eqref{estimsq}}{\les} r\lt( E +\tau^{-d} E\rt)\\
 &\stackrel{\eqref{L2boundf}}{\les} \lt(\tau^{-d} E\rt)^{\frac{d+2}{d+1}}.
\end{align*}
Combining these two estimates and taking into account that since $\tau\ll1$ and $E\ll1$, $E^{\frac{3}{2}}\les  \lt(\tau^{-d} E\rt)^{\frac{d+2}{d+1}}$, we find \eqref{eq:estimbulk}.
Notice also for further reference that using the same argument, we obtain
\begin{equation}\label{eq:estimjbulk}
 \int_{B_{1}}\int_0^{1-\tau} |j^\bulk|^2\les E +\lt(\tau^{-d} E\rt)^{\frac{d+2}{d+1}}.
\end{equation}

\medskip
{\it Step 2.} We now define $(\rho^{\textrm{bdr},\bulk},j^{\textrm{bdr},\bulk})$, supported in $ \partial B_1\times [0,1-\tau]$ so that 
\begin{equation}\label{conteqbdrbulk}
 \int_{\partial B_1}\int_0^{1} \partial_t \zeta \rho^{\textrm{bdr},\bulk}+ \nabbdr \zeta\cdot j^{\textrm{bdr},\bulk}  =\int_{\partial B_1}\int_0^{1-\tau} \zeta \lt( \hat{f} -\chi_{(2\tau,1-2\tau)}\frac{1}{1-4\tau} (\Brhop-\Brhom)\rt)
\end{equation}
holds and
\begin{equation}\label{step4.2}
 \int_{\partial B_1} \int_0^1 \frac{1}{\rho^{\bdr,\bulk}}|j^{\bdr,\bulk}|^2\les |\log \tau| E^{\frac{d+3}{d+2}}.
\end{equation}
Notice that combining \eqref{conteqbulk} and \eqref{conteqbdrbulk} yields
\begin{multline}\label{conteqbulk2}
 \int_{B_1}\int_0^{1} \partial_t \zeta \rho^\bulk + \nabla \zeta \cdot j^\bulk+\int_{\partial B_1}\int_0^{1} \partial_t \zeta \rho^{\textrm{bdr},\bulk}+ \nabbdr \zeta\cdot j^{\textrm{bdr},\bulk}\\
 =\int_{B_1}\zeta_{1-\tau} \lt(1-\frac{m_+^\lay}{|B_1|}-\frac{1}{|B_1|}\int_{\partial B_1}  (\Brhop-\Brhom)\rt) -\zeta_0\lt(1-\frac{m_+^\lay}{|B_1|}\rt) 
 +\int_{\partial B_1}\int_0^{1-\tau}\zeta f.
\end{multline}
  We make the ansatz $(\rho^{\bdr,\bulk},j^{\bdr,\bulk}):=(\rho^\bdr_1+\rho^\bdr_2, j^\bdr_1+j^\bdr_2)$ (see Figure \ref{fig:rhobulkbdr}) requiring that 
 \begin{figure}\begin{center}
 \resizebox{10.cm}{!}{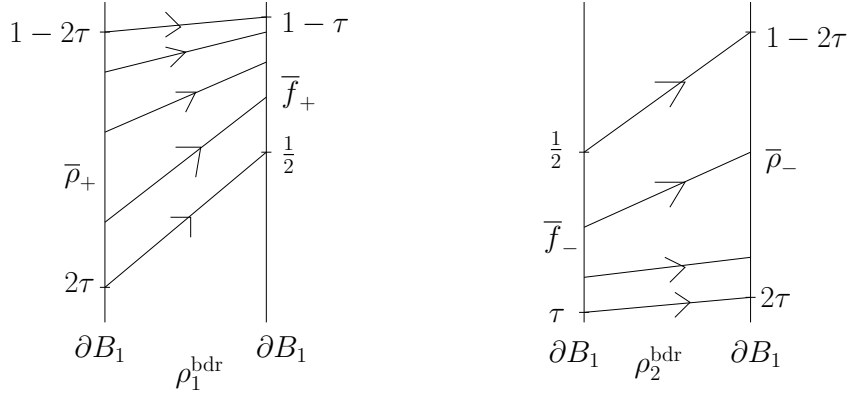} 
   \caption{The definition of $\rho^{\bdr,\bulk}$.} \label{fig:rhobulkbdr}
 \end{center}
 \end{figure}
\begin{equation}\label{rhobdr1}
\int_{\partial B_1}\int_0^{1} \partial_t \zeta \rho^{\textrm{bdr}}_1+ \nabbdr \zeta\cdot j^{\textrm{bdr}}_1=\int_{\partial B_1}\int_\tau^{1-\tau}  \zeta\lt( \chi_{(\frac{1}{2},1-\tau)} \frac{2}{1-2\tau}\Bf_+-\chi_{(2\tau,1-2\tau)}\frac{1}{1-4\tau}\Brhop\rt) 
\end{equation}
and
\begin{equation}\label{rhobdr2}
\int_{\partial B_1}\int_0^{1} \partial_t \zeta \rho^{\textrm{bdr}}_2+ \nabbdr \zeta\cdot j^{\textrm{bdr}}_2=\int_{\partial B_1}\int_\tau^{1-\tau}\zeta \lt(\chi_{(2\tau,1-2\tau)}\frac{1}{1-4\tau}\Brhom- \chi_{(\tau,\frac{1}{2})}\frac{2}{1-2\tau} \Bf_-\rt) ,
\end{equation}
so that by definition \eqref{def:fhat} of $\hat f$, \eqref{conteqbdrbulk} holds. Let $(\rho^{\textrm{bdr}}_1,j^\bdr_1)$ be given by Lemma \ref{lem:defboundaryflux} for 
\[
 f_1=\Brhop,  \quad f_2= \Bf_+, \quad a=2\tau, \quad b=1-2\tau, \quad c=\frac{1}{2},\quad \textrm{and } \quad d=1-\tau 
\]
so that 
\[
 \frac{1}{(d-b)-(c-a)}\log \frac{d-b}{c-a}= -\frac{2}{1-6\tau}\log \frac{2\tau}{1-4\tau}\les |\log \tau|.
\]
Thanks to \eqref{conteqbdrconstruct}, we have \eqref{rhobdr1} and by \eqref{estimW2construct} combined with \eqref{L2boundrho}, we have 
\begin{equation}\label{estimW2rhobdr1}
 \int_{\partial B_1}\int_0^1 \frac{1}{\rho^\bdr_1}|j^\bdr_1|^2\les |\log \tau| E^{\frac{d+3}{d+2}}.
\end{equation}
Similarly, using Lemma \ref{lem:defboundaryflux} with 
\[
 f_1=\Bf_-,  \quad f_2= \Brhom, \quad a=\tau, \quad b=\frac{1}{2}, \quad c=2\tau,\quad \textrm{and } \quad d=1-2\tau 
\]
to define $(\rho^\bdr_2,j^\bdr_2)$, we obtain that \eqref{rhobdr2} holds and that \eqref{estimW2rhobdr1} is also satisfies by $ (\rho^\bdr_2,j^\bdr_2)$. By subadditivity this proves \eqref{step4.2}
\\

\medskip
{\it Step 3.} We now define and estimate the quantities related to the terminal layer (in time). In this step we deal with the construction in the time interval $[0,1-\tau]$  (see Figure \ref{fig:rholay1}) and  
define $(\rho^{\bdr,\lay},j^{\bdr,\lay})$ supported  $\partial B_1\times [0,1-\tau]$ and $(\rho^\lay,j^\lay)$ supported in $B_1\times[0,1-\tau]$ such that (recall the definition \eqref{defBflay+} of $\Bf_+^\lay$)
\begin{multline}\label{conteqlay1}
  \int_{B_1}\int_0^{1-\tau}\partial_t \zeta \rho^{\lay}+ \nabla \zeta \cdot j^{\lay} +\int_{\partial B_1}\int_0^{1-\tau} \partial_t \zeta \rho^{\bdr,\lay}+ \nabbdr \zeta \cdot j^{\bdr,\lay}\\
  = \int_{\partial B_1} \zeta_{1-\tau} \Bf^\lay_+-\int_{B_1} \zeta_0 \frac{m_+^\lay}{|B_1|}, 
\end{multline}
and 
\begin{equation}\label{Step4.3estim}
\int_{B_1}\int_{0}^{1-\tau} |j^\lay|^2+\int_{\partial B_1}\int_0^{1-\tau} \frac{1}{\rho^{\bdr,\lay}}|j^{\bdr,\lay}|^2
 \les \tau^2 E+D.
 \end{equation}
 Let $\phi^\lay$ be the solution of 
  \[\begin{cases}
   \Delta \phi^\lay =\frac{1}{|B_1|}\int_{\partial B_1} \Brhop^\lay=\frac{m_+^\lay}{|B_1|} &\textrm{in } B_1\\[8pt]
   \frac{\partial \phi^\lay}{\partial \nu}= \Brhop^\lay &\textrm{on } \partial B_1,
  \end{cases}\]
  with $\int_{\partial B_1} \phi^\lay=0$ (recall that $\Brhop^\lay$ was defined in Lemma \ref{goodR}). By \eqref{CalZyg} and H\"older's inequality combined with  \eqref{L2boundrholay},
  \begin{equation}\label{eq:estimenerphilay}
   \int_{B_1} |\nabla \phi^\lay|^2\les \tau^2 E +D.
  \end{equation}
We then let 
 \begin{figure}\begin{center}
 \resizebox{13.cm}{!}{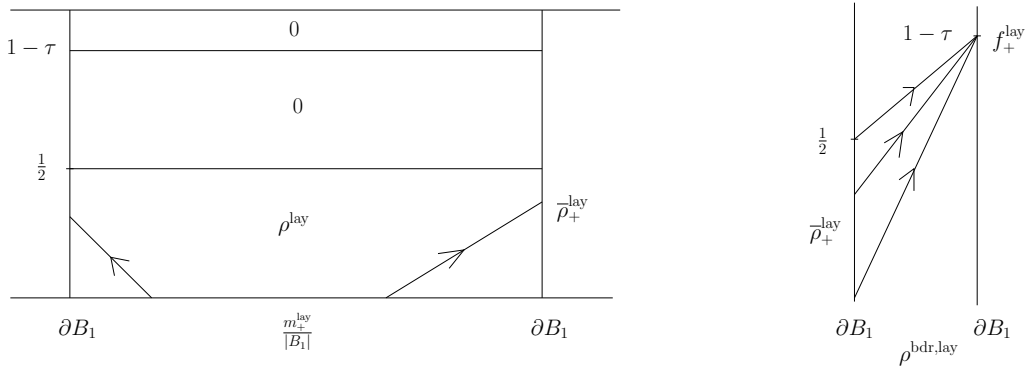} 
   \caption{The definition of $\rho^{\lay}$ and $\rho^{\bdr,\lay}$.} \label{fig:rholay1}
 \end{center}
 \end{figure}
\[
 \rho^\lay:= (1-2t)\frac{m_+^\lay}{|B_1|} \qquad \textrm{and} \qquad j^\lay:= 2 \nabla \phi^\lay \qquad \textrm{for } \quad t\in (0,\frac{1}{2}),
\]
and extend them by zero for $t\in(\frac{1}{2},1-\tau)$.
Note that  \eqref{boundsmhatf1} automatically implies  the smallness hypothesis \eqref{Linftyconstraintrholay}. In view of the boundary value problem defining $\phi^\lay$ we have 
\begin{equation}\label{conteqlaybulk}
 \int_{B_1}\int_0^{1-\tau}\partial_t \zeta \rho^{\lay}+ \nabla \zeta \cdot j^{\lay}=\int_{\partial B_1}\int_0^{\frac12} 2\zeta \Brhop^\lay -\int_{B_1} \zeta_0 \frac{m^\lay_+}{|B_1|},  
\end{equation}
and  \eqref{eq:estimenerphilay} translates into
\begin{equation}\label{step 4.3lay}
 \int_{B_1}\int_{0}^{1-\tau} |j^\lay|^2 \les \tau^2 E + D.
\end{equation}
Let $(\rho^{\bdr,\lay},j^{\bdr,\lay})$ be defined by Lemma \ref{lem:defboundaryflux} with
\[
 f_1=\Brhop^\lay,  \quad f_2= \Bf_+^\lay, \quad a=0, \quad b=\frac{1}{2}, \quad  \textrm{and } \quad c=d=1-\tau 
\]
so that 
\[
 \frac{1}{(d-b)-(c-a)}\log \frac{d-b}{c-a}=\frac{1}{2}\log\frac{1-2\tau}{2(1-\tau)}\les 1.
\]
For these choices, \eqref{conteqbdrconstruct} turns into 
\begin{equation*}
 \int_{\partial B_1}\int_0^{1-\tau} \partial_t \zeta \rho^{\bdr,\lay}+\nabla \zeta\cdot j^{\bdr,\lay}=\int_{\partial B_1} \zeta_{1-\tau} \Bf_+^\lay-\int_{\partial B_1}\int_0^{\frac12} 2\zeta \Brhop^\lay 
\end{equation*}
so that combining with \eqref{conteqlaybulk} we find \eqref{conteqlaymoins}. By \eqref{estimW2construct} and \eqref{L2boundrholay} we also obtain
\begin{equation*}
 \int_{\partial B_1}\int_0^{1-\tau} \frac{1}{\rho^{\bdr,\lay}}|j^{\bdr,\lay}|^2\les \tau^3 E^{\frac{d+3}{d+2}}+ \tau E^{\frac{1}{d+2}} D,
\end{equation*}
which combined with \eqref{step 4.3lay} and the fact that $\tau^3 E^{\frac{d+3}{d+2}}+ \tau E^{\frac{1}{d+2}} D\les \tau^2E +D$ gives \eqref{Step4.3estim}.\\

\medskip
{\it Step 4.}
We are left with the construction in $[1-\tau,1]$. We define $(\rho^{\bdr,\lay},j^{\bdr,\lay})$ supported  $\partial B_1\times [1-\tau,1]$ and $(\rho^\lay,j^\lay)$ supported in $B_1\times[1-\tau,1]$ such that
\begin{multline}\label{conteq4.4}
 \int_{B_1}\int_{1-\tau}^1 \partial_t \zeta \rho^\lay+ \nabla \zeta \cdot j^\lay +\int_{\partial B_1}\int_{1-\tau}^1 \partial_t \zeta \rho^{\bdr,\lay}+ \nabbdr \zeta \cdot j^{\bdr,\lay} \\
 =\int_{B_1} \zeta_1 d\mu -\int_{B_1}\zeta_{1-\tau}\lt( 1-\frac{m_+^\lay}{|B_1|}-\frac{1}{|B_1|}\int_{\partial B_1} (\Brhop-\Brhom)\rt) -\int_{\partial B_1} \zeta_{1-\tau} \Bf^\lay_+ +\int_{\partial B_1} \int_{1-\tau}^1 \zeta f 
\end{multline}
and
\begin{equation}\label{estim4.4}
 \int_{B_1}\int_{1-\tau}^1 \frac{1}{\rho^{\lay}}|j^\lay|^2+ \int_{\partial B_1}\int_{1-\tau}^1 \frac{1}{\rho^{\bdr,\lay}}|j^{\bdr,\lay}|^2\les \tau E +\tau^{-1} D.
\end{equation}
Note that combining \eqref{conteqlay1} and  \eqref{conteq4.4}, we get
\begin{multline}\label{conteqlayfinal}
 \int_{B_1}\int_0^1  \partial_t \zeta \rho^{\lay}+\nabla \zeta \cdot j^\lay+\int_{\partial B_1}\int_0^1 \partial_t \zeta \rho^{\bdr,\lay}+\nabbdr \zeta \cdot j^{\bdr,\lay}=\int_{\partial B_1} \int_{1-\tau}^1 \zeta f \\
 +\int_{B_1} \zeta_1 d\mu -\int_{B_1}\zeta_{1-\tau}\lt(1-\frac{m_+^\lay}{|B_1|}-\frac{1}{|B_1|}\int_{\partial B_1} (\Brhop-\Brhom)\rt) -\zeta_0\frac{m_+^\lay}{|B_1|}. 
\end{multline}

The construction of $(\rho^{\bdr,\lay},j^{\bdr,\lay})$  takes care of the outgoing flux $f_+^\lay$ (recall \eqref{defflay+}) in $[1-\tau,1]$ by defining 
\[
 \rho^{\bdr,\lay}:= \int_t^1 f^\lay_+ \qquad \textrm{and} \qquad j^{\bdr,\lay}:=0 \qquad \textrm{on } \quad [1-\tau,1],
\]
and thus at no cost.
The construction of $(\rho^\lay,j^\lay)$ for $t\in (1-\tau,1)$ is done in several steps (see Figure \ref{fig:rholay2}).
\begin{figure}\begin{center}
 \resizebox{10.cm}{!}{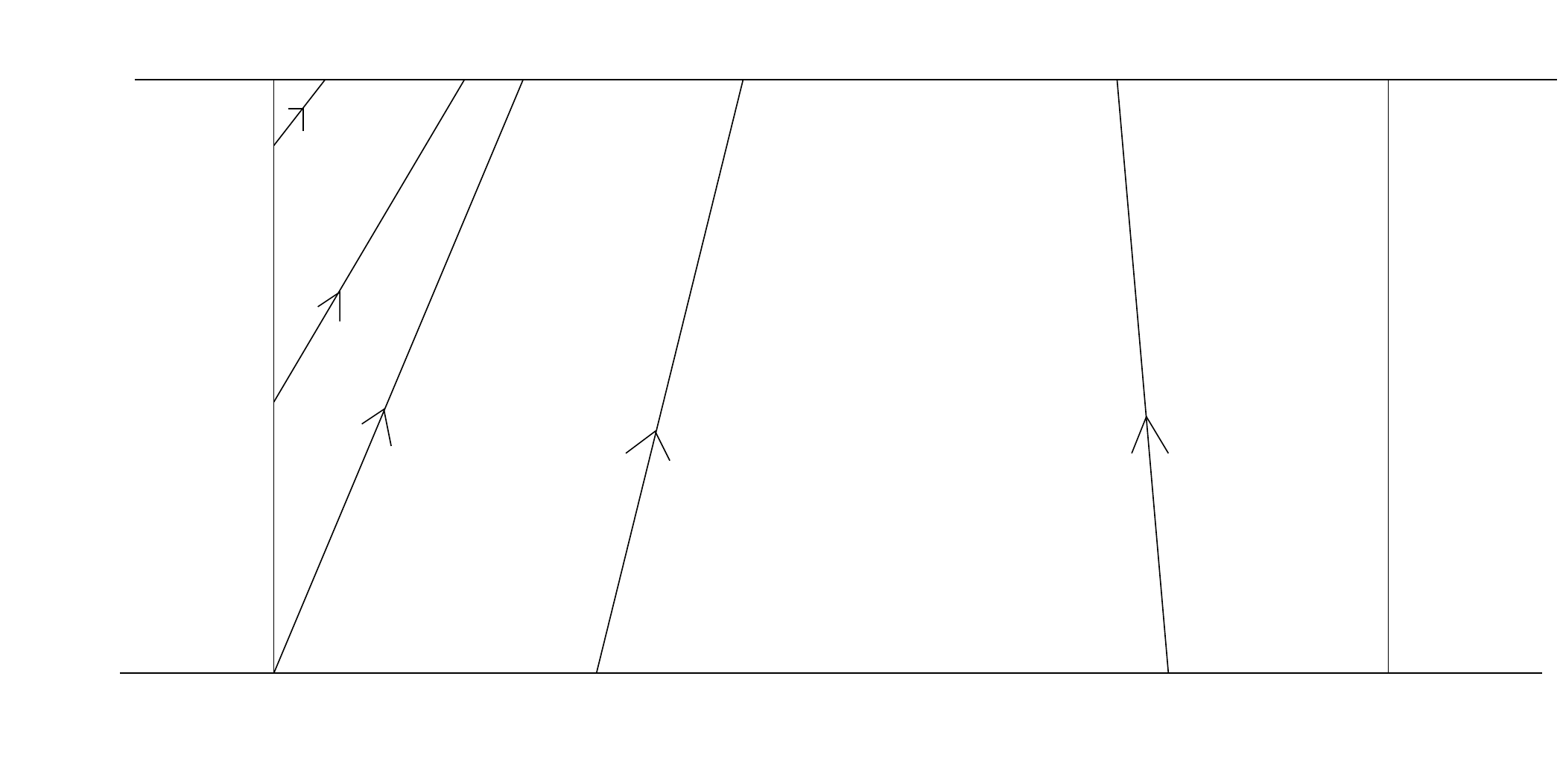} 
   \caption{The definition of $\rho^{\lay}_-$ and $\rho^{\lay}_{\textrm{in}}$. } \label{fig:rholay2}
 \end{center}
 \end{figure}
 We first take care of the incoming flow $f_-$ (recall \eqref{def:fpm}) in $[1-\tau,1]$  and to this purpose take the corresponding bulk density from $X$ itself,
\[
 \int_{B_1} \int_{1-\tau}^1 \zeta d\rho^\lay_-:=\int_{\Omega} \chi_{1-\tau<t_-<t_+\le 1}(X) \int_{t_-}^{t_+} \zeta(X,t) 
\]
and bulk flux
\[
 \int_{B_1} \int_{1-\tau}^1 \xi\cdot dj^\lay_-:=\int_{\Omega} \chi_{1-\tau<t_-<t_+\le 1}(X) \int_{t_-}^{t_+} \xi(X,t)\cdot \dot{X}.
\]
With these definitions, it is readily seen that $j^\lay_-\ll \rho^\lay_-$. Let $\mu_-:= \rho_-^\lay (\cdot,1)$  and denote by $f_+^{\textrm{thr}}$ the flux coming from particles that enter and leave $B_1$ during $(1-\tau,1)$, i.e. the particles passing through $B_1$, 
\[
\int_{\partial B_1}\int_{1-\tau}^1 \zeta df_+^{\textrm{thr}}:=\int_{\Omega}\chi_{1-\tau\le t_-<t_+\le 1}(X)\zeta(X(t_+), t_+), 
\]
so that $f_+=f_+^\lay+f_+^{\textrm{thr}}$ on $\partial B_1\times [1-\tau,1]$ (recall the definitions \eqref{def:fpm} and \eqref{defflay+}).
By the same argument that led to  \eqref{conteqloc} we have
\begin{equation*}
 \int_{B_1}\int_{1-\tau}^1 \partial_t \zeta \rho^\lay_-+ \nabla \zeta \cdot j_-^\lay =\int_{B_1} \zeta_1 d\mu_- +\int_{\partial B_1} \int_{1-\tau}^1 \zeta (f_+^{\textrm{thr}}-f_-),
\end{equation*}
so that 
\begin{multline}\label{conteqlaymoins}
 \int_{B_1}\int_{1-\tau}^1 \partial_t \zeta \rho^\lay_-+ \nabla \zeta \cdot j_-^\lay +\int_{\partial B_1}\int_{1-\tau}^1 \partial_t \zeta \rho^{\bdr,\lay}+ \nabbdr \zeta \cdot j^{\bdr,\lay} \\
 =\int_{B_1} \zeta_1 d\mu_- -\int_{\partial B_1} \zeta_{1-\tau} \Bf^\lay_+ +\int_{\partial B_1} \int_{1-\tau}^1 \zeta f.
\end{multline}
Furthermore, by \eqref{dualBB}
\begin{align*}
 \int_{B_1}\int_{1-\tau}^1 \frac{1}{2\rho^\lay_-}|j^\lay_-|^2&=\sup_{\xi\in C^0(\overline{B}_1\times[1-\tau,1],\R^d)} \int_{B_1}\int_{1-\tau}^1 \xi\cdot j^{\lay_-}-\frac{|\xi|^2}{2} \rho^\lay_- \\
 &=\sup_{\xi\in C^0(\overline{B}_1\times[1-\tau,1],\R^d)} \int_{\Omega}  \chi_{1-\tau<t_-<t_+\le 1}(X)\int_{t_-}^{t_+} \xi(X,t) \cdot \dot{X} -\frac{|\xi|^2}{2}\\
 &\le \int_{\Omega}  \chi_{1-\tau<t_-<t_+\le 1}(X)\int_{t_-}^{t_+}\frac{1}{2} |\dot{X}|^2.
\end{align*}
Let us point out that using an approximation argument one could actually show that equality holds in the previous inequality. Using that the trajectories $X$ are straight lines, we have for $1-\tau<t_-<t_+\le 1$,  
\[\int_{t_-}^{t_+} |\dot{X}|^2=(t_+-t_-) |X(1)-X(0)|^2\le \tau |X(1)-X(0)|^2\]
so that 
\begin{multline}\label{eq:estimenerrhomoins}
 \int_{B_1}\int_{1-\tau}^1 \frac{1}{\rho^\lay_-}|j^\lay_-|^2\le \tau \int_{\Omega}  \chi_{1-\tau<t_-<t_+\le 1}(X)|X(1)-X(0)|^2\\
 \stackrel{\eqref{LinftyboundT}}{\le}\tau\int_{\Omega} \chi_{|X(0)|\le 2}(X) |X(1)-X(0)|^2=\tau E.
\end{multline}

It remains to connect  $\rho_{1-\tau}^\bulk=1-\frac{m_+^\lay}{|B_1|}-\frac{1}{|B_1|}\int_{\partial B_1} (\Brhop-\Brhom)=: \Lambda$, cf. \eqref{conteqbulk} to   
 what remains from the target measure $\mu$ after subtracting  $\rho^\lay_-(\cdot,1)=\mu_-$ i.e. we need to connect $\Lambda$ to $\mu':=\mu-\mu_-$. Since $\mu'$ coincides with the measure defined in \eqref{defmu'}, by \eqref{distdata}, 
\[
 W_2^2(\Lambda,\mu')\les\tau^2 E+D.
\]
We can thus use the Benamou-Brenier formulation of optimal transport  \eqref{BBwass} to  find $(\rho_{\textrm{in}}^\lay,j_{\textrm{in}}^\lay)$ rescaled from $[0,1]$ to $[1-\tau,1]$ and
such that, 
\begin{equation}\label{conteqlay2}
 \int_{B_1}\int^1_{1-\tau} \partial_t \zeta \rho_{\textrm{in}}^\lay+ \nabla \zeta \cdot j_{\textrm{in}}^\lay=\int_{B_1} \zeta_1 d\mu'-\int_{B_1}\zeta_{1-\tau}\lt( 1-\frac{m_+^\lay}{|B_1|}-\frac{1}{|B_1|}\int_{\partial B_1} (\Brhop-\Brhom)\rt) 
\end{equation}
and
\begin{equation}\label{eq:estimenerrhomoins2}
 \int_{B_1} \int_{1-\tau}^1 \frac{1}{\rho_{\textrm{in}}^\lay}| j_{\textrm{in}}^\lay|^2\les \tau E +\tau^{-1} D.
\end{equation}
We thus let for $t\in [1-\tau,1]$, $\rho^\lay:= \rho^\lay_-+\rho_{\textrm{in}}^\lay$ and $j^\lay:=j^\lay_-+j_{\textrm{in}}^\lay$. Combining \eqref{conteqlaymoins} and \eqref{conteqlay2} we obtain \eqref{conteq4.4}.
Moreover, using the subadditivity of $\int \frac{1}{\rho}|j|^2$,  \eqref{eq:estimenerrhomoins}, \eqref{eq:estimenerrhomoins2} and the fact that in $[1-\tau,1]$, $j^{\bdr,\lay}=0$ we conclude the proof of \eqref{estim4.4}.

\medskip

{\it Step 5.} Combining \eqref{conteqbulk2}  and \eqref{conteqlayfinal}, we see that \eqref{eq:conteqtildrho} holds. Plugging \eqref{eq:estimbulk}, \eqref{eq:estimjbulk},
\eqref{step4.2}, \eqref{Step4.3estim} and \eqref{estim4.4}, into \eqref{splitenergy}, we find

\begin{align*}
\int_{\R^d} \int_0^{1} \frac{1}{\tilde{\rho}}|\tilde{j}|^2- \int_{B_1} |\nabla \phi|^2&\les \tau E+(\tau^{-d} E)^{\frac{d+2}{d+1}} +D+\lt( E +(\tau^{-d} E)^{\frac{d+2}{d+1}}\rt)^{\frac12} \lt(\tau^2 E+D\rt)^{\frac12} \\
 &\qquad +\tau^2 E+D +|\log \tau| E^{\frac{d+3}{d+2}}  +\tau E +\tau^{-1} D\\
 &\les \tau E +(\tau^{-d} E)^{\frac{d+2}{d+1}}+|\log \tau| E^{\frac{d+3}{d+2}}+ \tau^{-1}D,
\end{align*}
where we used Young's inequality together with the fact that $\tau\ll1$ and $E+D\ll1$. Since $ (\tau^{-d} E)^{\frac{d+2}{d+1}}+|\log \tau| E^{\frac{d+3}{d+2}}$ is super-linear in $E$, 
there exists $0<\eps(\tau)\ll1$ such that if $E\le \eps(\tau)$, 
\[
 (\tau^{-d} E)^{\frac{d+2}{d+1}}+|\log \tau| E^{\frac{d+3}{d+2}}\les \tau E.
\]
Therefore, if $E+D\le \eps(\tau)$,
\[
 \int_{\R^d}\int_0^{1} \frac{1}{\tilde{\rho}}|\tilde{j}|^2- \int_{B_1} |\nabla \phi|^2\les \tau E +\tau^{-1}D,
\]
which  together with \eqref{eq:locmin2} proves \eqref{eq:distjphiprop2}.
\end{proof}
Combining \eqref{eq:distjphiprop} and \eqref{eq:distjphiprop2}, we obtain our main estimate,  Proposition \ref{prop:intromaineulerian} which we now recall for the reader's convenience.
\begin{proposition*}
For every $0<\tau\ll 1$, there exist positive  constants  $\eps(\tau)$ and  $C(\tau)$ such that if $E+D\ll \eps(\tau)$, then  letting $\phi$ be defined via \eqref{def:phi} we have 
 \begin{equation}\label{eq:distjphiprop3}
  \int_{B_{\frac{1}{2}}}\int_0^1\frac{1}{\rho}|j-\rho\nabla \phi|^2\les \tau E+ C(\tau) D.
 \end{equation}

\end{proposition*}
Arguing exactly as in \cite[Proposition 4.6]{GO}, using the Benamou-Brenier formula \eqref{BBwass}, Lemma \ref{lem:Linftybound} and the harmonicity of $\nabla \phi$ (where $\phi$ is defined in \eqref{def:phi}), 
this result can be translated into Lagrangian terms, which gives Proposition \ref{prop:detintro}, i.e.
\begin{proposition*}
 For every $0<\tau\ll 1$, there exist positive  constants $\eps(\tau)$ and  $C(\tau)$ such that if $E+D\le \eps(\tau)$, then there exists a function $\phi$ with harmonic gradient in $B_{\frac{1}{4}}$ and such that 
 \begin{equation}\label{eq:distTphilag}
  \int_{B_{\frac{1}{4}}}|T-(x+\nabla \phi)|^2\les \tau E+ C(\tau) D
 \end{equation}
and 
\begin{equation}\label{estimphi}
  \int_{B_{\frac{1}{4}}}|\nabla \phi|^2\les E.
\end{equation}

\end{proposition*}
With this estimate at hand, we can now prove as in \cite[Proposition 4.7]{GO} one step of a Campanato iteration (recall that $E(\mu,T,R)$ and $D(\mu,O,R)$ are defined in \eqref{def:E} and \eqref{def:D}), i.e.\ Theorem \ref{theo:det intro} which we now recall.
 
\begin{theorem*}
 For every $0<\tau\ll 1$, there exist positive  constants $\eps(\tau)$,  $C(\tau)$ and $\theta>0$ such that if $E(\mu,T,R)+D(\mu,O,R)\le \eps(\tau)$, then there exists a symmetric matrix $B$ and a vector $b\in \R^d$ such that 
 \begin{equation}\label{eq:estimBb}
  |B-Id|^2+\frac{1}{R^2}|b|^2\les E(\mu,T,R),
 \end{equation}
and letting $\hat x:=B^{-1} x$, $\hat \Omega:= B^{-1} \Omega$ and then 
\begin{equation}\label{def:hat}
 \hat{T}(\hat x):=B(T(x)-b) \qquad \textrm{and} \qquad \hat \mu:=\hat{T}\# \chi_{\hat \Omega} d \hat{x},
\end{equation}
we have 
\begin{equation}\label{eq:mainonestep}
 E(\hat \mu,\hat T,\theta R)\le \tau E(\mu,T,R)+C(\tau) D(\mu,O,R).
\end{equation}

\end{theorem*}
\begin{proof}
 The proof is analogous to the one of \cite[Proposition 4.7]{GO} with minor modifications. Still, we give the proof for the reader's convenience. By rescaling, we may assume that $R=1$
 and we then recall that $E=E(\mu,T,1)$ and $D=D(\mu,O,1)$. Let $\tau'$ to be fixed later on and let then $\phi$ be given by Proposition \ref{prop:detintro} for $\tau'$. We define $b:=\nabla \phi(0)$, $A:=\nabla^2\phi(0)$ 
 and then $B:=e^{-\frac{A}{2}}$ so that $B$ is symmetric. Since $\nabla \phi$ is harmonic, we obtain from \eqref{estimphi} and the mean value formula that \eqref{eq:estimBb} holds.  Defining $\hat T$ and $\hat \mu$ as in \eqref{def:hat},
 we get 
 \begin{align*}
  E(\hat \mu, \hat T, \theta )&=\frac{1}{\theta^{d+2}}\int_{B(B_{2\theta})} |\det B|^{-1}|B(T-b)-B^{-1} x|^2\\
  &\stackrel{\eqref{eq:estimBb}}{\les} \frac{1}{\theta^{d+2}}\int_{B_{4\theta}}|T-b-B^{-2} x|^2\\
  &\les  \frac{1}{\theta^{d+2}} \int_{B_{4\theta}}|T-(x+\nabla \phi)|^2 +\frac{1}{\theta^{d+2}} \int_{B_{4\theta}}|(B^{-2}-Id-A)x|^2 \\
  &\qquad +\frac{1}{\theta^{d+2}} \int_{B_{4\theta}}|\nabla \phi-b-Ax|^2\\
  &\les  \frac{1}{\theta^{d+2}} \int_{B_{4\theta}}|T-(x+\nabla \phi)|^2 +|B^{-2}-Id-A|^2  +\theta^{-2} \sup_{B_{4\theta}} |\nabla \phi-b-Ax|^2.
 \end{align*}
Recalling that $B=e^{-\frac{A}{2}}$, $b=\nabla \phi(0)$ and $A=\nabla^2 \phi(0)$, we conclude using again the mean value formula for $\nabla \phi$ 
\begin{align*}
 E(\hat \mu, \hat T, \theta )&\stackrel{\eqref{eq:distTphilag}}{\les} \theta^{-(d+2)}(\tau' E+C(\tau') D) +|\nabla^2 \phi(0)|^4+ \theta^2\sup_{B_{4\theta}} |\nabla^3 \phi|^2\\
 &\stackrel{\eqref{estimphi}}{\le} C\lt(  \theta^{-(d+2)}(\tau' E+C(\tau') D) +E^2 +\theta^2 E\rt).
\end{align*}
Choosing first $\theta$ small enough so that $C(E+\theta^2) \le\frac{\tau}{2}$ and then $\tau'$ small enough so that $C\theta^{-(d+2)} \tau'\le \frac{\tau}{2}$, we see that we can guarantee that \eqref{eq:mainonestep} is satisfied. 
\end{proof}

\section{Application to the optimal matching problem}\label{sec:application}
\subsection{Quantitative bounds on $T$}
We now turn back to the optimal matching problem and combine Theorem \ref{theo:estimDrstarper} and Theorem \ref{theo:det intro} to obtain the desired quantitative estimate on the transport map. Let us recall that we work here in dimension $d=2$. \\
Let us recall 
that for for every dyadic $L$,   we consider  $\mu$  a realization of
the $Q_L-$periodic Poisson point process (see Section \ref{sec:Poi}) and let $T=T_{\mu,L}$ be the optimal transport map between $\frac{\mu(Q_L)}{L^2}$ and $\mu$  for the periodic transport problem \eqref{perwass}, i.e.\ a minimizer of  $\Wper(\frac{\mu(Q_L)}{L^2},\mu)$. 
By Theorem \ref{theo:estimDrstarper}, there exist a constant $c>0$ and  a random variable $r_{*,L}$  such that  
$\sup_{L} \EE_L\sqa{\exp\lt(\frac{c r_{*,L}^2}{\log  2 r_{*,L}}\rt)}<\infty$ and such that  if $2 r_{*,L}\le L$, recalling \eqref{eq:goodboundmu},
\begin{equation}\label{eq:goodboundmu2}
  \frac{\mu(Q_L)}{L^2}\in \lt[\frac{1}{2},\, 2\rt], \qquad \spt \mu \cap B_{r_{*,L}}\neq \emptyset.
\end{equation}
 For $\mu$ such that $2 r_{*,L}\le L$, we  let $\hat{\mu}:= \frac{L^2}{\mu(Q_L)}\mu$ so that $T$ is also the optimal transport map (for the periodic transport problem \eqref{perwass}) between the Lebesgue measure and $\hat{\mu}$.
 By \eqref{eq:concentrationrstarper} of Theorem \ref{theo:estimDrstarper}, we have that for all dyadic $\ell$ with
 $2 r_{*,L} \le \ell \le L$,

 \begin{equation}\label{bounddata2}
 \frac{1}{\ell^4} W_2^2\lt(\hat{\mu}_{\ell}, \frac{\hat{\mu}(Q_{\ell})}{\ell^2}\rt)\les \frac{\log\lt(\frac{\ell}{r_{*,L}}\rt)}{\lt(\frac{\ell}{r_{*,L}}\rt)^2}
\end{equation}
so that  by \eqref{eq:concentrationrstarper} also
\begin{equation}\label{boundEinit2}
 \frac{1}{L^4} \Wper\lt(\hat{\mu}, 1\rt)\les \frac{\log\lt(\frac{L}{r_{*,L}}\rt)}{\lt(\frac{L}{r_{*,L}}\rt)^2}.
\end{equation}
Let us  recall (see Section \ref{sec:optimaltransp} and in particular Remark \ref{rem:selection} for more details) that  by \cite{cordero}, there exists a  convex function $\psi$ on $\R^2$ such that the map $T$ can be identified on $\R^2$ with a measurable selection of the subgradient $\partial \psi$ of $\psi$.

 Let   $\yL=\yL(\mu):=\argmin_{\spt \mu} |y|$ (which is uniquely defined $\PP_L-$a.e.) and $x_L$ the barycenter of its pre-image under $T$, i.e.
\[
 \xL=\xL(\mu):=\frac{1}{|T^{-1}(\yL)|}\int_{T^{-1}(\yL)} x dx,
\]
so that the map $\mu\to \xL$ is $\PP_L-$measurable. Let us show that  $T(x_L)=y_L$. 
By convexity of $\psi$ the set $(\partial \psi)^{-1}(y_L)=\partial \psi^*(y_L)$ is convex. Since $\psi$ is differentiable a.e.,
we have $|(\partial \psi)^{-1}(y_L)\backslash T^{-1}(y_L)|=0$ and $|(\partial \psi)^{-1}(y_L)|=\mu(y_L)>0$ so that $x_L$ lies in the interior of $(\partial \psi)^{-1}(y_L)$.
Since $\nabla \psi= y_L$ a.e. on this set, $\psi$ is affine inside $(\partial \psi)^{-1}(y_L)$ and thus $\psi$ is differentiable at $x_L$ with $T(x_L)=\nabla \psi(x_L)=y_L$.
Therefore, by the definition of $\yL$ and \eqref{eq:goodboundmu2} we  have for $\mu$ such that $2 r_{*,L}\le L$,  
\begin{equation}\label{estimx0}
 T(\xL)=\yL \qquad \textrm{and} \qquad |\yL|\le r_{*,L}.
\end{equation}

Finally, we can prove our main result, Theorem \ref{theo:mainresult intro} which we recall for the convenience of the reader.
\begin{theorem*}
 Let $L\gg 1$ be dyadic and $\mu$ be a $Q_L-$periodic Poisson point process. Then, if $\mu$ is such that $ r_{*,L}\ll L$,
\begin{equation}\label{estimxomega}
 |\xL|^2\les r_{*,L}^2 \log^3\lt(\frac{L}{r_{*,L}}\rt)
\end{equation}
and for every  $2 r_{*,L}\le \ell\le L$,
 \begin{equation}\label{eq:estimT}
  \frac{1}{\ell^4} \int_{B_\ell(\xL)} |T-(x-\xL)|^2\les \frac{\log^3\lt(\frac{\ell}{r_{*,L}}\rt)}{\lt(\frac{\ell}{r_{*,L}}\rt)^2}.
 \end{equation}

\end{theorem*}
\begin{proof}
{\it Step 1.}[The setup]
By periodicity we have 
\begin{equation}\label{eq:defE}
 E:=\frac{1}{L^4} \int_{B_L} |T-x|^2\les \frac{1}{L^4} \int_{Q_{L}}|T-x|^2=\frac{1}{L^4} \Wper\lt(\hat{\mu}, 1\rt)\stackrel{\eqref{boundEinit2}}{\les} \frac{\log\lt(\frac{L}{r_{*,L}}\rt)}{\lt(\frac{L}{r_{*,L}}\rt)^2}.
\end{equation}
Let  $\tilde \mu:= \nabla \psi \# (\chi_{B_L} dx)$ so that $T$  is the Euclidean optimal transport map between $\chi_{B_L} dx$ and $\tilde \mu$. By the $L^\infty$ bound \eqref{LinftyboundT}, we have that $\tilde \mu=\hat{\mu}$ on $Q_{L}$.
Let finally 
\begin{equation}\label{defT0}
 T_0(x):=T(\xL+x)-\xL \qquad \textrm{and }\qquad \mu_0:=\tilde \mu (\cdot+\xL)
\end{equation}
so that \eqref{estimx0} becomes
\begin{equation}\label{estimx0hat}
 T_0(0)=\yL-\xL \qquad \textrm{with} \qquad |\yL|\le r_{*,L}.
\end{equation}
Denote by $\ell_0$ the largest dyadic $\ell$ such that $B_{2\ell}(\xL)\subset Q_L$. Fix $0<\tau\ll 1$ for which Theorem \ref{theo:det intro} applies. By \eqref{eq:defE} and the $L^\infty$ bound \eqref{LinftyboundT}, $|\xL|\les E^{\frac{1}{4}} L \ll L$ so that $\ell_0\sim L$. 
By \eqref{eq:defE} we thus have  that (recall \eqref{def:E} and \eqref{def:D})
\begin{multline*}
 E_0:=E(\mu_0,T_0,\ell_0)= \frac{1}{(2 \ell_0)^4} \int_{B_{2\ell_0}} |T_0-x|^2 \\
 \textrm{and } \quad D_0:= D(\mu_0,Q_L -\xL,\ell_0)= \frac{1}{\ell_0^4}  W_2^2\lt(\tilde \mu\restr Q_{L}, \frac{\tilde \mu(Q_{L})}{|Q_{L}|}\rt)
\end{multline*}
satisfy for $L$ large enough
\begin{equation}\label{estimE0}
 E_0+D_0\le \eps(\tau).
\end{equation}
 We also let 
 \[B_0:=Id, \qquad b_0:=0, \qquad  \Omega_0:=B_L-x_L \qquad \textrm{and } \qquad O_0:=Q_L-x_L.\]
Let $\theta>0$ be given by Theorem \ref{theo:det intro}. Without loss of generality we may assume that $\theta$ is dyadic i.e. $\theta=2^{-j}$ for some $j\in \N$. For $k\ge 1$, let $\ell_k:=\theta^k \ell_0$ and notice that $\ell_k$ is also  dyadic since $\ell_0$ is.
It is of course enough to show that  \eqref{eq:estimT} holds for  $\ell=\ell_k$. We now prove by induction that there exist $C_*, C_0, C_1, C_2>0$ sufficiently large but universal such that for every $k\ge 0$ such that $\ell_k\ge C_* r_{*,L}$,
we can find a symmetric matrix $B_k$ and a vector $b_k$  such that 
\begin{equation}\label{eq:induction1}
 |B_k-Id|^2+\frac{1}{\ell_{k}^2}|b_k|^2\le C_0 \frac{\log \lt(\frac{\ell_{k}}{r_{*,L}}\rt) }{\lt(\frac{\ell_k}{r_{*,L}}\rt)^2},  
\end{equation}
and letting $T_k(x):=B_k(T_{k-1}(B_k x)-b_k)$, $\Omega_k:=B_{k}^{-1} \Omega_{k-1}$ and $\mu_k:= T_k\# \chi_{\Omega_k}$ we have for $E_k:= E(\mu_k,T_k,\ell_k)$,
\begin{equation}\label{eq:induction2}
 E_k\le C_1  \frac{\log \lt(\frac{\ell_{k}}{r_{*,L}}\rt) }{\lt(\frac{\ell_k}{r_{*,L}}\rt)^2}, 
\end{equation}
and $T_k$ is the optimal transport map between $\chi_{\Omega_k}$ and $\mu_k$. Moreover, we may find a target set $O_k$ such that letting $D_k:= D(\mu_k,O_k, \ell_k)$, we have 
\begin{equation}\label{eq:inductionbis}
 B_{2\ell_k}\subset O_k \qquad \textrm{and} \qquad D_k\le C_2  \frac{\log \lt(\frac{\ell_{k}}{r_{*,L}}\rt) }{\lt(\frac{\ell_k}{r_{*,L}}\rt)^2}.
\end{equation}

As we shall argue, letting $A_0=Id$, $a_0=0$ and then for $k\ge 1$
\begin{equation}\label{def:Ak}
 A_k:= B_k A_{k-1} \qquad \textrm{and} \qquad a_k:= B_k a_{k-1} +B_k b_k,
\end{equation}
that is
\[
 A_k=B_k B_{k-1}\cdots B_0 \qquad \textrm{and} \qquad a_k= \sum_{i=0}^k B_k B_{k-1} \cdots B_i b_i,
\]
this entails 
\begin{equation}\label{Tk}
 T_k(x)=A_kT_0(A_k^* x)-a_k\stackrel{\eqref{defT0}}{=}A_kT(A_k^*x+x_L)-(A_kx_L+a_k), 
\end{equation}
 where $A^*$ denotes the transpose of $A$, 
\begin{equation}\label{eq:induction3}
 |A_k-Id|^2\les \frac{\log \lt(\frac{\ell_{k}}{r_{*,L}}\rt)}{\lt(\frac{\ell_k}{r_{*,L}}\rt)^{2}} \qquad \textrm{and }\qquad |a_k|^2\les k^2 r_{*,L}^2 \log\lt(\frac{L}{r_{*,L}}\rt).
\end{equation}
Notice that \eqref{eq:induction2} and \eqref{eq:inductionbis} in particular imply that if $\ell_k\ge C_* r_{*,L}$, then 
\begin{equation}\label{eq:hypsmallness}
 E_k+D_k\le \eps(\tau).
\end{equation}
\medskip

{\it Step 2.}[The iteration argument]
By \eqref{estimE0} the induction hypothesis is satisfied for $k=0$. Let us assume that it holds for $k-1$.\\
\smallskip

{\it Step 2.1.}[Proof of \eqref{eq:induction1} and \eqref{eq:induction2}]  Thanks to \eqref{eq:hypsmallness} we may apply  Theorem \ref{theo:det intro} with $R=\ell_{k-1}$ and $O=O_{k-1}$ (recall that we fixed $\tau\ll1$) to find a symmetric matrix $B_k$
and a vector $b_k\in \R^2$ such that 
\[
 |B_k-Id|^2+\frac{1}{\ell_{k}^2}|b_k|^2\le C E_{k-1}\stackrel{\eqref{eq:induction2}}{\le} C C_1 \frac{\log \lt(\frac{\ell_{k-1}}{r_{*,L}}\rt) }{\lt(\frac{\ell_{k-1}}{r_{*,L}}\rt)^2} \le C_0  \frac{\log \lt(\frac{\ell_{k}}{r_{*,L}}\rt)}{\lt(\frac{\ell_{k}}{r_{*,L}}\rt)^2} ,
\]
if $C_0$ is taken large enough (depending only on $C_1$). From this we see that  \eqref{eq:induction1} is satisfied. Moreover, by \eqref{eq:mainonestep}
\[
 E_k\le \tau E_{k-1} +C(\tau) D_{k-1}\stackrel{\eqref{eq:induction2}}{\le} C_1\lt( \tau +\frac{C(\tau)C_2 }{C_1}\rt) \frac{\log \lt(\frac{\ell_{k-1}}{r_{*,L}}\rt)}{\lt(\frac{\ell_{k-1}}{r_{*,L}}\rt)^2}. 
\]
Now if $C_1\ge \frac{1}{1-\tau} C(\tau) C_2$, since the function $f(t):= \frac{\log t}{t^2}$ is decreasing for $t$ large enough, if $\ell_k=\theta \ell_{k-1}\ge C_* r_{*,L}$ for some universal constant $C_*$ large enough then  $ f(\ell_k/r_{*,L})\ge f(\ell_{k-1}/r_{*,L})$ and thus
\[
 E_k\le  C_1 f\lt(\frac{\ell_{k-1}}{ r_{*,L}}\rt)\le C_1 f\lt(\frac{\ell_{k}}{ r_{*,L}}\rt)= C_1\frac{\log \lt(\frac{\ell_{k}}{r_{*,L}}\rt)}{\lt(\frac{\ell_{k}}{r_{*,L}}\rt)^2},
\]
which proves  \eqref{eq:induction2}.\\
\smallskip

{\it Step 2.2.}[Optimality of $T_k$]
Since $T_{k-1}$ is the optimal transport map between $\chi_{\Omega_{k-1}}$ and $\mu_{k-1}$, by Brenier's Theorem \cite[Theorem 2.12]{Viltop}, there exists a convex map $\psi_{k-1}$ such that $T_{k-1}=\nabla \psi_{k-1}$.
Then $T_k=\nabla \psi_k$ for the convex function $\psi_k(x):=\psi_{k-1}(B_k x)- b_k\cdot B_k x$ so that $T_k$ is the optimal transport map between $\chi_{\Omega_k}$ and $\mu_k$ (see \cite[Theorem 2.12]{Viltop}).\\
\smallskip

{\it Step 2.3.}[Derivation of \eqref{eq:induction3}]
 For $k>1$ and $i\le k$ we first prove that 
for $\ell_k/r_{*,L}\gg1$
\begin{equation}\label{toproveAk}
 \sum_{j=i}^k \lt(\frac{\log \lt(\frac{\ell_{j}}{r_{*,L}}\rt)}{\lt(\frac{\ell_j}{r_{*,L}}\rt)^2}\rt)^{\frac12}\les \frac{\log^{\frac12} \lt(\frac{\ell_{k}}{r_{*,L}}\rt)}{\frac{\ell_k}{r_{*,L}}}.
\end{equation}
Since 
the function $\frac{\log^{\frac12} t}{t^2}$ is decreasing for $t$ large enough, we obtain
\begin{align*}
  \sum_{j=i}^k \lt(\frac{\log \lt(\frac{\ell_{j}}{r_{*,L}}\rt) }{\lt(\frac{\ell_j}{r_{*,L}}\rt)^2}\rt)^{\frac12}&{\lesssim} \sum_{j=i}^k\int_{\frac{\ell_{j+1}}{r_{*,L}}}^{\frac{\ell_{j}}{r_{*,L}}} \frac{\log^{\frac12} t}{t^2}\\
  &=\int_{\frac{\ell_{k+1}}{r_{*,L}}}^{\frac{\ell_{i}}{r_{*,L}}} \frac{\log^{\frac12} t}{t^2}\\
  &\le \int_{\frac{\ell_{k+1}}{r_{*,L}}}^\infty \frac{\log^{\frac12} t}{t^2}\\
  &\les\frac{1+\log^{\frac12} \lt(\frac{\ell_{k+1}}{r_{*,L}}\rt)}{\frac{\ell_{k+1}}{r_{*,L}}}\les \frac{\log^{\frac12} \lt(\frac{\ell_{k}}{r_{*,L}}\rt)}{\frac{\ell_k}{r_{*,L}}},
\end{align*}
which proves \eqref{toproveAk}.\\
We can now make a downward induction on $i$ to show that \eqref{eq:induction1} implies that for $k>1$ and $i\le k$ 
\begin{equation}\label{eq:downwardinduction}
 |B_kB_{k-1}\cdots B_i-Id|\le 2 \sqrt{C_0} \sum_{j=i}^k \lt(\frac{\log \lt(\frac{\ell_{j}}{r_{*,L}}\rt)}{\lt(\frac{\ell_j}{r_{*,L}}\rt)^2}\rt)^{\frac12}
\end{equation}
which combined  with \eqref{toproveAk} implies
\begin{equation}\label{estimprodBk}
 |B_k B_{k-1}\cdots B_i-Id|^2\les \frac{\log \lt(\frac{\ell_{k}}{r_{*,L}}\rt)}{\lt(\frac{\ell_k}{r_{*,L}}\rt)^{2}}.
\end{equation}
 Notice that \eqref{estimprodBk} in particular gives $|B_k B_{k-1}\cdots B_i|\le 2$ provided we chose $C_*$ large enough. Estimate \eqref{estimprodBk} contains the first part of \eqref{eq:induction3} and the second part would also follow since for every $i$,
\[
 |B_k B_{k-1}\cdots B_i b_i|\stackrel{\eqref{estimprodBk}}{\le} 2  |b_i|\stackrel{\eqref{eq:induction1}}{\le} 2 \lt(C_0 {r_{*,L}^2} \log\lt(\frac{L}{r_{*,L}}\rt)\rt)^{\frac12}.
\]
 It thus remains to prove \eqref{eq:downwardinduction}  which clearly holds for $i=k$ by \eqref{eq:induction1}. Assume \eqref{eq:downwardinduction} holds for $i$. Then as already pointed out, \eqref{estimprodBk} implies $|B_kB_{k-1}\cdots B_i|\le 2 $ for $\frac{\ell_k}{r_{*,L}}$ large enough so that we can estimate
\begin{align*}
|B_kB_{k-1}\cdots B_{i-1}-Id| &\le |B_kB_{k-1}\cdots B_i-Id|+|B_kB_{k-1}\cdots B_i(B_{i-1}-Id)|\\
&\le 2\sqrt{C_0} \sum_{j=i}^k \lt(\frac{\log \lt(\frac{\ell_{j}}{r_{*,L}}\rt) }{\lt(\frac{\ell_j}{r_{*,L}}\rt)^2}\rt)^{\frac12} + |B_k\cdots B_i| |B_{i-1}-Id| \\
& \stackrel{\eqref{eq:induction1}}{\le} 2\sqrt{C_0} \sum_{j=i}^k \lt(\frac{\log \lt(\frac{\ell_{j}}{r_{*,L}}\rt) }{\lt(\frac{\ell_{j}}{r_{*,L}}\rt)^2}\rt)^{\frac12} +2\sqrt{C_0} \lt(\frac{\log \lt(\frac{\ell_{i-1}}{r_{*,L}}\rt)}{\lt(\frac{\ell_{i-1}}{r_{*,L}}\rt)^2}\rt)^{\frac12}\\
&\le 2\sqrt{C_0}\sum_{j=i-1}^k \lt(\frac{\log \lt(\frac{\ell_{j}}{r_{*,L}}\rt) }{\lt(\frac{\ell_j}{r_{*,L}}\rt)^2}\rt)^{\frac12}.
\end{align*}
\smallskip

{\it Step 2.4.}[Proof of \eqref{eq:inductionbis}]
We first notice that since $\theta\ll 1$,  $\ell_{k}\ll \ell_{k-1}$ and we recall that  $\ell_{k-1}$ is dyadic.  We set 
\begin{equation}\label{defOk}
O_k:=A_k Q_{\ell_{k-1}} -(A_k \xL+a_k) 
\end{equation}
and notice that by the $L^\infty$ bound \eqref{LinftyboundT} applied to $T_k$ we have
\[
 |A_k \yL-(A_k \xL+a_k)|\stackrel{\eqref{Tk}\& \eqref{estimx0hat}}{=}|T_k(0)|\les \ell_k E_k^{\frac{1}{4}}
\]
so that from \eqref{eq:induction2} in the form of $E_k\ll1$, the first part of \eqref{eq:induction3} in the form of $|A_k-Id|\ll1$ and  \eqref{estimx0hat},
\begin{equation}\label{estimaxoak}
|A_k \xL+a_k|\le |A_k \yL-(A_k \xL+a_k)|+ |A_k \yL| \les\ell_k E_k^{\frac{1}{4}} +r_{*,L} \ll \ell_k.
\end{equation}
Therefore, using again that $|A_k-Id|\ll 1$ together with the fact that $B_{2\ell_k}\subset Q_{\frac{ \ell_{k-1}}{2}}$ imply that $B_{2 \ell_k}\subset A_k Q_{\ell_{k-1}} -(A_k \xL+a_k)=O_k$ so that the first part of \eqref{eq:inductionbis} holds.\\

Let us prove that the second part of \eqref{eq:inductionbis} also holds. Let $\widetilde{T}^k$ be the optimal transport map between  the constant measure on $Q_{\ell_{k-1}}$ and the restriction of the measure $\tilde \mu$ to this set
i.e.
\[
 W_2^2\lt(\tilde \mu\restr Q_{\ell_{k-1}},\frac{\tilde\mu(Q_{\ell_{k-1}})}{|Q_{\ell_{k-1}}|}\chi_{Q_{\ell_{k-1}}}\rt)=\int_{Q_{\ell_{k-1}}} |\widetilde{T}^k-y|^2 \frac{\tilde\mu(Q_{\ell_{k-1}})}{|Q_{\ell_{k-1}}|}.
\]
 We then let for $z\in O_k$, $\widehat{T}^k(z):= A_k\widetilde{T}^k(A_k^{-1}(z+A_k\xL+a_k))-(A_k \xL+a_k)$. We first show that $\widehat{T}^k\# \frac{\mu_k(O_k)}{|O_k|}\chi_{O_k}=\mu_k \restr O_k$. For this we notice that by definition of $\mu_k$, if 
\[
\tilde \mu\restr Q_{\ell_{k-1}}=\alpha_0 \sum_i \delta_{y_i}
\]
then 
\begin{equation}\label{mukrestOk}
 \mu_k\restr O_k= \frac{\alpha_0}{|\det A_k|}\sum_i \delta_{A_k y_i-(A_k \xL+a_k)},
\end{equation}
 so that $\mu_k(O_k)=\frac{\tilde\mu(Q_{\ell_{k-1}})}{|\det A_k|}$ and $|O_k|=|\det A_k| |Q_{\ell_{k-1}}|$. For $\zeta\in C^0(O_k)$ we thus have
\begin{align*}
 \int_{O_k} \zeta \widehat{T}^k \# \frac{\mu_k(O_k)}{|O_k|}&=\int_{O_k} \zeta(A_k\widetilde{T}^k(A_k^{-1}(z+A_k \xL+a_k))-(A_k \xL+a_k)) \frac{\mu_k(O_k)}{|O_k|}\\
 &=\int_{Q_{\ell_{k-1}}} \zeta(A_k\widetilde{T}^k-(A_k \xL+a_k)) \frac{\mu_k(O_k)}{|O_k|} |\det A_k|\\
 &=\int_{Q_{\ell_{k-1}}}\zeta(A_ky -(A_k \xL+a_k))\frac{\mu_k(O_k)}{|O_k|} \frac{|Q_{\ell_{k-1}}|}{\tilde\mu(Q_{\ell_{k-1}})}|\det A_k| d\tilde\mu(y)\\
 &=\int_{Q_{\ell_{k-1}}}\zeta(A_ky -(A_k \xL+a_k))|\det A_k|^{-1}d\tilde\mu(y)\\
 &\stackrel{\eqref{mukrestOk}}{=}\int_{O_k} \zeta d\mu_k,
\end{align*}
proving that indeed $\widehat{T}^k\# \frac{\mu_k(O_k)}{|O_k|}\chi_{O_k}=\mu_k \restr O_k$. If we now use $\widehat{T}^k$ as competitor for the optimal
transport problem between $\frac{\mu_k(O_k)}{|O_k|}\chi_{O_k}$ and   $\mu_k \restr O_k$, we obtain
\begin{align*}
 D_k&\le \frac{1}{\ell_k^4} \int_{O_k} |\widehat{T}^k-z|^2 \frac{\mu_k(O_k)}{|O_k|}\\
 &= \frac{1}{\ell_k^4} \int_{O_k} |A_k \widetilde{T}^k(A_k^{-1}(z+A_k \xL+a_k))-(A_k \xL+a_k)-z|^2\frac{\mu_k(O_k)}{|O_k|}\\
 &\stackrel{\eqref{eq:induction3}\& \eqref{defOk}}{\les}  \frac{1}{\ell_k^4} \int_{Q_{\ell_{k-1}}} |\widetilde{T}^k-y|^2  \frac{\tilde\mu(Q_{\ell_{k-1}})}{|Q_{\ell_{k-1}}|}\\
 &\stackrel{\eqref{bounddata2}}{\les} \frac{\log \lt(\frac{\ell_{k}}{r_{*,L}}\rt)}{\lt(\frac{\ell_k}{r_{*,L}}\rt)^2}, 
\end{align*}
where we used that  $\tilde\mu=\hat \mu$ in $Q_{\ell_{k-1}}$. This concludes the proof of the second part of \eqref{eq:inductionbis}.\\
\smallskip

{\it Step 3.}[Conclusion]
We can thus iterate this procedure up to $K=\lt\lfloor \frac{\log \frac{\ell_0}{C_* r_{*,L}}}{|\log \theta|}\rt\rfloor\sim \log \frac{L}{r_{*,L}}$. By \eqref{eq:induction3} we have 
\begin{equation}\label{estima}
 |A_{K}^{-1}a_K|^2\les r_{*,L}^2 \log^3\lt(\frac{L}{r_{*,L}}\rt).
\end{equation}
Using  the $L^\infty$ bound \eqref{LinftyboundT} for $T_K$, we obtain 
\begin{multline}\label{eq:estimx0ak}
 |\xL+A_K^{-1}a_K|^2\stackrel{\eqref{eq:induction3}}{\les} |A_K T(\xL)-(A_K \xL+a_K)|^2 +| T(\xL)|^2 
 \stackrel{\eqref{Tk}}{\les} |T_K(0)|^2 +|\yL|^2\\
 \stackrel{\eqref{estimx0hat}}{\les} r_{*,L}^2
\end{multline}
which together with \eqref{estima} gives \eqref{estimxomega}. \\

We now prove \eqref{eq:estimT}.  
Since $B_{\ell_{k+1}}(\xL)\subset A_k^*B_{\ell_k} +\xL$ and recalling that  $T_k(x)=A_k T(A_k^* x+\xL)-(A_k \xL+a_k)$ (see \eqref{Tk}), we can first estimate 
\begin{align}
\lefteqn{ \frac{1}{\ell_{k+1}^4}\int_{B_{\ell_{k+1}}(\xL)} |T-(x+A_k^{-1}a_k)|^2}\nonumber\\
& \le \frac{1}{\ell_{k+1}^4}\int_{A_k^*B_{\ell_{k}}+\xL} |T-(x+A_k^{-1}a_k)|^2\nonumber\\
 &\stackrel{\eqref{eq:induction3}}{\les} \frac{1}{\ell_{k}^4} \int_{B_{\ell_k}}|A_k T(A_k^* y+\xL) -A_k(A_k^* y+\xL+A_k^{-1}a_k)|^2\nonumber\\
 &\les \frac{1}{\ell_{k}^4} \int_{B_{\ell_k}}|A_k T(A_k^* y+\xL)-(A_k \xL+a_k) -y|^2 + \frac{1}{\ell_{k}^4}\int_{B_{\ell_k}}|(Id-A_kA_k^*) y|^2\nonumber\\ 
 &\stackrel{\eqref{eq:induction3}\&\eqref{Tk}}{\les} \lt(\frac{1}{\ell_{k}^4}\int_{B_{\ell_k}} |T_k-y|^2 \rt) +|Id-A_k A_k^*|^2\stackrel{\eqref{eq:induction3}\&\eqref{eq:induction2}}{\les} \frac{\log \lt(\frac{\ell_k}{r_{*,L}}\rt)}{\lt(\frac{\ell_k}{r_{*,L}}\rt)^2} \label{eq:optimaldecay}.
\end{align}
Since by definition (recall \eqref{def:Ak}) $a_i=B_i a_{i-1}+ B_i b_i$, we have $A_i^{-1} a_i-A_{i-1}^{-1} a_{i-1} = A_{i-1}^{-1} b_i$ and thus for $1\le k< K$, we have 
\begin{multline}\label{estimaminak}
 |A_K^{-1}a_K-A_k^{-1}a_k|^2\le\lt(\sum_{i=k+1}^K |A_i^{-1} a_i-A_{i-1}^{-1} a_{i-1}|\rt)^2= \lt(\sum_{i=k+1}^K |A_{i-1}^{-1} b_i|\rt)^2\\
 \stackrel{\eqref{eq:induction3}}{\les} \lt(\sum_{i=k+1}^K |b_i|\rt)^2\stackrel{\eqref{eq:induction1}}{\les} r_{*,L}^2 (K-k)^2 \log\lt(\frac{\ell_k}{r_{*,L}}\rt)\\
 \les r_{*,L}^2 \log^3\lt(\frac{\ell_k}{r_{*,L}}\rt).
\end{multline}
Noticing that by \eqref{eq:estimx0ak}, it is enough to prove \eqref{eq:estimT} with $A_K^{-1} a_K$ instead of $-\xL$, we conclude by \eqref{eq:optimaldecay} and \eqref{estimaminak} that 
\begin{align*}
\lefteqn{ \frac{1}{\ell_{k+1}^4}\int_{B_{\ell_{k+1}}(\xL)} |T-(x+A_K^{-1}a_K)|^2}\\
&\les \lt(\frac{1}{\ell_{k+1}^4}\int_{B_{\ell_{k+1}}(\xL)} |T-(x+A_k^{-1}a_k)|^2\rt)+ \frac{1}{\ell_k^2} |A_K^{-1} a_K-A_k^{-1}a_k|^2\les  \frac{ \log^3\lt(\frac{\ell_k}{r_{*,L}}\rt)}{\lt(\frac{\ell_k}{r_{*,L}}\rt)^2},
\end{align*}
 and obtain \eqref{eq:estimT}.
\end{proof}

\begin{remark}\label{rem:optimal rate}
 We would like to highlight that, although our estimate \eqref{eq:estimT} is not optimal with respect to  the power on the logarithmic term, estimate \eqref{eq:optimaldecay} leads to the optimal estimate
 \[
  \inf_{\xi\in \R^2} \ \frac{1}{\ell^4} \int_{B_\ell(x_L)} |T-(x-\xi)|^2\les \frac{\log \lt(\frac{\ell}{r_{*,L}}\rt)}{\lt(\frac{\ell}{r_{*,L}}\rt)^2} \qquad \qquad \forall\  2 r_{*,L}\le \ell\le L.
 \]
The suboptimal rate in \eqref{eq:estimT} comes from the bound \eqref{eq:induction3} which does not take cancellations into account. 
 \end{remark}

\subsection{Locally optimal couplings between Lebesgue and Poisson}

In this section we show how Theorem \ref{theo:mainresult intro} can be used to derive locally optimal couplings between the Lebesgue measure and the Poisson measure on $\R^2$. 
For this we will use the optimal transport maps $T_{L}=T_{\mu,L}$ constructed above and pass to the limit as $L\to \infty$. Since the transport cost per unit volume diverges logarithmically, see \cite{AmStTr16} or Section \ref{sec:Poi}, we will need to use a renormalization procedure.
Therefore, while the approximating couplings enjoy strong stationarity properties, cf.\ \eqref{eq:covariantliftT}, the limiting couplings themselves will not. However, the shift stationarity property will be shown to survive in the second-order increments of the corresponding Kantorovich potentials. 
In order to set up the limit procedure, we need to equip the configuration space $\Gamma$ (see Section \ref{sec:Poi}) and the set of potentials with a topology.

 We equip $\Gamma$ with the topology obtained by testing against continuous and compactly supported functions.  
Denote the space of all real-valued convex functions $\psi:\R^2\to\R$ by $\mathcal K$. We equip $\mathcal K$ (and $C^0(\R^2)$) with the topology of uniform convergence on compact sets. Let us point out that with these topologies, both $\Gamma$ and $\K$ are metrizable, which makes them Polish spaces. On $\mathcal{P}(\Gamma\times\K)$, we will consider the weak topology given by testing against functions in $C_b(\Gamma\times \K)$.

Denote by $\widehat{\psi}_L$ the convex function on $\R^2$ such that $T_L=\nabla \widehat{\psi}_L$ on $Q_L$ and \eqref{Cordero} holds, i.e.\ $\nabla \widehat{\psi}_L(x+z)=\nabla \widehat{\psi}_L(x)+z$ for all $(x,z)\in \R^2\times(L\Z)^2$ and
$\nabla \widehat{\psi}_L\# dx= \frac{L^2}{\mu(Q_L)}\,\mu$.\\

Since $\xL$ (defined above Theorem \ref{theo:mainresult intro}) is logarithmically diverging in $L$, we will need to translate either the Lebesgue measure or the Poisson measure by a logarithmically diverging
factor in order to pass to the limit. Since the Lebesgue measure (on $\R^2$) is invariant under such translations while the Poisson point process is not, it is better to make this shift in the domain rather than in the image and set
\begin{equation}\label{def:shift} \psi_L(x):=\widehat \psi_L(x+\xL).\end{equation}
Note that 
\begin{equation}\label{eqnabpsi}
 \nabla \psi_L\#dx=\frac{L^2}{\mu(Q_L)}\,\mu,
\end{equation}
$\nabla \psi_L(0)=0$ and  the Legendre conjugate $\psi^*_L$ of $\psi_L$ satisfies
$$\psi^*_L(y)=\widehat\psi^*_L(y)-\xL\cdot y.$$
By adding a constant to $\psi_L$ we may assume that $\psi_L(0)=0$. Notice that by \eqref{D2hpsi} of Lemma \ref{lem:stationa} and recalling that $D^2_h\psi^*(y)= \psi^*(y+h)+\psi^*(y-h)-2\psi^*(y)$, we still have 
\begin{equation}\label{D2hpsiL}
 D^2_h \psi^*_{\theta_z \mu, L}(y)=D^2_h \psi^*_{\mu,L}(y+z).
\end{equation}
Let us point out that because of the shift introduced in \eqref{def:shift}, the same invariance does not hold for $\psi_L$.\\

 For a given $\mu\in \Gamma$, the bound \eqref{eq:estimT} directly translates into locally uniformly $L^2-$bounds for $\nabla \psi_{\mu,L}$ which by convexity of $\psi_{\mu,L}$ yields
compactness of $(\psi_{\mu,L})_L$ in $\mathcal{K}$ (see \eqref{eq:uniformpotential} below). Therefore, up to subsequence, $\psi_{\mu,L}$ converges locally uniformly to a convex function $\psi_\mu$ satisfying $\nabla \psi_\mu \# dx=\mu$. However since we do not have any uniqueness property of this limit, 
the subsequence depends {\it a priori} on $\mu$ and we need to pass to the limit in the sense of Young measures. For this purpose, we first define the map
$$ \Psi_L:\Gamma\to\mathcal K, \mu \mapsto \psi_L, $$
which is  measurable and  depends only on $\mu\LL Q_L$, and then  
 define the probability measure $q_L\in \mathcal{P}(\Gamma\times \K)$ by
\begin{equation}\label{def:qL} q_L:=(id,\Psi_L)\#\PP_L=\PP_L \otimes \delta_{\Psi_L}.\end{equation}
We will show that the sequence $(q_L)_L$ is tight and that any limit point $q$ gives full mass to pairs $(\mu,\psi)$ such that $\nabla \psi\#dx=\mu$ and such that the second-order increments of $\psi^*$ are shift covariant. The  crucial ingredient is the following lemma that gives us a uniform control on the potentials $\psi_L$.

\begin{lemma}\label{lem:uniformpotential}
 There exists a constant $C>1$ such that for every dyadic $L$ and every $\mu\in \Gamma$ such that  $ r_{*,L}\ll L$, there holds for $x\in \R^2$ 
\begin{equation}\label{eq:uniformpotential} 
\frac{1}{4}|x|^2 -  Cr_{*,L}^2 \leq \psi_L(x) \leq |x|^2 + C r_{*,L}^2 . 
\end{equation}
Therefore, letting for $\lambda\in \R$, 
\begin{equation}\label{def:Flambda}
 F_\lambda:=\lt\{\psi \in \K \, : \, \frac{1}{4}|x|^2-\lambda\le \psi\le |x|^2+\lambda\rt\}
\end{equation}
this means that if  $ r_{*,L}\ll L$, $\psi_L\in F_\lambda$ for every $\lambda\ge C r_{*,L}^2$.
\end{lemma}
\begin{proof}
Let us prove that 
\begin{equation}\label{firstboundpsiL}
 \sup_{ r_{*,L}\ll \ell}\, \frac{1}{\ell^4} \int_{B_\ell} |\nabla \psi_L - x|^2\ll 1.
\end{equation}
For $\ell\le L$, this directly follows from  \eqref{eq:estimT} and the definition of $\psi_L$. If now $\ell=k L$ for some $k\in \N$,  
 $Q_L-$periodicity of the function  $(\nabla \psi_L-x)$ yields 
\[
 \frac{1}{\ell^4} \int_{Q_\ell} |\nabla \psi_L - x|^2= \frac{1}{ L^2 \ell^2} \int_{Q_L}|\nabla \psi_L - x|^2\le \frac{1}{L^4}\int_{Q_L}|\nabla \psi_L - x|^2 \stackrel{ \eqref{bounddata2}}{\ll} 1
\]
so that \eqref{firstboundpsiL} can be also obtained for $\ell\ge L$.
 Letting $f(x):=\psi_L(x) - \frac{|x|^2}{2}$, this implies together with the $L^\infty-$bound given by Lemma \ref{lem:Linftybound} that for $r_{*,L}\ll \ell$,
$$ \sup_{B_\ell} |\nabla f|\ll \ell, $$
 which can be rewritten as $|\nabla f(x)|\ll  |x|$ for $r_{*,L}\ll |x|$. Using  $f(0)=0$, we obtain from integration $|f(x)|\ll |x|^2$ for $r_{*,L}\ll \ell$.
Going back to the definition of $f$, this concludes the proof of \eqref{eq:uniformpotential}.
\end{proof}

This lemma endows us with the necessary compactness to prove the main result of this subsection,  the convergence of $(\Psi_L)_L$ in terms of Young measures. This is precisely Theorem \ref{theo:limit intro} which we recall for the convenience of the reader.

\begin{theorem*}
 The sequence of probability measure $(q_L)_{L}$  (cf \eqref{def:qL}) is tight in $\mathcal P(\Gamma\times\mathcal K)$. Moreover, any accumulation point $q$ satisfies the following properties:
\begin{enumerate}
 \item[(i)] The first marginal of $q$ is the Poisson point process;
 \item[(ii)] $q$ almost surely $\nabla \psi\#dx=\mu$;
\item[(iii)] for any $h,z\in\R^2$ and $f\in C_b(\Gamma\times C^0(\R^2))$ there  holds
$$ \int_{\Gamma\times \K} f(\mu,D^2_h\psi^*) dq = \int_{\Gamma\times \K} f(\theta_{-z}\mu,D^2_h \psi^*(\cdot -z)) dq .  $$
\end{enumerate}
\end{theorem*}

\begin{proof}
{\it Step 1.}
We start with tightness. Since  trivially $\mu\LL Q_L \to \mu$ in $\Gamma$ we have that $\PP_L\to\PP$ weakly in $\mathcal P(\Gamma)$. 
In particular, the sequence $(\PP_L)_L$ is tight and for any $\eps>0$ there is a compact set $\Gamma_\eps\subset\Gamma$ such that for all $L$ we have 
$\PP_L(\Gamma_\eps)\ge 1-\eps.$ Since by Theorem \ref{theo:estimDrstarper} we have $\sup_L \EE_L\sqa{\exp\lt(\frac{c r_{*,L}^2}{\log (2r_{*,L})}\rt)}<\infty$, there is a constant $\lambda$ such that for each $L$ large enough
 \[
  \PP_L(\{ r_{*,L}\le \sqrt{\lambda}\})\ge 1-\eps.
 \]
 Lemma \ref{lem:uniformpotential} implies that
\[
\Psi_L(\{ r_{*,L}\le \sqrt{\lambda}\}) \subset F_{\lambda},
\]
so that
\begin{equation}\label{estimqL}
 q_L(\Gamma\times F_\lambda)\ge 1-\eps.
\end{equation}
Because of convexity, local boundedness yields local compactness in uniform topology. Thus setting $K_\eps:=\Gamma_\eps\times F_{\lambda}$ we have that $K_\eps$ is compact and
$$ q_L\bra{(K_\eps)^c} \leq 2\eps, $$ 
 which proves tightness. Moreover, since $\PP_L\to\PP$ weakly in $\mathcal P(\Gamma)$ item (i) is shown.
\medskip\\
{\it Step 2.}
To show (ii) we define for $k, n \in \N$ the set $G^{k,n}\subset \Gamma\times\mathcal K$ by (recall the definition of $F_k$ given in \eqref{def:Flambda})
\[  G^{k,n}:= \lt\{(\mu,\psi)\in \Gamma\times F_k \, : \, \lt(1-\frac{1}{n}\rt) \mu\le \nabla \psi\# dx\le \lt(1+\frac{1}{n}\rt) \mu\rt\}
\]
 and put $G:=\cap_{n\in \N}\cup_{k\in\N} G^{k,n}$. The claim would follow provided we can prove that $q(G)=1$.\\
 We first show that for fixed $k,n\in\N$ the set $G^{k,n}$ is closed. Let $(\mu_m,\psi_m)_{m\in \N}\in G^{k,n}$ be a sequence converging to some $(\mu,\psi)\in \Gamma\times \K$. Since $F_k$ is closed, we have $\psi\in F_k$ and we only need to prove that 
\begin{equation}\label{eq:toproveclose}
 \lt(1-\frac{1}{n}\rt) \mu\le \nabla \psi\# dx\le \lt(1+\frac{1}{n}\rt) \mu,
\end{equation}
which by weak convergence of $\mu_m$ to $\mu$ and the fact that  $(\mu_m,\psi_m)$ satisfies \eqref{eq:toproveclose}, would be proven provided we show that $\nabla \psi_m\#dx$ weakly converges up to subsequence to $\nabla \psi\#dx$.\\ 

Let $f\in C_c(\R^2)$ be fixed and let us prove that up to subsequence, 
\begin{equation}\label{eq:toproveclose2}
 \int_{\R^2} f (\nabla \psi_m)\to \int_{\R^2} f(\nabla \psi).
\end{equation}
 By local uniform convergence of the convex functions $\psi_m$ to $\psi$, if $p_m\in \partial \psi_m(x)$  with $p_m\to p$, then $p\in \partial \psi(x)$. Therefore,  $\nabla \psi_m$ converges a.e.\ to $\nabla \psi$.
Let $r>0$ be such that $\spt f\subset B_r$.  In order to apply the dominated convergence theorem and conclude the proof of \eqref{eq:toproveclose2}, we need to prove that
there exists $R$ depending only on $k$ and $r$ such that if $|x|\ge B_R$, then $|\nabla \psi_m(x)|\ge r$. This is a simple consequence of the fact that $\psi_m\in F_k$ and the monotonicity of $\nabla \psi_m$. Indeed, since $\psi_m\in F_k$,
\[
 \frac{1}{4}|x|^2- k\le \psi_m\le |x|^2+k
\]
so that at every point $x$ of differentiability of $ \psi_m$, since 
\[
  \frac{1}{4}|x|^2- 2k\le \psi_m(x)-\psi_m(0)\le \nabla \psi_m(x)\cdot x\le|\nabla \psi_m(x)| |x|, 
\]
we have 
\[
 |\nabla \psi_m (x)|\ge \frac{1}{4}|x| -\frac{2k}{|x|}.
\]
This gives the claim and shows that $G^{k,n}$ is indeed a closed set.

Since $G^{k,n}$ is measurable,  $G=\cap_n \cup_k G^{k,n}$ is also measurable. Let $q$ be an accumulation point of $(q_L)_L$ so that up to subsequence $q_{L}\to q$. For a given $\eps>0$, let us prove that for every $n$ and for $k$ large enough (depending only on $\eps$) and every $L$ large enough (depending only on $k$, $n$ and $\eps$),
\begin{equation}\label{eq:toproveGkn}
 q_L(G^{k,n})\ge 1-\eps.
\end{equation}
Since 
\[
 q_L((G^{k,n})^c)\le q_L(\Gamma\times (F^k)^c)+q_L\lt(\lt\{(\mu,\psi) \, : \, \lt(1-\frac{1}{n}\rt) \mu\le \nabla \psi\# dx\le \lt(1+\frac{1}{n}\rt) \mu\rt\}^c\rt),
\]
it is enough to prove that each of the terms on the right-hand side are smaller than $\frac{\eps}{2}$ for $k$, $n$ and $L$ large enough. The first term is estimated in \eqref{estimqL} and we just need to consider the second term. 
For every $L>0$ and $\mu\in \Gamma$, we have by \eqref{eqnabpsi}  
\begin{multline*}
 q_L\lt(\lt\{(\mu,\psi) \, : \, \lt(1-\frac{1}{n}\rt) \mu\le \nabla \psi\# dx\le \lt(1+\frac{1}{n}\rt) \mu\rt\}^c\rt)\\
 =\PP_L\sqa{ \mu(Q_L)\notin L^2\lt[ \frac{n}{n+1}, \frac{n}{n-1}\rt]},
\end{multline*}
which by Cram\'er-Chernoff's bounds for the Poisson distribution with intensity $L^2$ (see \cite{BouLuMa}) gives 
\[
 q_L\lt(\lt\{(\mu,\psi) \, : \, \lt(1-\frac{1}{n}\rt) \mu\le \nabla \psi\# dx\le \lt(1+\frac{1}{n}\rt) \mu\rt\}^c\rt)\le \exp\lt(-C \frac{L^2}{n^2}\rt),
\]
concluding the proof of \eqref{eq:toproveGkn}.

Now for fixed $k, n\in \N$ large enough, since $G^{k,n}$ is closed, we have by \eqref{eq:toproveGkn}
$$1-\eps\le \limsup_L q_{L}(G^{k,n})\leq q(G^{k,n}).$$
Using that for every $k, n\in \N$,  $G^{k,n+1}\subset G^{k,n}$  and that $G=\cap_n\cup_k G^{k,n}$, we obtain that  for every $\eps>0$,
 \[
  q(G)\ge 1-\eps,
 \]
which concludes the proof.
\medskip\\
{\it Step 3.} To show (iii) fix an accumulation point $q$ and a subsequence, still denoted by $(q_L)_L$ converging weakly to $q$. Since for fixed $\lambda>0$, the Legendre transform $\psi \to \psi^*$ is continuous from $F_\lambda$ to $\K$,  for every $h\in \R^2$ and $\lambda>0$ the map $\psi\mapsto D^2_h\psi^*$ is continuous  on $F_\lambda$ (recall \eqref{def:Flambda}) with values in $C^0(\R^2)$.
Hence, the convergence $q_L\to q$ together with \eqref{estimqL} readily implies for all $f\in C_b(\Gamma\times C^0(\R^2))$ that also
$$ \int_{\Gamma\times \K} f(\mu,D^2_h\psi^*) dq_L\to \int_{\Gamma\times \K} f(\mu, D^2_h\psi^*) dq.$$
By \eqref{D2hpsiL}, we have $q_L$ almost surely $D^2_h\psi^*_{\mu}=D^2_h\psi^*_{\theta_z \mu}(\cdot -z)$. Using the invariance of $\PP_L$ under $\theta$ an the definition of $q_L$ we have for fixed $z\in\R^2$
\begin{align*}
 \int_{\Gamma\times \K} f(\mu,D^2_h\psi^*)dq_L
=&\int_{\Gamma} f(\mu,D^2_h\psi^*_{\mu}) d\PP_L\\
=& \int_{\Gamma} f(\theta_{-z}\theta_z\mu, D^2_h\psi^*_{\theta_z \mu}(\cdot -z))d\PP_L\\
=& \int_{\Gamma} f(\theta_{-z} \mu, D^2_h\psi^*_\mu(\cdot -z))d\PP_L\\
=& \int_{\Gamma\times \K} f(\theta_{-z} \mu, D^2_h\psi^*(\cdot -z))dq_L.
\end{align*}
Since  for fixed $z\in \R^2$, $\theta_{-z}$ is continuous on $\Gamma$, for every such $z$ and  $\lambda>0$ the map $(\mu,\psi)\to f(\theta_{-z}\mu,D^2_h \psi^*)\in C_b(\Gamma\times F_\lambda)$ so that by weak convergence $q_L\to q$ combined again with \eqref{estimqL} we have
$$\int_{\Gamma\times \K} f(\theta_{-z} \mu, D^2_h\psi^*(\cdot -z))dq_L \to \int_{\Gamma\times \K} f(\theta_{-z} \mu, D^2_h\psi^*(\cdot -z)) dq$$
which implies the thesis.

\end{proof}

\providecommand{\bysame}{\leavevmode\hbox to3em{\hrulefill}\thinspace}
\providecommand{\MR}{\relax\ifhmode\unskip\space\fi MR }
\providecommand{\MRhref}[2]{%
  \href{http://www.ams.org/mathscinet-getitem?mr=#1}{#2}
}
\providecommand{\href}[2]{#2}

 \end{document}

%% file: fbar.pdf_t
\begin{picture}(0,0)%
\includegraphics{fbar.pdf}%
\end{picture}%
\setlength{\unitlength}{4144sp}%
\begingroup\makeatletter\ifx\SetFigFont\undefined%
\gdef\SetFigFont#1#2#3#4#5{%
  \reset@font\fontsize{#1}{#2pt}%
  \fontfamily{#3}\fontseries{#4}\fontshape{#5}%
  \selectfont}%
\fi\endgroup%
\begin{picture}(10017,5866)(661,-6461)
\put(9451,-5011){\makebox(0,0)[lb]{\smash{{\SetFigFont{17}{20.4}{\familydefault}{\mddefault}{\updefault}{\color[rgb]{0,0,0}$f_-^R$}%
}}}}
\put(2116,-2941){\makebox(0,0)[lb]{\smash{{\SetFigFont{17}{20.4}{\familydefault}{\mddefault}{\updefault}{\color[rgb]{0,0,0}$f_+^R$}%
}}}}
\put(10036,-4516){\makebox(0,0)[lb]{\smash{{\SetFigFont{17}{20.4}{\familydefault}{\mddefault}{\updefault}{\color[rgb]{0,0,0}$X(t_-^R)$}%
}}}}
\put(2791,-2941){\makebox(0,0)[lb]{\smash{{\SetFigFont{17}{20.4}{\familydefault}{\mddefault}{\updefault}{\color[rgb]{0,0,0}$X(t_+^R)$}%
}}}}
\put(6121,-3976){\makebox(0,0)[lb]{\smash{{\SetFigFont{17}{20.4}{\familydefault}{\mddefault}{\updefault}{\color[rgb]{0,0,0}$\rho$}%
}}}}
\put(676,-1231){\makebox(0,0)[lb]{\smash{{\SetFigFont{17}{20.4}{\familydefault}{\mddefault}{\updefault}{\color[rgb]{0,0,0}$t=1$}%
}}}}
\put(676,-1951){\makebox(0,0)[lb]{\smash{{\SetFigFont{17}{20.4}{\familydefault}{\mddefault}{\updefault}{\color[rgb]{0,0,0}$t=1-\tau$}%
}}}}
\put(676,-6001){\makebox(0,0)[lb]{\smash{{\SetFigFont{17}{20.4}{\familydefault}{\mddefault}{\updefault}{\color[rgb]{0,0,0}$t=0$}%
}}}}
\put(2566,-6361){\makebox(0,0)[lb]{\smash{{\SetFigFont{17}{20.4}{\familydefault}{\mddefault}{\updefault}{\color[rgb]{0,0,0}$\partial B_R$}%
}}}}
\put(6166,-6361){\makebox(0,0)[lb]{\smash{{\SetFigFont{17}{20.4}{\familydefault}{\mddefault}{\updefault}{\color[rgb]{0,0,0}$0$}%
}}}}
\put(9766,-6361){\makebox(0,0)[lb]{\smash{{\SetFigFont{17}{20.4}{\familydefault}{\mddefault}{\updefault}{\color[rgb]{0,0,0}$\partial B_R$}%
}}}}
\put(5806,-826){\makebox(0,0)[lb]{\smash{{\SetFigFont{17}{20.4}{\familydefault}{\mddefault}{\updefault}{\color[rgb]{0,0,0}$\rho_1=\mu$}%
}}}}
\put(7471,-6271){\makebox(0,0)[lb]{\smash{{\SetFigFont{17}{20.4}{\familydefault}{\mddefault}{\updefault}{\color[rgb]{0,0,0}$\rho_0=1$}%
}}}}
\end{picture}%

%% file: rhobar.pdf_t
\begin{picture}(0,0)%
\includegraphics{rhobar.pdf}%
\end{picture}%
\setlength{\unitlength}{4144sp}%
\begingroup\makeatletter\ifx\SetFigFont\undefined%
\gdef\SetFigFont#1#2#3#4#5{%
  \reset@font\fontsize{#1}{#2pt}%
  \fontfamily{#3}\fontseries{#4}\fontshape{#5}%
  \selectfont}%
\fi\endgroup%
\begin{picture}(6420,6989)(4621,-6143)
\put(9091,-3931){\makebox(0,0)[lb]{\smash{{\SetFigFont{17}{20.4}{\familydefault}{\mddefault}{\updefault}{\color[rgb]{0,0,0}$\partial B_R$}%
}}}}
\put(4636,-5236){\makebox(0,0)[lb]{\smash{{\SetFigFont{17}{20.4}{\familydefault}{\mddefault}{\updefault}{\color[rgb]{0,0,0}$\overline{\rho}_+^R$}%
}}}}
\put(5446,-3526){\makebox(0,0)[lb]{\smash{{\SetFigFont{17}{20.4}{\familydefault}{\mddefault}{\updefault}{\color[rgb]{0,0,0}$\Pi$}%
}}}}
\put(6391,-2851){\makebox(0,0)[lb]{\smash{{\SetFigFont{17}{20.4}{\familydefault}{\mddefault}{\updefault}{\color[rgb]{0,0,0}$X(t_+^R)$}%
}}}}
\put(10936,-5281){\makebox(0,0)[lb]{\smash{{\SetFigFont{17}{20.4}{\familydefault}{\mddefault}{\updefault}{\color[rgb]{0,0,0}$t=0$}%
}}}}
\put(10936,-1321){\makebox(0,0)[lb]{\smash{{\SetFigFont{17}{20.4}{\familydefault}{\mddefault}{\updefault}{\color[rgb]{0,0,0}$t=1-\tau$}%
}}}}
\put(11026,-151){\makebox(0,0)[lb]{\smash{{\SetFigFont{17}{20.4}{\familydefault}{\mddefault}{\updefault}{\color[rgb]{0,0,0}$t=1$}%
}}}}
\end{picture}%

%% file: rhobarlay.pdf_t
\begin{picture}(0,0)%
\includegraphics{rhobarlay.pdf}%
\end{picture}%
\setlength{\unitlength}{4144sp}%
\begingroup\makeatletter\ifx\SetFigFont\undefined%
\gdef\SetFigFont#1#2#3#4#5{%
  \reset@font\fontsize{#1}{#2pt}%
  \fontfamily{#3}\fontseries{#4}\fontshape{#5}%
  \selectfont}%
\fi\endgroup%
\begin{picture}(11208,14087)(2996,-6143)
\put(9091,-3931){\makebox(0,0)[lb]{\smash{{\SetFigFont{17}{20.4}{\familydefault}{\mddefault}{\updefault}{\color[rgb]{0,0,0}$\partial B_R$}%
}}}}
\put(12781,-2131){\makebox(0,0)[lb]{\smash{{\SetFigFont{17}{20.4}{\familydefault}{\mddefault}{\updefault}{\color[rgb]{0,0,0}$O$}%
}}}}
\put(11566,209){\makebox(0,0)[lb]{\smash{{\SetFigFont{17}{20.4}{\familydefault}{\mddefault}{\updefault}{\color[rgb]{0,0,0}$\mu\llcorner O$}%
}}}}
\put(10846,-5236){\makebox(0,0)[lb]{\smash{{\SetFigFont{17}{20.4}{\familydefault}{\mddefault}{\updefault}{\color[rgb]{0,0,0}$t=0$}%
}}}}
\put(10936,-331){\makebox(0,0)[lb]{\smash{{\SetFigFont{17}{20.4}{\familydefault}{\mddefault}{\updefault}{\color[rgb]{0,0,0}$t=1$}%
}}}}
\put(4366,-871){\makebox(0,0)[lb]{\smash{{\SetFigFont{17}{20.4}{\familydefault}{\mddefault}{\updefault}{\color[rgb]{0,0,0}$X(t_+^R)$}%
}}}}
\put(5941,974){\makebox(0,0)[lb]{\smash{{\SetFigFont{17}{20.4}{\familydefault}{\mddefault}{\updefault}{\color[rgb]{0,0,0}$\widehat{\Pi}$}%
}}}}
\put(6391,5474){\makebox(0,0)[lb]{\smash{{\SetFigFont{17}{20.4}{\familydefault}{\mddefault}{\updefault}{\color[rgb]{0,0,0}$\overline{\rho}^{R,\textrm{lay}}_+$}%
}}}}
\put(10936,-1276){\makebox(0,0)[lb]{\smash{{\SetFigFont{17}{20.4}{\familydefault}{\mddefault}{\updefault}{\color[rgb]{0,0,0}$t=1-\tau$}%
}}}}
\put(11386,4619){\makebox(0,0)[lb]{\smash{{\SetFigFont{20}{24.0}{\familydefault}{\mddefault}{\updefault}{\color[rgb]{0,0,0}$\frac{\mu(O)}{|O|}$}%
}}}}
\end{picture}%

%% file: muprime.pdf_t
\begin{picture}(0,0)%
\includegraphics{muprime.pdf}%
\end{picture}%
\setlength{\unitlength}{4144sp}%
\begingroup\makeatletter\ifx\SetFigFont\undefined%
\gdef\SetFigFont#1#2#3#4#5{%
  \reset@font\fontsize{#1}{#2pt}%
  \fontfamily{#3}\fontseries{#4}\fontshape{#5}%
  \selectfont}%
\fi\endgroup%
\begin{picture}(11208,14087)(2996,-6143)
\put(9091,-3931){\makebox(0,0)[lb]{\smash{{\SetFigFont{17}{20.4}{\familydefault}{\mddefault}{\updefault}{\color[rgb]{0,0,0}$\partial B_R$}%
}}}}
\put(11926,-1861){\makebox(0,0)[lb]{\smash{{\SetFigFont{17}{20.4}{\familydefault}{\mddefault}{\updefault}{\color[rgb]{0,0,0}$O$}%
}}}}
\put(6976,6824){\makebox(0,0)[lb]{\smash{{\SetFigFont{17}{20.4}{\familydefault}{\mddefault}{\updefault}{\color[rgb]{0,0,0}$\Gamma$}%
}}}}
\put(7651,-1861){\makebox(0,0)[lb]{\smash{{\SetFigFont{17}{20.4}{\familydefault}{\mddefault}{\updefault}{\color[rgb]{0,0,0}$\rho_{1-\tau}$}%
}}}}
\put(11206,-5281){\makebox(0,0)[lb]{\smash{{\SetFigFont{17}{20.4}{\familydefault}{\mddefault}{\updefault}{\color[rgb]{0,0,0}$t=0$}%
}}}}
\put(10846,-1231){\makebox(0,0)[lb]{\smash{{\SetFigFont{17}{20.4}{\familydefault}{\mddefault}{\updefault}{\color[rgb]{0,0,0}$t=1-\tau$}%
}}}}
\put(10846,-331){\makebox(0,0)[lb]{\smash{{\SetFigFont{17}{20.4}{\familydefault}{\mddefault}{\updefault}{\color[rgb]{0,0,0}$t=1$}%
}}}}
\put(7156,4799){\makebox(0,0)[lb]{\smash{{\SetFigFont{17}{20.4}{\familydefault}{\mddefault}{\updefault}{\color[rgb]{0,0,0}$g$}%
}}}}
\put(10756,4574){\makebox(0,0)[lb]{\smash{{\SetFigFont{17}{20.4}{\familydefault}{\mddefault}{\updefault}{\color[rgb]{0,0,0}$f_{g}$}%
}}}}
\put(5401,-376){\makebox(0,0)[lb]{\smash{{\SetFigFont{17}{20.4}{\familydefault}{\mddefault}{\updefault}{\color[rgb]{0,0,0}$f_{\tilde \mu}$}%
}}}}
\put(9271,-511){\makebox(0,0)[lb]{\smash{{\SetFigFont{17}{20.4}{\familydefault}{\mddefault}{\updefault}{\color[rgb]{0,0,0}$\tilde \mu$}%
}}}}
\end{picture}%

%% file: rhobulk.pdf_t
\begin{picture}(0,0)%
\includegraphics{rhobulk.pdf}%
\end{picture}%
\setlength{\unitlength}{4144sp}%
\begingroup\makeatletter\ifx\SetFigFont\undefined%
\gdef\SetFigFont#1#2#3#4#5{%
  \reset@font\fontsize{#1}{#2pt}%
  \fontfamily{#3}\fontseries{#4}\fontshape{#5}%
  \selectfont}%
\fi\endgroup%
\begin{picture}(8394,4122)(1429,-4576)
\put(2296,-3796){\makebox(0,0)[lb]{\smash{{\SetFigFont{17}{20.4}{\familydefault}{\mddefault}{\updefault}{\color[rgb]{0,0,0}$\tau$}%
}}}}
\put(5266,-4471){\makebox(0,0)[lb]{\smash{{\SetFigFont{17}{20.4}{\familydefault}{\mddefault}{\updefault}{\color[rgb]{0,0,0}$1-\frac{m_+^{\textrm{lay}}}{|B_1|}$}%
}}}}
\put(2251,-3211){\makebox(0,0)[lb]{\smash{{\SetFigFont{17}{20.4}{\familydefault}{\mddefault}{\updefault}{\color[rgb]{0,0,0}$\overline{f}_-$}%
}}}}
\put(8551,-3571){\makebox(0,0)[lb]{\smash{{\SetFigFont{17}{20.4}{\familydefault}{\mddefault}{\updefault}{\color[rgb]{0,0,0}$2\tau$}%
}}}}
\put(8596,-2446){\makebox(0,0)[lb]{\smash{{\SetFigFont{17}{20.4}{\familydefault}{\mddefault}{\updefault}{\color[rgb]{0,0,0}$\overline{\rho}_+-\overline{\rho}_-$}%
}}}}
\put(8551,-1411){\makebox(0,0)[lb]{\smash{{\SetFigFont{17}{20.4}{\familydefault}{\mddefault}{\updefault}{\color[rgb]{0,0,0}$1-2\tau$}%
}}}}
\put(4456,-826){\makebox(0,0)[lb]{\smash{{\SetFigFont{17}{20.4}{\familydefault}{\mddefault}{\updefault}{\color[rgb]{0,0,0}$1-\frac{m_+^{\textrm{lay}}}{|B_1|}-\frac{1}{|B_1|}\int_{\partial B_1}\overline{\rho}_+-\overline{\rho}_-$}%
}}}}
\put(2296,-2446){\makebox(0,0)[lb]{\smash{{\SetFigFont{17}{20.4}{\familydefault}{\mddefault}{\updefault}{\color[rgb]{0,0,0}$\frac{1}{2}$}%
}}}}
\put(2026,-1591){\makebox(0,0)[lb]{\smash{{\SetFigFont{17}{20.4}{\familydefault}{\mddefault}{\updefault}{\color[rgb]{0,0,0}$-\overline{f}_+$}%
}}}}
\put(1936,-1006){\makebox(0,0)[lb]{\smash{{\SetFigFont{17}{20.4}{\familydefault}{\mddefault}{\updefault}{\color[rgb]{0,0,0}$1-\tau$}%
}}}}
\put(4996,-2806){\makebox(0,0)[lb]{\smash{{\SetFigFont{17}{20.4}{\familydefault}{\mddefault}{\updefault}{\color[rgb]{0,0,0}$\rho^{\textrm{bulk}}$}%
}}}}
\put(2611,-4426){\makebox(0,0)[lb]{\smash{{\SetFigFont{17}{20.4}{\familydefault}{\mddefault}{\updefault}{\color[rgb]{0,0,0}$\partial B_1$}%
}}}}
\put(8236,-4426){\makebox(0,0)[lb]{\smash{{\SetFigFont{17}{20.4}{\familydefault}{\mddefault}{\updefault}{\color[rgb]{0,0,0}$\partial B_1$}%
}}}}
\end{picture}%

%% file: rhobulkbdr.pdf_t
\begin{picture}(0,0)%
\includegraphics{rhobulkbdr.pdf}%
\end{picture}%
\setlength{\unitlength}{4144sp}%
\begingroup\makeatletter\ifx\SetFigFont\undefined%
\gdef\SetFigFont#1#2#3#4#5{%
  \reset@font\fontsize{#1}{#2pt}%
  \fontfamily{#3}\fontseries{#4}\fontshape{#5}%
  \selectfont}%
\fi\endgroup%
\begin{picture}(6780,3532)(1921,-4346)
\put(4366,-2266){\makebox(0,0)[lb]{\smash{{\SetFigFont{17}{20.4}{\familydefault}{\mddefault}{\updefault}{\color[rgb]{0,0,0}$\frac{1}{2}$}%
}}}}
\put(6751,-2221){\makebox(0,0)[lb]{\smash{{\SetFigFont{17}{20.4}{\familydefault}{\mddefault}{\updefault}{\color[rgb]{0,0,0}$\frac{1}{2}$}%
}}}}
\put(6706,-3031){\makebox(0,0)[lb]{\smash{{\SetFigFont{17}{20.4}{\familydefault}{\mddefault}{\updefault}{\color[rgb]{0,0,0}$\overline{f}_-$}%
}}}}
\put(6751,-3706){\makebox(0,0)[lb]{\smash{{\SetFigFont{17}{20.4}{\familydefault}{\mddefault}{\updefault}{\color[rgb]{0,0,0}$\tau$}%
}}}}
\put(8686,-2311){\makebox(0,0)[lb]{\smash{{\SetFigFont{17}{20.4}{\familydefault}{\mddefault}{\updefault}{\color[rgb]{0,0,0}$\overline{\rho}_-$}%
}}}}
\put(8686,-1231){\makebox(0,0)[lb]{\smash{{\SetFigFont{17}{20.4}{\familydefault}{\mddefault}{\updefault}{\color[rgb]{0,0,0}$1-2\tau$}%
}}}}
\put(4366,-1096){\makebox(0,0)[lb]{\smash{{\SetFigFont{17}{20.4}{\familydefault}{\mddefault}{\updefault}{\color[rgb]{0,0,0}$1-\tau$}%
}}}}
\put(4366,-1726){\makebox(0,0)[lb]{\smash{{\SetFigFont{17}{20.4}{\familydefault}{\mddefault}{\updefault}{\color[rgb]{0,0,0}$\overline{f}_+$}%
}}}}
\put(1936,-1186){\makebox(0,0)[lb]{\smash{{\SetFigFont{17}{20.4}{\familydefault}{\mddefault}{\updefault}{\color[rgb]{0,0,0}$1-2\tau$}%
}}}}
\put(2431,-2446){\makebox(0,0)[lb]{\smash{{\SetFigFont{17}{20.4}{\familydefault}{\mddefault}{\updefault}{\color[rgb]{0,0,0}$\overline{\rho}_+$}%
}}}}
\put(2431,-3436){\makebox(0,0)[lb]{\smash{{\SetFigFont{17}{20.4}{\familydefault}{\mddefault}{\updefault}{\color[rgb]{0,0,0}$2\tau$}%
}}}}
\put(2521,-4021){\makebox(0,0)[lb]{\smash{{\SetFigFont{17}{20.4}{\familydefault}{\mddefault}{\updefault}{\color[rgb]{0,0,0}$\partial B_1$}%
}}}}
\put(3421,-4246){\makebox(0,0)[lb]{\smash{{\SetFigFont{17}{20.4}{\familydefault}{\mddefault}{\updefault}{\color[rgb]{0,0,0}$\rho^{\textrm{bdr}}_1$}%
}}}}
\put(8641,-3571){\makebox(0,0)[lb]{\smash{{\SetFigFont{17}{20.4}{\familydefault}{\mddefault}{\updefault}{\color[rgb]{0,0,0}$2\tau$}%
}}}}
\put(4141,-4021){\makebox(0,0)[lb]{\smash{{\SetFigFont{17}{20.4}{\familydefault}{\mddefault}{\updefault}{\color[rgb]{0,0,0}$\partial B_1$}%
}}}}
\put(7516,-4156){\makebox(0,0)[lb]{\smash{{\SetFigFont{17}{20.4}{\familydefault}{\mddefault}{\updefault}{\color[rgb]{0,0,0}$\rho^{\textrm{bdr}}_2$}%
}}}}
\put(8371,-4066){\makebox(0,0)[lb]{\smash{{\SetFigFont{17}{20.4}{\familydefault}{\mddefault}{\updefault}{\color[rgb]{0,0,0}$\partial B_1$}%
}}}}
\put(6751,-4066){\makebox(0,0)[lb]{\smash{{\SetFigFont{17}{20.4}{\familydefault}{\mddefault}{\updefault}{\color[rgb]{0,0,0}$\partial B_1$}%
}}}}
\end{picture}%

%% file: rholay.pdf_t
\begin{picture}(0,0)%
\includegraphics{rholay.pdf}%
\end{picture}%
\setlength{\unitlength}{4144sp}%
\begingroup\makeatletter\ifx\SetFigFont\undefined%
\gdef\SetFigFont#1#2#3#4#5{%
  \reset@font\fontsize{#1}{#2pt}%
  \fontfamily{#3}\fontseries{#4}\fontshape{#5}%
  \selectfont}%
\fi\endgroup%
\begin{picture}(11955,4482)(1921,-4846)
\put(2296,-2446){\makebox(0,0)[lb]{\smash{{\SetFigFont{17}{20.4}{\familydefault}{\mddefault}{\updefault}{\color[rgb]{0,0,0}$\frac{1}{2}$}%
}}}}
\put(1936,-1006){\makebox(0,0)[lb]{\smash{{\SetFigFont{17}{20.4}{\familydefault}{\mddefault}{\updefault}{\color[rgb]{0,0,0}$1-\tau$}%
}}}}
\put(8596,-2986){\makebox(0,0)[lb]{\smash{{\SetFigFont{17}{20.4}{\familydefault}{\mddefault}{\updefault}{\color[rgb]{0,0,0}$\overline{\rho}^{\textrm{lay}}_+$}%
}}}}
\put(11746,-2131){\makebox(0,0)[lb]{\smash{{\SetFigFont{17}{20.4}{\familydefault}{\mddefault}{\updefault}{\color[rgb]{0,0,0}$\frac{1}{2}$}%
}}}}
\put(11971,-4426){\makebox(0,0)[lb]{\smash{{\SetFigFont{17}{20.4}{\familydefault}{\mddefault}{\updefault}{\color[rgb]{0,0,0}$\partial B_1$}%
}}}}
\put(13861,-916){\makebox(0,0)[lb]{\smash{{\SetFigFont{17}{20.4}{\familydefault}{\mddefault}{\updefault}{\color[rgb]{0,0,0}$f_+^\textrm{lay}$}%
}}}}
\put(12781,-826){\makebox(0,0)[lb]{\smash{{\SetFigFont{17}{20.4}{\familydefault}{\mddefault}{\updefault}{\color[rgb]{0,0,0}$1-\tau$}%
}}}}
\put(12736,-4741){\makebox(0,0)[lb]{\smash{{\SetFigFont{17}{20.4}{\familydefault}{\mddefault}{\updefault}{\color[rgb]{0,0,0}$\rho^{\textrm{bdr}, \textrm{lay}}$}%
}}}}
\put(2566,-4426){\makebox(0,0)[lb]{\smash{{\SetFigFont{17}{20.4}{\familydefault}{\mddefault}{\updefault}{\color[rgb]{0,0,0}$\partial B_1$}%
}}}}
\put(5401,-1726){\makebox(0,0)[lb]{\smash{{\SetFigFont{17}{20.4}{\familydefault}{\mddefault}{\updefault}{\color[rgb]{0,0,0}$0$}%
}}}}
\put(5356,-781){\makebox(0,0)[lb]{\smash{{\SetFigFont{17}{20.4}{\familydefault}{\mddefault}{\updefault}{\color[rgb]{0,0,0}$0$}%
}}}}
\put(5221,-3166){\makebox(0,0)[lb]{\smash{{\SetFigFont{17}{20.4}{\familydefault}{\mddefault}{\updefault}{\color[rgb]{0,0,0}$\rho^\textrm{lay}$}%
}}}}
\put(11656,-3256){\makebox(0,0)[lb]{\smash{{\SetFigFont{17}{20.4}{\familydefault}{\mddefault}{\updefault}{\color[rgb]{0,0,0}$\overline{\rho}^{\textrm{lay}}_+$}%
}}}}
\put(5266,-4471){\makebox(0,0)[lb]{\smash{{\SetFigFont{17}{20.4}{\familydefault}{\mddefault}{\updefault}{\color[rgb]{0,0,0}$\frac{m_+^{\textrm{lay}}}{|B_1|}$}%
}}}}
\put(8281,-4426){\makebox(0,0)[lb]{\smash{{\SetFigFont{17}{20.4}{\familydefault}{\mddefault}{\updefault}{\color[rgb]{0,0,0}$\partial B_1$}%
}}}}
\put(13636,-4426){\makebox(0,0)[lb]{\smash{{\SetFigFont{17}{20.4}{\familydefault}{\mddefault}{\updefault}{\color[rgb]{0,0,0}$\partial B_1$}%
}}}}
\end{picture}%

%% file: rholay2.pdf_t
\begin{picture}(0,0)%
\includegraphics{rholay2.pdf}%
\end{picture}%
\setlength{\unitlength}{4144sp}%
\begingroup\makeatletter\ifx\SetFigFont\undefined%
\gdef\SetFigFont#1#2#3#4#5{%
  \reset@font\fontsize{#1}{#2pt}%
  \fontfamily{#3}\fontseries{#4}\fontshape{#5}%
  \selectfont}%
\fi\endgroup%
\begin{picture}(9567,4810)(3226,-5282)
\put(9001,-691){\makebox(0,0)[lb]{\smash{{\SetFigFont{17}{20.4}{\familydefault}{\mddefault}{\updefault}{\color[rgb]{0,0,0}$\mu'$}%
}}}}
\put(5221,-1591){\makebox(0,0)[lb]{\smash{{\SetFigFont{17}{20.4}{\familydefault}{\mddefault}{\updefault}{\color[rgb]{0,0,0}$\rho_{-}^\textrm{lay}$}%
}}}}
\put(3601,-1051){\makebox(0,0)[lb]{\smash{{\SetFigFont{17}{20.4}{\familydefault}{\mddefault}{\updefault}{\color[rgb]{0,0,0}$1$}%
}}}}
\put(8191,-5191){\makebox(0,0)[lb]{\smash{{\SetFigFont{17}{20.4}{\familydefault}{\mddefault}{\updefault}{\color[rgb]{0,0,0}$\Lambda$}%
}}}}
\put(5536,-781){\makebox(0,0)[lb]{\smash{{\SetFigFont{17}{20.4}{\familydefault}{\mddefault}{\updefault}{\color[rgb]{0,0,0}$\mu_-$}%
}}}}
\put(4411,-2851){\makebox(0,0)[lb]{\smash{{\SetFigFont{17}{20.4}{\familydefault}{\mddefault}{\updefault}{\color[rgb]{0,0,0}$f_-$}%
}}}}
\put(3241,-4651){\makebox(0,0)[lb]{\smash{{\SetFigFont{17}{20.4}{\familydefault}{\mddefault}{\updefault}{\color[rgb]{0,0,0}$1-\tau$}%
}}}}
\put(4681,-4966){\makebox(0,0)[lb]{\smash{{\SetFigFont{17}{20.4}{\familydefault}{\mddefault}{\updefault}{\color[rgb]{0,0,0}$\partial B_1$}%
}}}}
\put(8056,-3121){\makebox(0,0)[lb]{\smash{{\SetFigFont{17}{20.4}{\familydefault}{\mddefault}{\updefault}{\color[rgb]{0,0,0}$\rho_{\textrm{in}}^{\textrm{lay}}$}%
}}}}
\put(11521,-4966){\makebox(0,0)[lb]{\smash{{\SetFigFont{17}{20.4}{\familydefault}{\mddefault}{\updefault}{\color[rgb]{0,0,0}$\partial B_1$}%
}}}}
\end{picture}%